\documentclass[a4paper,11pt]{article}
\usepackage{pdfsync}

\usepackage[plainpages=false]{hyperref}
\usepackage{amsfonts,amsmath,amssymb,amsthm,enumerate}
\usepackage{latexsym,lscape,rawfonts,mathrsfs,mathtools}
\usepackage{enumitem,pifont}
\usepackage{cases}

\usepackage[all]{xy}
\usepackage{eufrak}
\usepackage{makeidx}
\usepackage{graphicx,psfrag}
\usepackage{pstool}
\usepackage{tikz-cd}
\usepackage{float}

\usepackage{array,tabularx,tabu}

\usepackage{setspace}

\usepackage[titletoc,title]{appendix}

\usepackage{txfonts}

\newcommand{\R}{\ensuremath{\mathbb{R}}}

\newcommand{\CP}{\ensuremath{\mathbb{CP}}}

\newcommand{\ba}{\begin{align*}}
\newcommand{\ea}{\end{align*}}
\newcommand{\na}{\nabla}
\newcommand{\la}{\langle}
\newcommand{\ra}{\rangle}
\newcommand{\lc}{\left(}
\newcommand{\rc}{\right)}

\newcommand{\ep}{\epsilon}

\newcommand{\Rm}{\overset{\circ }{Rm_{\Sigma}}}
\newcommand{\Rc}{\overset{\circ }{Rc_{\Sigma}}}
\newcommand{\ds}{\Delta_{\Sigma}}

\makeatletter
\newcommand*\owedge{\mathpalette\@owedge\relax}
\newcommand*\@owedge[1]{%
\mathbin{%
\ooalign{%
$#1\m@th\bigcirc$\cr
\hidewidth$#1\m@th\wedge$\hidewidth\cr
}%
}%
}
\makeatother

\makeatletter
\def\ExtendSymbol#1#2#3#4#5{\ext@arrow 0099{\arrowfill@#1#2#3}{#4}{#5}}

\makeatother

\makeatletter
\def\ExtendSymbol#1#2#3#4#5{\ext@arrow 0099{\arrowfill@#1#2#3}{#4}{#5}}
\newcommand\longright[2][]{\ExtendSymbol{-}{-}{\rightarrow}{#1}{#2}}
\makeatother

\def\Xint#1{\mathchoice
{\XXint\displaystyle\textstyle{#1}}%
{\XXint\textstyle\scriptstyle{#1}}%
{\XXint\scriptstyle\scriptscriptstyle{#1}}%
{\XXint\scriptscriptstyle\scriptscriptstyle{#1}}%
\!\int}
\def\XXint#1#2#3{{\setbox0=\hbox{$#1{#2#3}{\int}$ }
\vcenter{\hbox{$#2#3$ }}\kern-.55\wd0}}

\def\aint{\Xint-}

\numberwithin{equation}{section}

\newtheorem{thm}{Theorem}[section]
\newtheorem{cor}[thm]{Corollary}
\newtheorem{prop}[thm]{Proposition}
\newtheorem{lem}[thm]{Lemma}
\newtheorem{conj}[thm]{Conjecture}

\newtheorem{rem}[thm]{Remark}

\newtheorem{defn}[thm]{Definition}

\newtheorem{exmp}[thm]{Example}

\newtheorem{assum}{Assumption}

\title{Rigidity of the round cylinders in Ricci shrinkers}
\author{Yu Li \quad and \quad Bing Wang}
\date{\today}

\begin{document}
\maketitle

\begin{abstract}
In this paper, we prove that the round cylinders are rigid in the space of Ricci shrinkers. Namely, any Ricci shrinker that is sufficiently close to $S^{n-1}\times \R$ in the pointed-Gromov-Hausdorff topology must itself be isometric to $S^{n-1}\times \R$.
\end{abstract}

\tableofcontents

\section{Introduction}

A Ricci shrinker is a triple $(M^n, g, f)$ of smooth manifold $M^n$, Riemannian metric $g$ and a smooth function $f$ satisfying 
\begin{align*} 
Rc+\text{Hess}\,f=\frac{1}{2}g,
\end{align*}
where the potential function $f$ is normalized so that
\begin{align} 
R+|\nabla f|^2&=f \label{E101}.
\end{align}
The Ricci shrinkers play essential roles in studying the singularities of the Ricci flow. For example,  it was proved by Enders-M\"uller-Topping \cite{EMT11} that any proper blowup sequence from a type-I Ricci flow converges smoothly to a nontrivial Ricci shrinker.

Due to its importance, the classification of Ricci shrinkers has attracted extensive attentions. 
In dimension $2$ or $3$,  we know that $\R^2,S^2,\R^3,S^3,S^2\times \R$, and their quotients make the complete list of all Ricci shrinkers (e.g.,~\cite{Ha95}\cite{Naber}\cite{NW}\cite{CCZ}). 
In higher dimensions,  there exist many nontrivial, non-product Ricci shrinkers (e.g.,~\cite{Ko90}\cite{Cao92}\cite{FIK}), and the classification of Ricci shrinkers is only achieved when extra assumptions are assumed.
Such assumptions include non-negativity of curvatures (e.g.,~\cite{MW17}\cite{LNW18}\cite{Na19}\cite{LN20}\cite{LW20}), restriction of the Weyl curvatures (e.g.,~\cite{PW10}\cite{MS13}\cite{CC13}\cite{CW15}), 
restriction of asymptotic behavior at infinity (e.g.,~\cite{KW15}\cite{KW17}), K\"ahler conditions (e.g.,~\cite{Ni05}\cite{CDS19}\cite{Cifarelli}), and others. 
In general, much less is known if no extra assumptions are assumed.

We can consider all Ricci shrinkers as one moduli space $\mathcal M$ equipped with pointed-Gromov-Hausdorff topology.  On each Ricci shrinker,  $f$ always achieves its minimum value (cf. \cite{CZ10}) at some point $p$, which can be assigned as a base point. 
A natural question is: which Ricci shrinker is rigid in $\mathcal M$, in the sense that there is no nearby (in the pointed Gromov-Hausdorff distance) Ricci shrinker other than itself? For instance, it follows from \cite{Hui85} and the weak-compactness theory developed in \cite{LLW21} that the spherical space form $S^n/\Gamma$ is rigid since any Ricci shrinker is a self-similar Ricci flow solution. Other rigid examples include $\CP^{2n}$ proved by Kr\"oncke \cite{Kr16} and $S^2 \times S^2$ proved by Sun-Zhu \cite{SZ21} very recently. 
Their proofs depend on a delicate local structure theory of $\mathcal{M}$ (cf.  \cite{PS15}\cite{Ko82}),  which is not available for non-compact Ricci shrinkers. 
This makes the rigidity problem for non-compact Ricci shrinkers much more involved. 
Up to now,  the only non-compact Ricci shrinkers known to be rigid are the Gaussian solitons $(\R^n,g_E)$, whose rigidity is proved through an entropy-gap argument (cf. \cite{Yo09}\cite{Yo12}\cite{LW20}).

Inspired by the fundamental work of Colding-Ilmanen-Minicozzi~\cite{CIM15} on mean curvature flow, and our earlier research on 4-dimensional Ricci shrinkers~\cite{LWs1},  we are interested in figuring out whether the generalized cylinders $S^k \times \R^{n-k}, 2 \leq k \leq n-1$
are rigid in the moduli space of Ricci shrinkers.    In this article, we confirm the rigidity of the round cylinders $S^{n-1} \times \R$, i.e., the cases $k=n-1$. 

\begin{thm}[Main Theorem]
\label{T100}
There exists a small constant $\hat \ep=\hat \ep(n)>0$ satisfying the following property.

Suppose $(M^n,p,g,f)$ is a Ricci shrinker such that
\begin{align}
d_{PGH} \left\{ (M^n,p,g), (S^{n-1}\times \R,p_c,g_c)\right\}<\hat \ep, \label{E100}
\end{align} 
then $(M,g)$ is isometric to $(S^{n-1}\times \R,g_c)$. Here $d_{PGH}$ means the pointed-Gromov-Hausdorff distance, $p$ is a minimum point of $f$ and $p_c$ is a fixed point of $ S^{n-1}\times \R$.
\end{thm}

The proof of Theorem~\ref{T100} relies heavily on the symmetry improvement technique of S. Brendle et al. (cf. \cite{Bre14}\cite{Bre20}\cite{BDNS21}\cite{BDS20}\cite{BK20}), the classification result of B. Kotschwar~\cite{Kot08} and the weak-compactness theory developed by H. Li, S. Huang and the authors (\cite{LLW21} \cite{HLW21} \cite{LW20}). 

\begin{proof}[Outline of the proof of Theorem~\ref{T100}]
The proof consists of four steps. 

\textit{Step 1. The base point $p$ is in an $\epsilon$-neck.}

Condition \eqref{E100} means that the base point $p$ has a neighborhood very close to the standard cylinder in the Gromov-Hausdorff topology, which is very rough.
However, by the weak-compactness theory developed in~\cite{LLW21}, we can improve the Gromov-Hausdorff topology to the $C^{\infty}$-Cheeger-Gromov topology (cf. Proposition 7.4 of~\cite{LLW21}). 
Therefore, by choosing $\hat{\epsilon}$ sufficiently small, we are able to show (cf.~Proposition \ref{P210b}) that $p$ is actually in the center of an evolving normalized $\ep$-neck (cf. Definition~\ref{def:neck_2}). 
We may understand that the point $p$ has $\epsilon$-symmetry with respect to the model space $S^{n-1} \times \R$. 

\textit{Step 2. Each point $x \in M$ is in an $\ep$-neck or a region $\epsilon$-close to the Bryant soliton, after proper rescaling. }

This step is the technical core of this paper. We characterize the $\epsilon$-symmetry in terms of curvature and potential function estimates and show that these estimates are almost preserved along the flow line of $\nabla f$. 
Roughly speaking, if the level set of $f$ is uniformly away from being critical (i.e., $\nabla f=0$), then each point locates in an $\epsilon$-neck. Moreover, the critical points are isolated, and the region is very close to the Bryant soliton near the critical point of $f$. Therefore, each point ``far away" from the critical set is in an $\ep$-neck, and each point ``near" critical set is in a region $\epsilon$-close to the Bryant soliton. These descriptions can be made precise, which is the key new ingredient of this argument. 

\textit{Step 3. $(M,g)$ is rotationally symmetric.}

This step utilizes the celebrated symmetry improvement technique developed by S. Brendle et al. (cf. Appendix~\ref{app:B} for details). Suppose $\epsilon$ is sufficiently small and every point on $(M, g)$ has $\epsilon$-symmetry (with respect to either the cylinder or Bryant soliton), then every point has $\frac{\epsilon}{2}$-symmetry. 
Consequently, this process will run forever and yield that 
every point $x \in M$ has $2^{-k} \epsilon$-symmetry for each positive integer $k$. Therefore, $(M,g)$ is rotationally symmetric.

\textit{Step 4. $(M,g)$ is isometric to $S^{n-1} \times \R$.}

Since $(M,g)$ is rotationally symmetric already, the classification result of Kotschwar~\cite{Kot08} applies. 
We know that $(M, g)$ is isometric to $S^n$, $S^{n-1} \times \R$  or $\R^{n}$.
By condition (\ref{E100}), we conclude that $(M, g)$ must be isometric to $S^{n-1} \times \R$. 
\end{proof}

From the discussion above, it is clear that the most difficult part of the proof is to analyze precisely the propagation of the $\epsilon$-symmetry along the flow line of $\nabla f$. 
Due to this propagation, if there is a neck region far away from the base point and close to the round cylinder, then the behavior of the far-end can be described precisely. 
Since $f \sim \frac{d^2(p, \cdot)}{4}$, the value of $f$ indicates the distance to the base point. 
For each pair $t \leq s$, we define $\Sigma(t, s)$ to be a connected component of the level set 
\begin{align}
\{x\in M \, \mid t \le f(x) \le s\}. \label{eqn:OF05_2}
\end{align}
In the particular case $t=s$, we define
\begin{align}
\Sigma(t) \coloneqq \Sigma(t, t)=\{x\in M \, \mid f(x) = t\}. \label{eqn:OG05_1}
\end{align}
If a point locates in a cylinder-like neighborhood, we say it is the center of $\epsilon$-neck. 
Alternatively, if a point locates in a neighborhood close to a steady soliton with mild singularities, we say it is the center of $\epsilon$-steady soliton conifold.
Note that the conifold  here is the one introduced by Chen-Wang (cf. Definition 1.2 of~\cite{CW17}) to denote the Riemannian space with mild singularities.
For the exact definitions of the aforementioned concepts, see Definition \ref{def:neck_2}, Definition \ref{def:soli} and Definition~\ref{dfn:B001}.  
With these terminologies, we can precisely describe the propagation of the $\epsilon$-symmetry in the following theorem.

\begin{thm}[Propagation of almost symmetry]
\label{T101}
For any positive constants $n,A,B,\ep$ and $\delta_0\in (0,1)$, there exist positive constants $\sigma=\sigma(n,\delta_0,A,B)$, $L=L(n,\delta_0,A,B)$ and $\eta=\eta(n,\delta_0,A,B,\ep)$ satisfying the following property.

Let $(M^n,g,f)$ be a Ricci shrinker with
\begin{align}
\begin{cases}
|\na^i Rm| \le BR^{\frac{i}{2}+1}, \quad \forall \, 0\le i \le 4 \quad &\text{on}\quad \Sigma(t_0,s_0); \label{E001} \\
R\le \ep_1 f \quad &\text{on} \quad \Sigma(t_0,s_0); \\
|\Rm| \le \ep_1' R \quad &\text{on}\quad \Sigma(t_0); \\
\left |R-\frac{n-1}{2} \right| \le \ep'_1 \quad &\text{on}\quad \Sigma(t_0); \\
(1-\delta_0)s_0\ge t_0 \ge n \ep_1^{-1}; \\
\boldsymbol{\mu}(g) \ge -A;\\
\Sigma(t_0) \text{ is diffeomorphic to } S^{n-1}.
\end{cases}
\end{align}
If $\ep_1 \le \sigma$ and $\ep_1' \le \eta$, then one of the following statements holds.
\begin{itemize}
\item[(a).] There exists an end $E$ with $\partial E=\Sigma(t_0)$ such that any point in $\Sigma(\eta^{-1} t_0,\infty) \subset E$ is the center of an evolving normalized $\ep$-neck. Moreover, $E$ is asymptotic to the round cylinder with rate $O(r^{-\tau(n)})$. 
\item[(b).] There exists an end $E$ with $\partial E=\Sigma(t_0)$ such that any point in $\Sigma(\eta^{-1} t_0,\infty) \subset E$ is the center of an evolving $\ep$-neck. Moreover, $E$ is asymptotic to a regular cone with cross section diffeomorphic to $S^{n-1}$. 
\item[(c).] There exists a compact set $E$ with $\partial E=\Sigma(t_0)$, a number $s\ge s_0$ and a point $q \in \Sigma(s)$ such that $R(q)=\sigma s$. Moreover, any point in $\Sigma(\eta^{-1} t_0,s) \subset E$ is the center of an evolving $\ep$-neck and any point in the cap $D:=E \backslash \Sigma(t_0,s)$ is the center of an $\ep$-steady soliton conifold. Furthermore, 
\begin{align}
\emph{diam}_g D \le \frac{L}{\sqrt{s}}, \quad \sup_D|f-s| \le L \quad \text{and} \quad L^{-1}s \le \inf_D R \le \sup_D R \le s+L.\label{E002}
\end{align} 
\end{itemize}
\end{thm}

Note that (\ref{E001}) on $\Sigma(t_0,s_0)$ is satisfied if the region is very close to the cylinder, where the constant $\epsilon_1'$ can be chosen as zero. Since the function $f$ contains no critical point on the region $\Sigma(t_0,s_0)$, one can follow Munteanu-Wang \cite{MW19} to define the tangential curvature operator $Rm_{\Sigma}$ restricted on each level set of $f$ and the term $\Rm$ is the ``traceless" part of $Rm_{\Sigma}$ (cf. (\ref{eqn:OH04_3}), (\ref{eqn:OH04_4})). If $R/f$ is small, the value of $|\Rm|$ measures how close $Rm_{\Sigma}$ resembles the curvature operator of the standard $S^{n-1}$. Theorem \ref{T101} implies that if a Ricci shrinker has a neck region sufficiently close to the cylinder in some weak sense, one can move forward along the direction of $\na f$ to determine the future geometric property. 
If we regard a Ricci shrinker as a Ricci flow solution, the control of the geometric behavior at spatial infinity amounts to that near the singular time $t=1$. 
Recall that the traditional pseudolocality on Ricci shrinkers (cf. ~\cite[Corollary $10.6$]{LW20}) means that the almost Euclidean property is preserved by the flow until time $t=1$.
Theorem~\ref{T101} claims instead that the $\epsilon$-symmetric property is almost preserved by the flow until time $t=1$. Therefore, Theorem~\ref{T101} can be understood as a generalized form of pseudolocality theorem on Ricci shrinkers. 
See Figure~\ref{fig:ends} for rough intuition. 
We emphasize here that the global bound of the curvature is not needed. With some extra assumptions like the boundedness of the curvature, Munteanu-Wang \cite[Theorem $1.6$]{MW19} has already obtained case (a) in Theorem~\ref{T101}.

\begin{figure}[H]
\begin{center}
\psfrag{A}[c][c]{(a). Cylindrical end}
\psfrag{B}[c][c]{(b). Conical end}
\psfrag{C}[c][c]{(c). Cap end}
\includegraphics[width=0.5 \columnwidth]{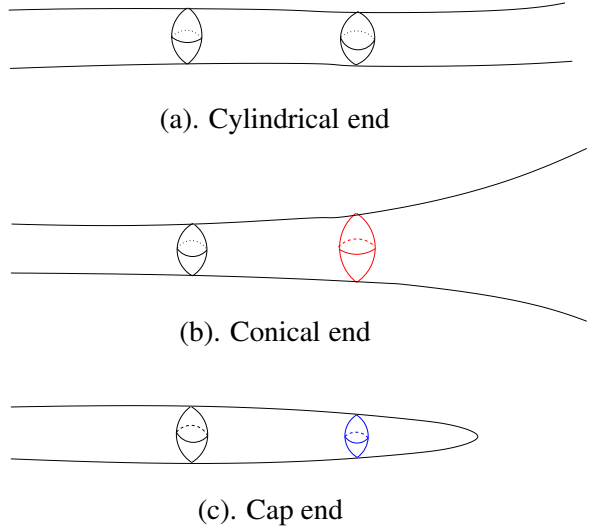}
\caption{Three types of end}
\label{fig:ends}
\end{center}
\end{figure}

We briefly discuss the proof of Theorem~\ref{T101}. We first show on the neck region satisfying \eqref{E001}, the estimate $ |\Rm| \le \ep_0 R$ holds on $\Sigma(t_0,s_0)$, where $\ep_0$ is a constant depending only on $\ep_1$ and $\ep'_1$. Since $f \ge t_0$ is large and $R/f \le \ep_1$ is small, 
all normal directional part of $Rm$ is small and insignificant. Moreover, we can show that the scalar curvature $R$ is almost constant along any level set $\Sigma$ of $f$ contained in the neck region. From the elliptic equation of $R$, one obtains an ODE by examining the behavior of $R$ along the gradient flow of $f$.

One can extend the neck region from $\Sigma(t_0,s_0)$ to $\Sigma(t_0,s)$ with $s$ to be maximal so that \eqref{E001} still holds. Then we show that $s$ is finite only if there exists a point $q \in \Sigma(s)$ with $R(q)=\ep_1 f(q)$. 
Here, we need to use the pseudolocality for Ricci shrinkers proved in \cite{LW20} to show that any point $x \in \Sigma(\eta^{-1} t_0,s)$ is the center of an evolving $\ep$-neck.
Depending on the behavior of $R$ along the gradient flow of $f$, only three cases can happen, as described in Theorem~\ref{T101}. 
If $R$ stays close to $\frac{n-1}{2}$, then $s=\infty$ and we obtain the case (a). If at some point in the neck $R$ is smaller than $\frac{n-1}{2}$ by some detectable small number, then $R$ will be quadratically decaying along the flow. In this case, $s=\infty$ and we obtain the case (b). If at some point in the neck $R$ is bigger than $\frac{n-1}{2}$ by some detectable small number, then $R$ will be increasing along the flow until it reaches the boundary $\Sigma(s)$. If this happens, $s <\infty$ and we obtain case (c).
In this case, it is important to show that there exists a compact cap $D$ with $\partial D=\Sigma(s)$. The reason is that  the behavior of the Ricci shrinker near the boundary $\Sigma(s)$ is modeled on a Ricci steady soliton conifold (cf. Definition~\ref{dfn:B001}) at a proper scale.
This can be proved through  a compactness argument (cf.~\cite{LLW21}\cite{HLW21}).   Since each model steady soliton conifold has only one end (cf. Theorem~\ref{thm:oneend}), the cap region $D$ is compact.  Further analysis implies the estimates in \eqref{E002}.

Now we continue to discuss Step 2 in the outline proof of Theorem \ref{T100} in more detail. 
First, one notices that any Ricci shrinker close to the cylinder satisfies the conditions of Theorem \ref{T101} for two neck regions. By extending those two neck regions, there are three possibilities.

\begin{itemize}
\item[(1).] The Ricci shrinker has two ends so that each point is the center of an evolving $\ep$-neck. 
\item[(2).] The Ricci shrinker has exactly one end so that outside a compact cap $D$, any point is the center of an evolving $\ep$-neck. 
\item[(3).] The Ricci shrinker is compact and it consists of two caps $D_1$ and $D_2$ such that outside those two caps, any point is the center of an evolving $\ep$-neck.
\end{itemize}

\noindent
In order to realize Step 2 in the proof of Theorem~\ref{T100}, the last key ingredient needed is the observation that the cap is modeled on the Bryant soliton in cases (2), (3) above. This fact follows from the classification result (cf.~Theorem \ref{thm:Bryant}) of any asymptotically cylindrical Ricci steady soliton conifold with the curvature condition PIC2, 
which is a generalization of the celebrated work of S. Brendle~\cite{Bre14}. On the other hand, if one considers the Ricci flow associated with a Ricci shrinker close to the cylinder, we show that locally the Ricci flow almost preserves the nonnegativity of some curvature conditions (cf. Theorem \ref{thm:T601}) defined by S. Brendle.
By analyzing the geometry along the gradient flow of $f$, one can show that the blowup limit near the cap region must have the PIC2 condition.
Consequently, the blowup limit must be a steady soliton conifold with PIC2 condition, which can be classified and can only be the Bryant soliton. 
Therefore, the cap regions are modeled on Bryant solitons. 
In conclusion, for a Ricci shrinker sufficiently close to a cylinder, each point lies in a region either $\epsilon$-close to the cylinder or the Bryant soliton. 
Roughly speaking, this means that the shrinker itself is $\ep$-symmetric (cf. Definition \ref{neck_symmetry}, Definition \ref{def:sym1} or Definition \ref{def:sym2} for precise definitions).
Therefore, we have collected sufficient technical preparation to apply the symmetry improvement argument of S. Brendle et al. \\

This paper is organized as follows. In Section 2, we discuss some basic properties for Ricci shrinkers and the associated Ricci flows. In Section 3, we obtain all the necessary technical estimates on the neck region. 
Section 4 focuses on the cap region and shows that it can be modeled on a Ricci steady soliton conifold. As a conclusion of the estimates in Section 3 and Section 4, we prove Theorem~\ref{T101} at the end of Section 4. 
In Section 5, we prove Theorem \ref{T100}. 
In the last section, we discuss possible generalizations of our theorems. \\

{\bf Acknowledgements}: 

Yu Li is supported by YSBR-001 and a research fund from USTC (University of Science and Technology of China). 
Bing Wang is supported by NSFC-11971452, NSFC-12026251, YSBR-001 and a research fund from USTC.

\section{Preliminaries}

For any Ricci shrinker $(M^n,g,f)$, the scalar curvature $R \ge 0$ by \cite[Corollary $2.5$]{CBL07}. Moreover, it follows from the strong maximum principle that $R>0$ unless $(M,g)$ is the Gaussian soliton.

We recall the following fundamental estimate of the potential function $f$.

\begin{lem}[\cite{CZ10} \cite{HM11}]
\label{L201}
Let $(M^n,g,f)$ be a Ricci shrinker. Then there exists a point $p \in M$ where $f$ attains its infimum and $f$ satisfies the quadratic growth estimate
\begin{align*}
\frac{1}{4}\left(d(x,p)-5n \right)^2_+ \le f(x) \le \frac{1}{4} \left(d(x,p)+\sqrt{2n} \right)^2
\end{align*}
for all $x\in M$, where $a_+ :=\max\{0,a\}$.
\end{lem}

For any Ricci shrinker $(M^n,g,f)$ with the normalization \eqref{E101}, the entropy is defined as
\begin{align*} 
\boldsymbol{\mu}=\boldsymbol{\mu}(g)\coloneqq \log \int\frac{e^{-f}}{(4\pi)^{n/2}}\, dV.
\end{align*}
We remark that $e^{\boldsymbol{\mu}}$ is comparable to the volume of the unit ball $B(p,1)$, see \cite[Lemma 2.5]{LLW21}.

Next we recall the following identities and elliptic equations on Ricci shrinkers (see \cite{PW10} \cite{MW19}):
\begin{align}
R_{ik}f_k=&\frac{1}{2}\na_iR=\na_k R_{ki}, \label{E202a} \\
R_{ijkl}f_l=&\nabla_j R_{ik}-\na_i R_{jk}=\na_l R_{ijkl}, \label{E202b}\\
\Delta_f R =&R-2|Rc|^2, \label{E202c}\\
\Delta_f R_{ij}=&R_{ij}-2R_{ikjl}R_{kl}.\label{E202d}
\end{align}
Here, the weighted Laplacian $\Delta_f \coloneqq \Delta-\la \na f, \na \cdot \ra$.

We will only consider the case that $\Sigma(t)=\Sigma(t,s) \cap \{x\in M \, \mid f(x) = t\}$. Notice that $\Sigma(t,\infty)$ can be either compact or non-compact and it represents a cap or an end of the Ricci shrinker, respectively. 

For the rest of the paper, we will mainly focus on $\Sigma(t)$ when $t$ is a noncritical value of $f$. Also, we omit $t$ in $\Sigma(t)$, if there is no confusion.

On $\Sigma$, we denote the unit normal vector $|\na f|^{-1}\na f$ by $e_n$ and assume $\{e_1,e_2,\cdots,e_n\}$ is a local orthonormal frame such that $\{e_1,\cdots,e_{n-1}\}$ are tangent to $\Sigma$. By using the same convention as in \cite{MW19}, we use subscript $n$ to denote $e_n$ and $a,b,c,d,\cdots$ to denote all tangential directions.

We have the following estimates of the curvature along the normal direction. The proof follows from direct calculations based on \eqref{E202a}-\eqref{E202d}.

\begin{prop}\label{P201}
For any Ricci shrinker $(M^n,g,f)$ such that $R \le \frac{3}{4}f$ on a level set $\Sigma$ of $f$, the following estimates hold on $\Sigma$.
\begin{align*} 
|R_{an}|\le &Cf^{-\frac{1}{2}}|\na R|,\\
|R_{nn}|\le &Cf^{-1}\lc R+|Rc|^2+|\na^2 R| \rc,\\
|R_{abcn}| \le & Cf^{-\frac{1}{2}}|\na Rc|, \\
|R_{ancn}| \le & Cf^{-1}\lc |Rc|+|\na^2 Rc|+|Rc|\,|Rm| \rc, 
\end{align*}
where $C=C(n)>0$.
\end{prop}

\begin{proof}
From \eqref{E101} and our assumption on $R$, we have $|\na f| \ge \frac{1}{2} \sqrt{f}$. We compute on $\Sigma$,
\begin{align*} 
|R_{an}|=\frac{|Rc(e_a,\na f)|}{|\na f|} \le Cf^{-\frac{1}{2}}|\la \na R, e_a\ra| \le Cf^{-\frac{1}{2}}|\na R|.
\end{align*}
from \eqref{E202a}. Moreover, it follows from \eqref{E202a} and \eqref{E202c} that
\begin{align*} 
|R_{nn}|=& |\na f|^{-2}|Rc(\na f,\na f)| \\
\le& Cf^{-1}|\la \na R,\na f \ra| \\
=& Cf^{-1}\left| \Delta R-R+2|Rc|^2 \right| \\
\le& Cf^{-1}\lc R+|Rc|^2+|\na^2 R| \rc.
\end{align*}
Similarly, we compute from \eqref{E202b},
\begin{align*} 
|R_{abcn}| \le |\na f|^{-1}|R_{abcl}f_l| \le Cf^{-\frac{1}{2}}|\na Rc|.
\end{align*}
For the last inequality, we have
\begin{align*} 
|R_{ancn}| \le& |\na f|^{-2}|R_{akcl}f_kf_l| \\
\le& Cf^{-1}|(\na_k R_{ac}-\na_aR_{kc})f_k| \\
=&Cf^{-1}|\na_k R_{ac}f_k-\na_a(R_{kc}f_k)+R_{kc}f_{ka}| \\
=&Cf^{-1}|\Delta R_{ac}-R_{ac}+2R_{akcl}R_{kl}-\na^2_{ac} R/2+R_{kc}(g_{ka}/2-R_{ka})| \\
=&Cf^{-1}|2R_{akcl}R_{kl}-R_{kc}R_{ka}-R_{ac}/2+\Delta R_{ac}-\na^2_{ac} R/2| \\
\le & Cf^{-1}\lc |Rc|+|\na^2 Rc|+|Rc|\,|Rm| \rc.
\end{align*}
In sum, the proof is complete.
\end{proof}

As in \cite{MW19}, we define the following tensors on $\Sigma$:
\begin{align}
&\overset{\circ }{Rc_{\Sigma}}=\overset{\circ }{R}_{ab}:=R_{ab}-\frac{1}{n-1}R\,g_{ab}, \label{E207a} \\
&U=U_{abcd}:=\frac{1}{\left( n-1\right) \left( n-2\right) }R\left(g_{ac}g_{bd}-g_{ad}g_{bc}\right), \label{eqn:OH04_3}\\
&\overset{\circ }{Rm_{\Sigma}}=\overset{\circ }{R}_{abcd}:=R_{abcd}-U_{abcd}, \label{eqn:OH04_4} \\
&V=V_{abcd}:=\frac{1}{n-3}\left( \overset{\circ }{R}_{ac}g_{bd}+\overset{\circ }{R}_{bd}g_{ac}-\overset{\circ }{R}_{ad}g_{bc}-\overset{\circ }{R}_{bc}g_{ad}\right), \label{eqn:OH04_5} \\
&W=W_{abcd}:=R_{abcd}-U_{abcd}-V_{abcd}. \label{eqn:OH04_6}
\end{align}
Notice that all tensors $\overset{\circ }{Rc_{\Sigma}}, U, \overset{\circ }{Rm_{\Sigma}}, V$ and $W$ can be extended to tensors on $\Sigma(s,t)$ by requiring they vanish on the normal direction $e_n$. 
By this convention, all tensors can be written in a coordinate-free form. For instance, we have
\begin{align} 
\overset{\circ }{Rc_{\Sigma}}=Rc-\frac{R}{n-1}g -T \label{E207f}
\end{align}
where 
\begin{align}
T \coloneqq \frac{1}{2|\na f|^2}(dR\otimes df+df\otimes dR)-\lc \frac{R}{(n-1)|\na f|^2}+\frac{Rc(\na f,\na f)}{|\na f|^4} \rc df \otimes df. \label{eqn:OH05_2}
\end{align}

Next, we have the following lower bound estimate of the scalar curvature, which essentially follows from \cite[Theorem $1$]{CLY11}.

\begin{prop} \label{P303}
For any Ricci shrinker $(M^n,g,f)$, suppose $\bar t \ge 2n$ and
\begin{align*}
\min_{\Sigma(\bar t)} \frac{Rf}{1+nf^{-1}}=\alpha>0.
\end{align*}
Then for any $x \in \Sigma(\bar t,\infty)$,
\begin{align*}
R(x)f(x) \ge \frac{R(x)f(x)}{1+nf^{-1}(x)} \ge \alpha.
\end{align*}
\end{prop}
\begin{proof}
We define $\phi=R-\alpha f^{-1}-\alpha n f^{-2}$. Then it follows from direct computations, see \cite[Equation (6)]{CLY11} for details, that on $\Sigma(\bar t,\infty)$
\begin{align*}
\Delta_f \phi \le \phi- \alpha nf^{-3} \lc \frac{f}{2}-n \rc- \alpha f^{-4}(2f+6n) |\na f|^2 \le \phi
\end{align*}
since $f \ge \bar t \ge 2n$. By our assumption, $\phi \ge 0$ on the boundary $\Sigma(\bar t)$. If $\Sigma(\bar t,\infty)$ is compact, then we conlude that $\phi \ge 0$ on $\Sigma(\bar t,\infty)$ from the maximum principle. If $\Sigma(\bar t,\infty)$ is non-compact, the maximum principle also applies since $ \liminf_{x \to \infty} \phi(x) \ge 0$ from its definition.

In sum, we have proved that on $\Sigma(\bar t,\infty)$, 
\begin{align*}
\frac{Rf}{1+nf^{-1}} \ge \alpha
\end{align*}
and the proof is complete.
\end{proof}

\subsection*{Ricci flow associated with a Ricci shrinker}

Recall that any Ricci shrinker $(M^n,g,f)$ can be regarded as a self-similar solution of the Ricci flow. Let ${\psi^t}: M \to M$ be a family of diffeomorphisms generated by $X(t)=\dfrac{1}{1-t}\nabla f$, and $\psi^{0}=\text{id}$.
In other words, we have
\begin{align} 
\frac{\partial}{\partial t} {\psi^t}(x)=\frac{1}{1-t}\nabla f\left({\psi^t}(x)\right). \label{E211}
\end{align}
It is well known that the rescaled pull-back metric $g(t)\coloneqq (1-t) (\psi^t)^*g$ satisfies the Ricci flow equation
\begin{align*} 
\partial_t g=-2Rc(g(t))
\end{align*}
for any $-\infty <t<1$. In particular, $g(0)=g$. For any Ricci shrinker, the associated Ricci flow is implicitly understood.

\begin{exmp}[Round cylinder] \label{ex:round}
The round cylinder $(S^{n-1}\times \R,g_c)$ is a Ricci shrinker where the scalar curvature is identically $\dfrac{n-1}{2}$ and the potential function $f_c=\dfrac{z^2}{4}+\dfrac{n-1}{2}$, where $z$ is the coordinate in the $\R$-factor. The associated Ricci flow is $\{(S^{n-1} \times \R,g_c(t)), \,-\infty <t<1\}$, where
\begin{align*} 
g_c(t)=2(n-2)(1-t)g_{S^{n-1}}+dz \otimes dz, \quad g_c(0)=g_c.
\end{align*}
\end{exmp}

Next, we have the following definition, which measures how close a parabolic neighborhood of a point in a Ricci flow is to the round cylinder, see also \cite[Section 11.8]{Pe1}.

\begin{defn}[Center of an evolving $\ep$-neck]\label{def:neck_2}
Let $(M^n,g(t))$ be a Ricci flow solution and let $(\bar{x},\bar{t})$ be a point in space-time with $R(\bar{x},\bar{t}) = \frac{n-1}{2}r^{-2}$. We say that $(\bar{x},\bar{t})$ is the center of an \textbf{evolving} $\ep$-neck $S^{n-1} \times \R$ if, after rescaling the metric by the factor $r^{-2}$, the parabolic neighborhood $B_{g(\bar{t})}(\bar{x},\ep^{-1} r) \times [\bar{t}-\ep^{-1} r^2,\bar{t}]$ is $\ep$-close in $C^{[\ep^{-1}]}$-topology to $\lc S^{n-1} \times \R,g_c(t) \rc$. The evolving neck is called an \textbf{evolving normalized} $\ep$-neck, if we further require that $|r-1|<\ep$. For any Ricci shrinker $(M^n,g,f)$, we say $x \in M$ is the center of an evolving (normalized) $\ep$-neck if $(x,0)$ is the center of an evolving (normalized) $\ep$-neck. 
\end{defn}

From the weak-compactness theory developed in \cite{LLW21}, we prove the following result.

\begin{prop}\label{P210b}
For any $n$ and $\ep>0$, there exists a constant $\eta_1=\eta_1(\ep,n)>0$ satisfying the following property.

Suppose $(M^n,p,g,f)$ is a Ricci shrinker such that
\begin{align*}
d_{PGH} \left\{ (M^n,p,g), \lc S^{n-1} \times \R,p_c, g_c \rc\right\}<\eta_1,
\end{align*} 
then $p$ is the center of an evolving normalized $\ep$-neck.
\end{prop}
\begin{proof}
Suppose otherwise, there exists a number $\bar \ep$ and a sequence of Ricci shrinkers $(M^n_i,p_i,g_i,f_i)$ such that
\begin{align*}
\lim_{i \to 0} d_{PGH} \left\{ (M_i^n,p_i,g_i), \lc S^{n-1} \times \R,p_c,g_c \rc \right\}=0
\end{align*} 
and $p_i$ is not the center of an $\bar \ep$-neck. By Proposition 5.8 of~\cite{LLW21}, the entropy $\boldsymbol{\mu}(g_i)$ is uniformly bounded from below. Using Theorem 1.1 of~\cite{LLW21}, the above convergence can be improved to be in the smooth topology
\begin{align*}
(M_i, p_i, g_i) \longright{C^{\infty}-Cheeger-Gromov} \lc S^{n-1} \times \R,p_c,g_c \rc.
\end{align*}
From the definition of the Ricci flow associated with a Ricci shrinker, it is clear that the corresponding sequence of Ricci flows $(M_i,p_i,g_i(t))_{t \le 0}$ converges smoothly to $( S^{n-1} \times \R,g_c(t) )_{t \le 0}$. Therefore, we obtain a contradiction.
\end{proof}

%
%

\section{Estimates on the neck region}

Throughout this section, we consider a Ricci shrinker $(M^n,p,g,f)$ such that the following assumptions as in Theorem \ref{T101} are satisfied.

\begin{align}
\begin{cases}
|\na^i Rm| \le BR^{\frac{i}{2}+1}, \quad \forall \, 0\le i \le 4 \quad &\text{on}\quad \Sigma(t_0,s); \tag{$\dagger$} \label{E401}\\
R\le \ep_1 f \quad &\text{on} \quad \Sigma(t_0,s); \\
|\Rm| \le \ep_1' R \quad &\text{on}\quad \Sigma(t_0); \\
\left |R-\frac{n-1}{2} \right| \le \ep'_1 \quad &\text{on}\quad \Sigma(t_0);\\
(1-\delta_0)s\ge t_0 \ge n\ep_1^{-1}; \\
\boldsymbol{\mu}(g) \ge -A;\\
\Sigma(t_0) \text{ is diffeomorphic to } S^{n-1}.
\end{cases}
\end{align}
Moreover, we set
\begin{align}
\ep_0:=\max\{\ep_1',\ep_1^{\frac{1}{2}}\}. \label{eqn:OH04_2}
\end{align}
Here, we assume $0<\delta_0 <1$ and $A,B \ge 1$. With fixed $n,\delta_0,A$ and $B$, we regard $\ep_1$ and $\ep_1'$ as two small positive parameters. In addition, the constant $s$ can be finite or $\infty$. Notice that for the round cylinder $S^{n-1} \times \R$, \eqref{E401} is satisfied for any $\ep_1$ and $\ep_1'$ provided that $t_0$ is sufficiently large. In the following analysis, we set 
\begin{align}
E \coloneqq \Sigma(t_0, s) \label{eqn:OH05_1}
\end{align}
and use the notation $\phi_1=O(\phi_2)$ if $|\phi_1| \le C |\phi_2|$ for some constant $C=C(n,A,B)>0$. 

One important consequence of $\boldsymbol{\mu}(g) \ge -A$ is the following no-local-collapsing result from \cite{LW20}.

\begin{thm} \emph{(\cite[Theorem 22]{LW20})}\label{thm:nonc}
For any Ricci shrinker $(M^n,g,f)$ with $\boldsymbol{\mu}(g) \ge -A$, there exists a $\kappa=\kappa(n,A)>0$ such that for any $B(q,r) \subset M$ with $R \le r^{-2}$, we have
\begin{align*} 
|B(q,r)| \ge \kappa r^n.
\end{align*}
\end{thm}

We next prove

\begin{lem} \label{L401}
On $E$, we have
\begin{align} 
Rf &\ge t_0,\label{E401aa}\\
|R_{an}|, \,|R_{abcn}| &=O(f^{-\frac{1}{2}} R^{\frac{3}{2}}), \label{E401ab}\\
|R_{nn}|, \,|R_{ancn}| &=O(f^{-1}R(1+R)). \label{E401ac}
\end{align}
\end{lem}

\begin{proof}
The inequality \eqref{E401aa} follows immediately from Proposition~\ref{P303}, since by~\eqref{E401},
\begin{align*}
\inf_{\Sigma(t_0)} \frac{Rf}{1+nf^{-1}} \ge \frac{(\frac{n-1}{2}-\ep_1')t_0}{1+nt^{-1}_0} \ge t_0.
\end{align*}
The equations \eqref{E401ab} and \eqref{E401ac} follow from Proposition \ref{P201} and \eqref{E401}.
\end{proof}

The next lemma follows from Lemma \ref{L401} and direct calculations from the definition \eqref{E207a}.
\begin{lem}
\label{LL402}
On $E$, we have
\begin{eqnarray}
&&|Rc|^2=|\Rc|^2+\frac{R^2}{n-1}+O(f^{-1}R^2(1+R)), \notag \\
&&\la U,V \ra \,, \la V,W \ra,\, \la W, U \ra =O(f^{-1}R^2(1+R)),\notag\\
&&|U|^2=\frac{2}{(n-1)(n-2)}R^2, \notag \\
&&|V|^2=\frac{4}{n-3} |\Rc|^2+O(f^{-1}R^2(1+R)), \notag \\
&&|Rm|^2=|\Rm|^2+\frac{2}{(n-1)(n-2)}R^2+O(f^{-1}R^2(1+R)),\notag \\
&& |Rm|^2=|U|^2+|V|^2+|W|^2+O(f^{-1}R^2(1+R)). \notag
\end{eqnarray}
\end{lem}

\begin{prop} \label{P401}
Under the assumption \eqref{E401}, there exists $\sigma_1=\sigma_1(n,B)>0$ such that
\begin{align} \label{E406a}
|\Rm| =O(\ep_0 R)
\end{align}
on $E$ if $\ep_0 \le \sigma_1$. 
\end{prop}
\begin{proof}
In the proof, all constants $c_i,i=1,2,\cdots$ are positive and depend only on $n$ and $B$. We define $s_0 \in [t_0,s]$ to be the largest number such that 
\begin{align*}
\frac{|\Rm|^2}{R^2} \le \tau\ep^2_0
\end{align*}
on $\Sigma(t_0,s_0)$, where $\tau>1$ is a large constant to be determined later. By the assumption \eqref{E401}, we have 
\begin{align*}
|\Rm| \le \ep_1' R
\end{align*}
on $\Sigma(t_0)$ and hence $s_0>t_0$ as $\ep_0 \ge \ep_1'$. If $s_0=s$, then the conclusion follows. Otherwise, we assume $s_0<s$.

Following \cite{MW19}, we define
\begin{align}
G \coloneqq &\frac{|Rm|^2}{R^2}-\frac{2}{(n-1)(n-2)}=\frac{|\Rm|^2}{R^2}+O(f^{-1}(1+R)) \notag\\
=&\frac{|\Rm|^2}{R^2}+O(t_0^{-1}+\ep_1)=\frac{|\Rm|^2}{R^2}+O(\ep_0^2). \label{E407aa}
\end{align}
where the equalities follow from Lemma \ref{LL402} and \eqref{E401}. Therefore, there exists $s_1 \in (t_0,s_0]$ such that
\begin{align}
\max_{\Sigma(t_0,s_0)} G=\max_{\Sigma(s_1)} G \in [\frac{\tau}{2}\ep_0^2,2 \tau \ep_0^2]. \label{E407aab}
\end{align}
It follows from a direct calculation, see \cite[(4.10),(4.11)]{MW19} for details, that
\begin{align}
\Delta_f G \ge -2\la \na G,\na \log R \ra+4R^{-3}P, \label{E407a}
\end{align}
where
\begin{align*}
P :=-2RR_{ijkl}R_{piqk}R_{pjql}-\frac{1}{2}RR_{ijkl}R_{ijpq}R_{pqkl}+|Rm|^2|Rc|^2.
\end{align*}
By similar calculations as \cite[(4.16)]{MW19} and our assumptions, we obtain
\begin{align}
P\ge& \frac{1}{n-1}R^2 |W|^2+\frac{4}{(n-1)(n-2)(n-3)}R^2|\Rc|^2-\frac{5}{2}R|W|^3 \notag \\
& -\frac{8}{(n-3)^2}R|\Rc|^3-\frac{6}{n-3}R|W||\Rc|^2-c_1 f^{-1}R^4(1+R). \label{E407ac}
\end{align}
From Lemma \ref{LL402}, we obtain on $\Sigma(t_0,s_0)$ that
\begin{align}
|W|^2+\frac{4}{n-3}|\Rc|^2=|\Rm|^2+O(f^{-1}R^2(1+R)) \le \tau\ep_0^2R^2+O((t_0^{-1}+\ep_1) R^2) \le c_2 \ep_0^2 R^2. \label{E407ad}
\end{align}
Now, it follows from \eqref{E407ac} and \eqref{E407ad} that
\begin{align}
P \ge \lc \frac{1}{n-1}-c_3 \ep_0\rc R^2|W|^2+\lc \frac{4}{(n-1)(n-2)(n-3)}-c_3 \ep_0 \rc R^2|\Rc|^2-c_3f^{-1} R^4(1+R). \label{E407ae}
\end{align}
Therefore, if $\ep_0$ is sufficiently small, we obtain
\begin{align*}
P \ge c_4 R^2 |\Rm|^2-c_5f^{-1}R^4(1+R). 
\end{align*}
It follows from \eqref{E407a} and \eqref{E407ae} that
\begin{align}
\Delta_f G \ge -2\la \na G,\na \log R \ra+4c_4RG-4c_5f^{-1}R(1+R). \label{E407aaa}
\end{align}

If we denote the Laplacian and inner product on the level set $\Sigma$ by $\Delta_{\Sigma}$ and $\la \,,\,\ra_{\Sigma}$ respectively, then 
\begin{align*}
\Delta G=\Delta_{\Sigma} G+G_{nn}+HG_n,
\end{align*}
where $H:=\text{Tr}_{\Sigma} \frac{\na^2 f}{|\na f|}$ is the mean curvature of the level set $\Sigma$. 
From the fact that $|\na^i Rm| \le BR^{\frac{i}{2}+1}$ for $0 \le i\le 4$, the term $G_{nn}$ can be estimated as (see \cite[(4.26)]{MW19} for details)
\begin{align}
G_{nn} \le c_6 f^{-1}(1+R)|\la \na G, \na f \ra|+O(f^{-1}R(1+R)). \label{E407ag}
\end{align}
Since
\begin{align*}
H=\frac{\Delta f-f_{nn}}{|\na f|}=\frac{\frac{n-1}{2}-R+R_{nn}}{|\na f|}=\frac{\frac{n-1}{2}-R+O(f^{-1}R(1+R))}{|\na f|}, 
\end{align*}
we have
\begin{align}
HG_n=O(f^{-1}(1+R)|\la \na G,\na f \ra|). \label{E407ah}
\end{align}
Moreover, 
\begin{align}
\la \na G, \na \log R \ra=&\la \na G, \na \log R \ra_{\Sigma}+\frac{1}{|\na f|^2} \la \na G, \na f \ra \la \na \log R, \na f \ra \notag \\
=& \la \na G, \na \log R \ra_{\Sigma}+O(f^{-1}(1+R))|\la \na G, \na f \ra| \label{E407ai}
\end{align}
since $ \la \na \log R, \na f \ra=2R^{-1}Rc(\na f,\na f)=O(R^{-1}fR_{nn})=O(1+R)$.

Combining \eqref{E407aaa}, \eqref{E407ag}, \eqref{E407ah} and \eqref{E407ai}, we derive on $\Sigma(t_0,s_0)$ that
\begin{align}
\ds G \ge \la \na G, \na f\ra-c_7f^{-1}(1+R) |\la \na G,\na f \ra|-2\la \na G,\na \log R \ra_{\Sigma}+c_8 RG-c_7f^{-1}R(1+R). \label{E407b}
\end{align}
By our assumption, there exists a point $z\in \Sigma(s_1)$ such that 
\begin{align*}
G(z)=\max_{\Sigma(s_1)} G=\max_{\Sigma(t_0,s_0)} G \in [\frac{\tau}{2}\ep_0^2,2 \tau \ep_0^2]. 
\end{align*}
From \eqref{E407b} and the maximum principle, we conclude at $z$ that
\begin{align*}
0 \ge \la \na G, \na f\ra-c_7f^{-1}(1+R) |\la \na G,\na f \ra|+c_8 RG-c_7 f^{-1}R(1+R).
\end{align*}
Since $\la \na G, \na f\ra \ge 0$ at $z$, we have
\begin{align*}
G \le c_9 f^{-1}(1+R)
\end{align*}
at $z$. Therefore, it follows from \eqref{E407aab} that
\begin{align*}
\frac{\tau}{2}\ep_0^2 \le G(z) \le c_9f^{-1}(1+R) \le c_{10} \ep_0^2,
\end{align*}
where for the last inequality we have used $t_0 \ge n \ep_1^{-1}$ and hence $t_0^{-1} =O( \ep_0^2)$. However, it contradicts our assumption \eqref{E401} if $\tau$ is sufficiently large.
\end{proof}

\begin{rem}
From the proof of Proposition \ref{P401}, we have shown that on $E$
\begin{align*}
|\Rm| \le C\ep_0 R
\end{align*}
for some constant $C$ depending only on $n$ and $B$.
\end{rem}

Next, we prove

\begin{lem} \label{L402}
With the same assumptions above, we have on $E$,
\begin{align}
\begin{dcases}
&|\Rc|=O(\ep_0 R), \\
&|\na \Rc|=O(\ep_0^{\frac{1}{2}} R^{\frac{3}{2}}), \\
&|\na^2 \Rc|=O(\ep_0^{\frac{1}{3}} R^2).
\end{dcases} \label{E408a}
\end{align}
\end{lem}
\begin{proof}
In the proof, all constants $C_i,i=1,2,\cdots$ are positive and depend only on $n,A$ and $B$.

From the definition of $\Rc$, we obtain
\begin{align*}
\overset{\circ}{R}_{ac}= \sum_{k=1}^{n-1} \overset{\circ}{R}_{akck}+R_{ancn}.
\end{align*}
Therefore, it follows from \eqref{E401ac} and \eqref{E406a} that
\begin{align}
|\Rc|=O(\ep_0 R)+O(f^{-1}R(1+R))=O(\ep_0 R). \label{E408b}
\end{align}
In addition, we have
\begin{align}
\begin{dcases}
&|\na \Rc|=O(R^{\frac{3}{2}}), \\
&|\na^2 \Rc|=O(R^2), \\
& |\na^3 \Rc|=O(R^{\frac{5}{2}}). 
\end{dcases} \label{E408c}
\end{align}
Indeed, it follows from \eqref{E207f} and \eqref{E401} that
\begin{align*}
|\na \Rc|=& O(|\na Rc|)+O(f^{-1}|\na R|(1+R))+O(f^{-\frac1 2} |\na^2 R|)+O(f^{-\frac1 2} |Rc|(1+R)) \\
=&O(R^{\frac 3 2})+O(f^{-1}R^{\frac 3 2}(1+R))+O(f^{-\frac1 2} R^2)+O(f^{-\frac1 2} R(1+R))=O(R^{\frac{3}{2}}).
\end{align*}
Here, we have used \eqref{E401aa}. The other two equalities in \eqref{E408c} can be proved similarly by using $|\na^i Rm| \le BR^{\frac{i}{2}+1}$ for $0 \le i\le 4$.

For any $q \in E$, we set $R(q)=r^{-2}$ and $r_1=C_1^{-1}r$ for a large constant $C_1>1$ determined later.

\textbf{Case 1}: $B(q,2r_1) \cap \Sigma(s)=\emptyset$.

From the fact that $|\na R|^2=O(R^3)$, we may assume that $C_1$ is large enough so that
\begin{align*}
C_1^{-1}r^{-2} \le R \le C_1r^{-2}
\end{align*}
on $B(q,2r_1)$. It follows from Theorem \ref{thm:nonc} and the Bishop-Gromov volume comparison that
\begin{align*}
C_2^{-1} \le \frac{|B(q,r_1)|}{r_1^n} \le C_2
\end{align*}
for some $C_2>1$. By the standard interpolation inequality (see \cite[Theorem $7.28$]{GT01}) that for any $\ep>0$
\begin{align*}
&\lc \int_{B(q,r_1)} |\na \Rc|^2 r_1^{6-n} \,dV \rc^{\frac 1 2} \\
\le& \ep \lc \int_{B(q,r_1)} |\Rc|^2 r_1^{4-n} \,dV+\int_{B(q,r_1)} |\na \Rc|^2 r_1^{6-n} \,dV+\int_{B(q,r_1)} |\na^2 \Rc|^2 r_1^{8-n} \,dV \rc^{\frac 1 2} \\
&+C_3\ep^{-1} \lc \int_{B(q,r_1)} |\Rc|^2 r_1^{4-n} \,dV \rc^{\frac 1 2} \le C_4 \ep+C_4 \ep^{-1} \ep_0,
\end{align*}
where we have used \eqref{E408b} and \eqref{E408c}.
Plugging $\ep=\ep_0^{\frac1 2}$ into the above inequality, we obtain
\begin{align}
&\int_{B(q,r_1)} |\na \Rc|^2 \,dV \le 2C_4 \ep_0 r_1^{n-6} \label{E408d}.
\end{align}
Since $|\na^2 \Rc| \le C_5 r_1^{-4}$ by \eqref{E408c}, it follows from mean value formula and \eqref{E408d} that
\begin{align*}
|\na \Rc| \le C_6 \ep_0^{\frac 1 2}r_1^{-3}.
\end{align*}
In other words, we have proved that $|\na \Rc| =O(\ep_0^{\frac{1}{2}} R^{\frac{3}{2}})$. The estimate of $|\na^2 \Rc|$ can be proved similarly.

\textbf{Case 2}: $B(q,2r_1) \cap \Sigma(s)\ne \emptyset$.

In this case, there exists $q' \in \Sigma(s)$ which can be connected by a shortest geodesic in $E$ from $q$ to $q'$ such that its length is smaller than $2r_1$. From $|\na R|^2=O(R^3)$, $R(q')$ is almost equal to $r^{-2}$, if $C_1$ is sufficiently large. 

Next, we denote the induced metric on $\Sigma(s)$ by $g_{\Sigma}$ and define $\delta_1>0$ to be the largest number such that $R \le 2r^{-2}$ on $\Omega:=B_{g_{\Sigma}}(q',\delta_1 r)$. In addition, we define $\rho=2\sqrt f$ and $\phi_t$ to be the family of diffeomorphisms generated by $\na \rho/|\na \rho|^2$.

Now we define a map
\begin{align*}
\Phi: \Omega \times (-\delta_2 r,0] \longrightarrow E
\end{align*}
by $\Phi(z,t)=\phi_t(z)$. We claim that there exists $C_7>1$ such that $\delta_2 \ge C_7^{-1}$ and $\Phi^* g$ is close to $\bar g:=g_{\Sigma}+dt^2$ in $C^2$ sense. More precisely,
\begin{align} \label{E408dxa}
|\Phi^* g-\bar g|_{C^k(\bar g)} =O(\ep^{\frac 1 2}_1 r^{-k})
\end{align}
for $k=0,1,2$. We only prove the case $k = 0$, and the other cases are similar. We first compute for any $t \in (-\delta_2 r,0]$,
\begin{align}\label{E408dxb}
\Phi^* g(\partial_t ,\partial_t)=\frac{1}{|\na \rho|^2}=\frac{f}{|\na f|^2}=1+\frac{R}{f-R}=1+O(\ep_1).
\end{align}

Next, we set $g(t)$ to be $\Phi^* g$ restricted on $ \Omega \times \{t\}$. From our definition of $\phi_t$, we compute
\begin{align*}
\partial_t g(t)= \frac{2 \sqrt f}{|\na f|^2} \na^2 f = \frac{2 \sqrt f}{|\na f|^2} (g(t)-2Rc)=O(f^{-\frac1 2} R).
\end{align*}
By integration, we have for any $t \in (-\delta_2 r,0]$,
\begin{align} \label{E408dxc}
|g(t)-g_{\Sigma}| = O(\delta_2 r f^{-\frac1 2} R) =O(\ep^{\frac 1 2}_1),
\end{align}
where for the last equality we have used $|\na R|^2=O(R^3)$ and \eqref{E408dxb}. Combining \eqref{E408dxb} and \eqref{E408dxc}, \eqref{E408dxa} is proved.

From \eqref{E408dxa}, it follows that $g$ is almost a product metric on a neighborhood of $\Sigma(s)$. By $|\na R|^2=O(R^3)$ again, there exists $C_8>1$ such that $\delta_1 \ge C_8^{-1}$. Moreover, we can guarantee that $B(q,2r_1) \cap E \subset \Phi \lc \Omega \times (-\delta_2 r,0] \rc$.

Therefore, the rest of the proof can be done similarly to Case 1 by considering a product neighborhood of $q$ with size $r_1$ instead.
\end{proof}

Next, we prove that the scalar curvature is almost constant on any level set contained in $E$. We first prove

\begin{lem} \label{L403}
With the same assumptions above, on any level set $\Sigma$ of $f$ contained in $E$, we have
\begin{align*}
|\na_{\Sigma} R| =O( \ep_0^{\frac{1}{2}} R^{\frac{3}{2}}).
\end{align*}
\end{lem}
\begin{proof}
Recall the decomposition of $Rc$ from \eqref{E207f}. Direct calculation implies
\begin{align*}
\frac{1}{2}\na_a R=& \sum_{b=1}^{n-1} \na_b R_{ab}+\na_n R_{an} \\
=& \sum_{b=1}^{n-1} \na_b \lc \overset{\circ}{R}_{ab}+\frac{R}{n-1} g_{ab}+T_{ab}\rc +\na_n R_{an} \\
=& \sum_{b=1}^{n-1}\lc \na_b \overset{\circ}{R}_{ab}+\na_b T_{ab}\rc+\frac{\na_a R}{n-1}+\na_n R_{an}.
\end{align*}
Therefore, we obtain
\begin{align}
\frac{n-3}{2(n-1)}\na_a R=\sum_{b=1}^{n-1} \lc \na_b \overset{\circ}{R}_{ab}+\na_b T_{ab}\rc+\na_n R_{an}. \label{E409a}
\end{align}
It is clear from the definition that 
\begin{align*}
\na_b T_{ab}=\frac{f_{ab}\na_b R}{|\na f|^2} =\frac{(g_{ab}/2-R_{ab})\na_b R }{|\na f|^2}. 
\end{align*}

Since $R_{ab}=\overset{\circ}{R}_{ab}+\frac{R}{n-1}g_{ab}=\frac{R}{n-1}g_{ab}+O(\ep_0 R)$, we have
\begin{align}
\sum_{b=1}^{n-1} \na_b T_{ab}=&\frac{1}{|\na f|^2} \lc \frac{\na_a R}{2}-\frac{R\na_a R}{n-1}+O(\ep_0R^{\frac 5 2}) \rc \notag \\
=&O(f^{-1}R^{\frac3 2}(1+R))+ O(\ep_0f^{-1}R^{\frac 5 2})=O(\ep_0 R^{\frac 3 2}). \label{E409ab}
\end{align}

Moreover, we compute
\begin{align}
\na_n R_{an}=& \frac{\la \na R_{an},\na f \ra}{|\na f|}=\frac{\Delta R_{an}-R_{an}+2R_{aknl}R_{kl}}{|\na f|} \notag \\
=& O(f^{-\frac 1 2} R^2)+O(f^{-1}R^{\frac 3 2})=O(\ep_0 R^{\frac{3}{2}}). \label{E409b}
\end{align}

Combining \eqref{E409a},\eqref{E409ab}, \eqref{E409b} and Lemma \ref{L402}, we immediately obtain
\begin{align*}
|\na_a R|=O( \ep_0^{\frac{1}{2}} R^{\frac{3}{2}}).
\end{align*}
\end{proof}

\begin{lem} \label{L404}
With the same assumptions above, for any level set $\Sigma$ contained in $E$ and $x,y \in \Sigma$, we have
\begin{align*}
\left| \frac{R(x)}{R(y)}-1 \right| =O(\ep_0^{\frac{1}{2}}).
\end{align*}
Moreover, the diameter of $\Sigma$ is $O(R^{-\frac{1}{2}}(x))$.
\end{lem}

\begin{proof}
For the second fundamental form $h$ of $\Sigma$, we have
\begin{align*}
h_{ab}=\frac{f_{ab}}{|\na f|}=\frac{\frac{g_{ab}}{2}-R_{ab}}{|\na f|}=O(f^{-\frac{1}{2}}(1+R))
\end{align*}
and hence the mean curvature
\begin{align*}
H=\sum h_{aa}=O(f^{-\frac{1}{2}}(1+R)).
\end{align*}

We denote the intrinsic Ricci curvature and scalar curvature of $\Sigma$ by $\tilde R_{ab}$ and $\tilde R$, respectively. It follows from the Gauss curvature equation that
\begin{align}
\tilde R_{ab}=&R_{ab}-R_{anbn}+Hh_{ab}-h_{ac}h_{cb} \notag \\
=& R_{ab}+O(f^{-1}R(1+R))+O(f^{-1}(1+R)^2)=R_{ab}+O(\ep_0^2 R) \label{E410ba}
\end{align}
where we have used \eqref{E401aa} and \eqref{E401ac}.

Similarly,
\begin{align}
\tilde R=&R-R_{nn}+H^2-|h|^2 \notag \\
=&R+O(\ep_0^2 R). \label{E410bb}
\end{align}
Combining \eqref{E410ba}, \eqref{E410bb} and Lemma \ref{L402}, we have
\begin{align}
\tilde R_{ab}=&\frac{R}{n-1}g_{ab}+\overset{\circ}{R_{ab}}+O(\ep_0^2 R) \notag \\
=& \frac{R}{n-1}g_{ab}+O(\ep_0 R)=\frac{\tilde R}{n-1}g_{ab}+O(\ep_0 \tilde R).
\label{E410a}
\end{align}
Now we fix a point $x \in \Sigma$ and set $R(x)=r^{-2}$. By Lemma \ref{L403}, we have $|\na_{\Sigma} R^{-\frac{1}{2}}| \le C_1 \ep_0^{\frac{1}{2}}$ and hence
\begin{align}
|R^{-\frac{1}{2}}(y)-r| \le C_1L\ep_0^{\frac{1}{2}}, \label{E410b}
\end{align}
provided that $d_{\Sigma}(x,y) \le L$. From \eqref{E410bb}, \eqref{E410a} and \eqref{E410b}, the proof of the Myers theorem indicates that 
\begin{align*}
\text{diam}\, \Sigma=O(r)
\end{align*}
and hence by \eqref{E410b},
\begin{align*}
R(y)=\lc 1+O(\ep_0^{\frac{1}{2}}) \rc r^{-2}.
\end{align*}
In sum, the proof is complete.
\end{proof}

For later applications, we prove

\begin{lem} \label{L405}
With the same assumptions above, we have on $E$,
\begin{align*}
|\Delta R|= O(\ep_0^{\frac{1}{3}} R^2).
\end{align*}
\end{lem}
\begin{proof}

It follows from definition that 
\begin{align}
\Delta R= \sum_{a=1}^{n-1} \na^2_{aa} R+\na^2_{nn} R. \label{E411aa}
\end{align}
We shall estimate the different terms of the right-hand side of (\ref{E411aa}) separately. 

Firstly, it is clear that
\begin{align}
\na^2_{nn} R=& \frac{\na^2 R(\na f,\na f)}{|\na f|^2} \notag \\
=& \frac{1}{|\na f|} \na_n \la \na R,\na f\ra-\frac{1}{|\na f|^2} \na^2 f(\na R,\na f) \notag \\
=& \frac{1}{|\na f|} \na_n \lc \Delta R-R+2|Rc|^2\rc-\frac{1}{|\na f|^2} \lc \frac{\la \na R, \na f \ra}{2}-Rc(\na R,\na f) \rc \notag \\
=&O(f^{-\frac{1}{2}}R^{\frac3 2}(1+R))+O(f^{-1}R(1+R))=O(\ep_0 R^2). \label{E411a}
\end{align}

Secondly, we move on to estimate $\nabla^2_{aa} R$. It follows from \eqref{E409a} that
\begin{align}
\frac{n-3}{2(n-1)} \na_{ca}^2 R=\sum_{b=1}^{n-1} \lc \na_c\na_b \overset{\circ}{R}_{ab}-\na_c\na_b T_{ab}\rc+\na_c\na_n R_{an}. \label{E411ab}
\end{align}
In light of \eqref{E408c}, we have
\begin{align}
|\na_c\na_b \overset{\circ}{R}_{ab}|=O(\ep_0^{\frac1 3} R^2). \label{E411ac}
\end{align}
Recall that $T$ is defined in (\ref{eqn:OH05_2}). Direct calculation implies that
\begin{align}
|\na_c \na_b T_{ab}|=&O(|\na f|^{-2} |\na R| |\na^3 f |)+O(|\na f|^{-2} |\na^2 R| |\na^2 f |)+O(|\na f|^{-2} |Rc| |\na^2 f |^2) \notag \\
=& O(f^{-1} |\na R| |\na Rc |)+O(f^{-1} |\na^2 R| (1+R))+O(f^{-1} |Rc| (1+R)^2) \notag \\
=& O(f^{-1} R^3)+O(f^{-1} R^2 (1+R))+O(f^{-1} R (1+R)^2)=O(\ep_0^2 R^2).
\label{E411ad}
\end{align}
Here, we have used the fact that some components of $\na^2 T$ along the normal direction vanish.
Furthermore, we compute
\begin{align}
\na_c\na_n R_{an}=& \na_c \lc \frac{\la \na R_{an},\na f \ra}{|\na f|} \rc= \na_c \lc \frac{\Delta R_{an}-R_{an}+2R_{aknl}R_{kl}}{|\na f|} \rc \notag \\
=& O(|\na f|^{-1}(|\na^3 Rc|+|\na Rc|+|\na Rm||Rm|))+O(|\na f|^{-2}|\na^2 f|(R+R^2)) \notag\\
=& O(f^{-\frac 1 2} R^{\frac3 2}(1+R))+O(f^{-1}R(1+R)^2)=O(\ep_0 R^2).
\label{E411ae}
\end{align}
Combining \eqref{E411ab}, \eqref{E411ac}, \eqref{E411ad} and \eqref{E411ae}, we obtain
\begin{align}
\na^2_{ca}R=O(\ep_0^{\frac{1}{3}} R^2). \label{E411b}
\end{align}

Finally, it is clear that \eqref{E411aa} follows from the combination of \eqref{E411a} and \eqref{E411b}. The proof of the lemma is complete. 
\end{proof}

For later applications, we prove

\begin{lem} \label{L406}
With the same assumptions above, we have on $E$,
\begin{align*}
|\Delta R_{ab}|=O(\ep_0^{\frac{1}{3}} R^2).
\end{align*}
\end{lem}
\begin{proof}
By definition,
\begin{align}
\Delta R_{ab}=\Delta \lc \overset{\circ}{R}_{ab}+\frac{R}{n-1}g_{ab}-T_{ab}. \rc \label{E411ba}
\end{align}

It is clear from \eqref{E408a} and Lemma \ref{L405} that
\begin{align*}
\Delta \lc \overset{\circ}{R}_{ab}+\frac{R}{n-1}g_{ab} \rc=O(\ep_0^{\frac{1}{3}} R^2).
\end{align*}

On the other hand, one can compute as \eqref{E411ad} that $|\Delta T_{ab}|=O(\ep_0^2 R^2)$. Now the conclusion follows from \eqref{E411ba}.
\end{proof}

Now we consider a family of diffeomorphisms $\varphi_t$ defined by
\begin{align}
\begin{dcases}
& \frac{d \varphi_t}{dt}=\frac{\na f}{|\na f|^2}, \\
&\varphi_{t_0}=\text{id} \quad \text{on} \quad \Sigma(t_0) .
\end{dcases} \label{E412}
\end{align}

By abuse of notation, we set $R=R(t)=R(\varphi_t(x))$ for $x \in \Sigma(t_0)$. Direct calculation yields that
\begin{align}
R_t=\frac{\la \na R ,\na f \ra}{|\na f|^2}=\frac{\Delta R-R+2|Rc|^2}{t-R}.
\label{E412a}
\end{align}

On the one hand, we compute on $E$,
\begin{align}
|Rc|^2=&\sum_{a,b=1}^{n-1} R_{ab}^2+2 \sum_{a=1}^{n-1}R^2_{an}+R^2_{nn} \notag\\
=& \sum_{a,b=1}^{n-1} \lc \overset{\circ}{R}_{ab}+\frac{R}{n-1}g_{ab} \rc^2+2 \sum_{a=1}^{n-1}R^2_{an}+R^2_{nn} \notag\\
=& \frac{R^2}{n-1}+O(\ep_0R^2)+O(f^{-1}R^3)+O(f^{-2}R^2(1+R)^2) \notag\\
=& \frac{R^2}{n-1}+O(\ep_0R^2),
\label{E412aa}
\end{align}
where we have used \eqref{E401ab} and \eqref{E401ac}.

On the other hand, $\Delta R=O(\ep_0^{\frac{1}{3}} R^2)$ from Lemma \ref{L405}. Combining this with \eqref{E412aa}, we obtain from \eqref{E412a} that
\begin{align}
tR_t=\lc \frac{2}{n-1}+X \rc R^2-R
\label{E413}
\end{align}
where $X$ is a function such that
\begin{align}
X=\frac{O(\ep_0^{\frac{1}{3}})t-1+\frac{2}{n-1}R}{t-R}=O(\ep_0^{\frac{1}{3}}).
\label{E413a}
\end{align}

Now we define
\begin{align}
\ep_2 \coloneqq \ep_0^{\frac{1}{4}} \label{eqn:OH05_3}
\end{align}
and analyze different cases depending on whether $R$ stays in the interval $[\frac{n-1}{2}-\ep_2, \frac{n-1}{2}+\ep_2]$. Notice that by the assumption \eqref{E401}, on $\Sigma(t_0)$, $R$ stays in this interval.

\subsection*{Cylindrical end}

Without loss of generality, we assume that $s$ is the largest number such that $E=\Sigma(t_0,s)$ satisfies \eqref{E401}. Throughout this subsection, we assume on $E$
\begin{align}
\left| R -\frac{n-1}{2} \right| < \ep_2. \label{E416ab}
\end{align}

\begin{prop} 
\label{prop:405a}
Under the assumption \eqref{E416ab}, there exists $\sigma_2=\sigma_2(n,A,B,\delta_0)>0$ such that if $\ep_0\le \sigma_2$, then $s=\infty$.
\end{prop}

\begin{proof}
We assume $s<\infty$ and derive a contradiction. By the definition of $s$, there exists a point $q \in \Sigma(s)$ such that at $q$ either
\begin{align*}
R=\ep_1 f \quad \text{or} \quad |\na^k Rm|=BR^{\frac{k}{2}+1}
\end{align*}
for some $0 \le k \le 4$. We claim that the first case cannot happen. Indeed, by \eqref{E416ab} that
\begin{align*}
\frac{n-1}{2}+\ep_2 \ge R(q) =\ep_1f(q) \ge \ep_1 t_0 \ge n
\end{align*}
which is impossible.

To exclude the second case, we assume there exists a sequence of Ricci shrinkers $(M^n_i, g_i, f_i) $ with the end $E_i=\Sigma(t_{0,i},s_i)$ satisfying \eqref{E401} and $\ep_{0,i} \to 0$. Moreover, there exists a $q_i \in \Sigma(s_i)$ at which
\begin{align} \label{E416xx}
|\na^k Rm_i|=BR_{i}^{\frac{k}{2}+1}
\end{align}
for a fixed $0 \le k \le 4$.

Now we fix a small number $\delta<\delta_0/2$ to be determined later and define $x_i$ to be a point such that $\psi_i^{\delta}(x_i)=q_i$, where $\psi^t$ is a family of diffeomorphisms defined in \eqref{E211}. We claim that such $x_i$ must exist and lie in $E_i$. 
Indeed, from \eqref{E211} we have for any $t \ge 0$ that
\begin{align*}
\frac{df_i(\psi_i^t(x_i))}{dt}=\frac{|\na f_i|^2(\psi^t_i(x_i))}{1-t}=\frac{f_i(\psi^t_i(x_i))-R_i(\psi^t_i(x_i))}{1-t} \le \frac{f_i(\psi^t_i(x_i))}{1-t}.
\end{align*}
Therefore, 
\begin{align}\label{E416ba}
t_i:=f_i(x_i) \ge s_i(1-\delta).
\end{align}
In particular, $x_i \in E_i$ since $(1-\delta_0)s_i \ge t_{0,i}$ by our assumption \eqref{E401}. On the other hand, by \eqref{E416ab} we have for any $0 \le t \le \delta$,
\begin{align*}
\frac{df_i(\psi_i^t(x_i))}{dt}=\frac{f_i(\psi^t_i(x_i))-R_i(\psi^t_i(x_i))}{1-t} \ge \frac{f_i(\psi^t_i(x_i))-n}{1-t}.
\end{align*}
By solving the above ODE, we obtain
\begin{align}\label{E416bb}
t_i \le (s_i-n)(1-\delta)+n= s_i(1-\delta)+n\delta .
\end{align}
\textbf{Claim.} 
\begin{align*}
(M_i, x_i, g_i) \longright{C^{\infty}-Cheeger-Gromov} (S^{n-1} \times \R, p_c,g_c).
\end{align*}

\emph{Proof of the Claim:} From \eqref{E416ba} and \eqref{E416bb}, for any $L>1$ we have $B_{g_i}(x_i,L) \subset E_i$, if $i$ is sufficiently large. Indeed, since $2\sqrt{f_i}$ is $1$-Lipschitz, we have for any $y \in B_{g_i}(x_i,L)$,
\begin{align}
|2\sqrt{f_i(y)}- 2\sqrt{t_i}|=|2\sqrt{f_i(y)}- 2\sqrt{f_i(x_i)}| \le L. \label{E418ba}
\end{align}
Since $\lim_{i \to \infty} t_i=+\infty$ by \eqref{E416ba}, we obtain
\begin{align}
\lim_{i \to \infty} \frac{f_i(y)}{t_i}=1 \label{E418c}
\end{align}
uniformly for $y \in B_{g_i}(x_i,L)$. By our assumptions \eqref{E401} and \eqref{E416ab}, the Riemannian curvature $|Rm_i|$ is uniformly bounded on $E_i$. Moreover, one can obtain the higher order estimates of $Rm_i$ on $B_{g_i}(x_i,L)$. Indeed, we take $\delta_1:=\delta_0/10$ and for any $(y,t) \in B_{g_i}(x_i,L) \times [-\delta_1,0]$, we compute
\begin{align}
|Rm_i|(y,t)=\frac{1}{1-t} |Rm_i|( \psi_i^t(y),0). \label{E418ca}
\end{align}
Similar to \eqref{E416ba}, if $i$ is sufficiently large, we obtain
\begin{align*}
f_i(y) \ge f_i(\psi_i^t(y)) \ge \frac{f_i(y)}{1+\delta_1} \ge t_{0,i}. 
\end{align*}
In other words, $\psi_i^t(y) \in E_i$. Therefore, it is clear from \eqref{E418ca} that $|Rm_i|$ is uniformly bounded on $ B_{g_i}(x_i,L) \times [-\delta_1,0]$. Now the higher order estimates of $|Rm_i|$ follow from Shi's local estimates \cite{Shi89A}.

From the above estimates of the curvature and Theorem \ref{thm:nonc}, we conclude that
\begin{align*}
(M^n_i, x_i, g_i) \longright{C^{\infty}-Cheeger-Gromov} (M^n_{\infty}, x_{\infty},g_{\infty}),
\end{align*}
where the limit is a complete smooth Riemannian manifold with bounded curvature. On the other hand, if we define $\tilde f_i:= t_i^{-\frac{1}{2}} (f_i-t_i)$, then
\begin{align}
|\na_{g_i} \tilde f_i|^2_{ g_i}=\frac{|\na_{g_i} f_i|_{g_i}^2}{ t_i}=\frac{f_i-R_i}{t_i}. \label{E419ax}
\end{align}
In addition, we obtain
\begin{align}
\text{Hess}_{g_i} \tilde f_i= t_i^{-\frac{1}{2}} \text{Hess}_{g_i}f_i=\frac{1}{2} t_i^{-\frac 1 2} g_i- t_i^{-\frac{1}{2}} Rc_i. \label{E419bx}
\end{align}
From the definition of $\tilde f_i$ and \eqref{E418ba}, we compute for any $y \in B_{g_i}(x_i,L)$ that
\begin{align}
\tilde f_i(y) =t_i^{-\frac{1}{2}}(f_i(y)-t_i) =t_i^{-\frac{1}{2}} (\sqrt{f_i(y)}-\sqrt{t_i})(\sqrt{f_i(y)}+\sqrt{t_i}) \le \frac{3}{2}L. \label{E419cx}
\end{align}

Therefore, it follows from \eqref{E419ax}, \eqref{E419bx}, \eqref{E419cx} and the standard regularity theorem for elliptic equations that $\tilde f_i$ converges smoothly to a smooth function $f_{\infty}$ on $M_{\infty}$. Moreover, we conclude from \eqref{E419bx} that
\begin{align}
\text{Hess}_{g_{\infty}} f_{\infty}=0. \label{E420x}
\end{align}

In addition, $f_{\infty}$ is not a constant since by \eqref{E416ab} and \eqref{E419ax},
\begin{align*}
|\na_{g_{\infty}} f_{\infty}|^2=1.
\end{align*}

From \eqref{E420x}, we conclude that $(M_{\infty},g_{\infty})$ is isometric to $(\Sigma_{\infty} \times \R,g'_{\infty} \times g_E )$, where $\Sigma_{\infty}$ is the level set of $f_{\infty}$. Since the second fundamental form of $\Sigma_{\infty}$ vanishes, it follows from \eqref{E406a} that the sectional curvature of $g'_{\infty}$ on $\Sigma_{\infty}$ is identically $\frac{1}{2(n-2)}$. Therefore, it is clear that $(M_{\infty}, g_{\infty})$ is isometric to $(S^{n-1} \times \R, g_c)$ and Claim is proved.

Now we fix the parameter $\delta$ small enough such that the pseudolocality Theorem \cite[Theorem 24]{LW20} can be applied. Combining with \cite[Theorem $3.1$]{CBL07}, we conclude that for any $L>0$,
\begin{align*}
|Rm_i| \le C(n)
\end{align*}
on $B_{g_i}(x_i,L) \times [0,\delta]$, if $i$ is sufficiently large. Therefore, we conclude that 
\begin{align*}
(M_i, x_i, g_i(t))_{t \in [0,\delta]} \longright{C^{\infty}-Cheeger-Gromov} (M_{\infty}, x_{\infty},g_{\infty}(t))_{t \in [0,\delta]}
\end{align*}
such that $(M_{\infty}, g_{\infty}(t))$ has unifomrly bounded curvature. From the uniqueness of the Ricci flow with bounded curvature \cite{CZ06} and the Claim, we have
\begin{align*}
(M_{\infty}, x_{\infty}, g_{\infty}(t))_{t \in [0,\delta]}=(S^{n-1} \times \R, p_c,g_c(t))_{t \in [0,\delta]}.
\end{align*}
Therefore, for $ 0 \le k \le 4$,
\begin{align*}
\lim_{i \to \infty} \frac{|\na^k Rm_i|}{R^{\frac{k}{2}+1}} (q_i,0)=\lim_{i \to \infty} \frac{|\na^k Rm_i|}{R^{\frac{k}{2}+1}} (x_i,\delta)=\frac{|\na^k Rm_{g_c}|}{R_{g_c}^{\frac{k}{2}+1}}=
\begin{cases}
\sqrt{\frac{2}{(n-1)(n-2)}} \quad &\text{if} \quad k=0,\\
0 \quad &\text{if} \quad k>0.
\end{cases}
\end{align*}
However, this contradicts \eqref{E416xx} since $B \ge 1$. In sum, we obtain a contradiction, and the proof is complete.
\end{proof}

From Proposition \ref{prop:405a}, $E$ is an end of the Ricci shrinker. Next, we show that any point in $E$ far away from $\Sigma(t_0)$ is the center of an evolving normalized $\ep$-neck.

\begin{prop} \label{prop:405b}
With the above assumptions \eqref{E401} and \eqref{E416ab}, for any $\ep>0$, there exists a constant $\eta_3=\eta_3(\ep,n,\delta_0,A,B)>0$ such that if $\ep_0\le \eta_3$, then any point $x \in \Sigma(\eta_3^{-1}t_0,\infty) \subset E$ is the center of an evolving normalized $\ep$-neck.
\end{prop}
\begin{proof}
We prove it by contradiction. Suppose for fixed $\ep$ the conclusion does not hold. Then there exists a sequence of Ricci shrinkers $(M^n_i,p_i,g_i,f_i)$ satisfying all the assumptions with $\ep_{0,i} \to 0$ and $x_i \in \Sigma(i^2 t_{0,i},\infty)$ which is not the center of an evolving normalized $\ep$-neck. By the identical proof as that of Proposition \ref{prop:405a}, we conclude that
\begin{align*}
(M_i, x_i, g_i) \longright{C^{\infty}-Cheeger-Gromov} (S^{n-1} \times \R, p_c,g_c).
\end{align*}
For a fixed number $L>0$, we consider any spacetime point $(y,\bar t) \in B_{g_i}(x_i,L) \times [-L,0]$. From
\begin{align*}
0\le \frac{df_i(\psi_i^t(y))}{dt}=\frac{f_i(\psi^t_i(y))-R_i(\psi^t_i(y))}{1-t} \le \frac{f_i(\psi^t_i(y))}{1-t},
\end{align*}
we have
\begin{align}\label{E420aa}
f_i(y) \ge f_i(\psi_i^{\bar t}(y)) \ge \frac{f_i(y)}{1-\bar t} \ge \frac{f_i(y)}{1+L}.
\end{align}

On the other hand, if we set $f_i(x_i)=t_i$, then similar to \eqref{E418c}, 
\begin{align} \label{E420ab}
\lim_{i \to \infty} \frac{f_i(y)}{t_i}=1
\end{align}
for $y \in B_{g_i}(x_i,L)$ uniformly. In particular, \eqref{E420aa} and \eqref{E420ab} imply $\psi_i^{\bar t}(y) \in E_i$ since $t_i \ge i^2 t_{0,i}$. Now we compute 
\begin{align*}
|Rm_i|(y,\bar t)=\frac{|Rm_i|(\psi_i^{\bar t}(y),0)}{1-\bar t}
\end{align*}
which is uniformly bounded for $(y,\bar t) \in B_{g_i}(x_i,L) \times [-L,0]$. Combining this fact with Theorem \ref{thm:nonc}, we conclude that
\begin{align*}
(M_i, x_i, g_i(t))_{t \le 0} \longright{C^{\infty}-Cheeger-Gromov} (M^n_{\infty}, x_{\infty},g_{\infty}(t))_{t \le 0}
\end{align*}
such that $(M^n_{\infty} ,g_{\infty}(t))_{t \le 0}$ has uniformly bounded curvature. By the same proof as in Proposition \ref{prop:405a}, $(M^n_{\infty}, g_{\infty}(0))=(S^{n-1} \times \R,g_c)$. Therefore, it follows from the backward uniqueness of the Ricci flow \cite[Theorem $1.1$]{Kot10} that $g_{\infty}(t)= g_c(t)$ for any $t \le 0$. However, this implies that $x_i$ is the center of an evolving normalized $\ep$-neck for large $i$, which is a contradiction.
\end{proof}

Next, we show that $E$ is asymptotic to the cylinder at some fixed rate depending only on $n$.

From Proposition \ref{prop:405b}, there exists a large number $t'>t_0$ such that on $E':=\Sigma(t',\infty)$ we have
\begin{align*}
|\na^i Rm| \le C(n)R^{\frac{i}{2}+1}
\end{align*}
for any $0 \le i\le 4$ and Theorem \ref{thm:nonc} holds for some $\kappa$ depending only on $n$. Therefore, it follows from \eqref{E407b} that on $E'$ we have
\begin{align*}
\ds G \ge \la \na G, \na f\ra-c_2 f^{-1}(1+R) |\la \na G,\na f \ra|-2\la \na G,\na \log R \ra_{\Sigma}+c_1 RG-c_2 f^{-1}R(1+R), 
\end{align*}
where $G=\frac{|Rm|^2}{R^2}-\frac{2}{(n-1)(n-2)}$ and $c_1,c_2$ depend only on $n$. If we set $k(t)=\sup_{\Sigma(t)} G$, then for any $t \ge t'$, there exists a maximum point $z \in \Sigma(t)$ such that at $z$,
\begin{align}
0 \ge \la \na G, \na f\ra-c_4 t^{-1} |\la \na G,\na f \ra|+c_3 R k-c_4 Rt^{-1}, \label{Ex418c}
\end{align}
for $c_3$ and $c_4$ depending on $n$, where we have used the fact that $R$ is uniformly bounded on $E$.

If $k(t) \ge 2c_4 c_3^{-1} t^{-1}$, then by \eqref{Ex418c} we have
\begin{align}
0 \ge \la \na G, \na f\ra-c_3 t^{-1} |\la \na G,\na f \ra|+\frac{c_4}{2}R k. \label{Ex418d}
\end{align}

It is clear from \eqref{Ex418d} that $ \la \na G, \na f\ra \le 0$ at $z$. Therefore, we obtain
\begin{align}
k'(t) \le \frac{\la \na G, \na f\ra}{|\na f|^2}(z) \le - \frac{c_4 R k(t)}{2|\na f|^2(1-c_3t^{-1})}(z) \le -\tau_1 k(t) t^{-1}
\label{Ex418e}
\end{align}
for $\tau_1=\tau_1(n)\in (0,1)$, where the last inequality follows from \eqref{E416ab}. If we consider $\max\{k(t), 2c_4 c_3^{-1} t^{-1} \}$ and \eqref{Ex418e}, it is easy to see that for any $t \ge t'$,
\begin{align}
k(t)=O( t^{-\tau_1}).
\label{Ex418f}
\end{align}

Combining \eqref{Ex418f} with \eqref{E407aa}, we conclude that
\begin{align}
\frac{|\Rm|}{R}=O(t^{-\tau_2}),
\label{Ex419}
\end{align}
on $\Sigma(t)$, where $\tau_2=\tau_1/2$.

Next, we prove

\begin{prop} \label{prop:406}
With the above assumptions, for any sequence $x_i \in E$ with $x_i \to \infty$, we have
\begin{align*}
(M, x_i, g) \longright{C^{\infty}-Cheeger-Gromov} \lc S^{n-1} \times \R,p_c,g_c \rc.
\end{align*}
\end{prop}

\begin{proof}
For any sequence $x_i \to \infty$ with $t_i=f(x_i)$, we conclude that
\begin{align*}
(M_i, x_i, g_i) \longright{C^{\infty}-Cheeger-Gromov} (M_{\infty}, x_{\infty},g_{\infty}),
\end{align*}
where the limit is a complete smooth Riemannian manifold. Here, the smooth convergence follows from the same proof as that of the Claim in Proposition \ref{prop:405a}. Moreover, by considering $f_i=t_i^{-\frac{1}{2}}(f_i-t_i)$, we can show as in Proposition \ref{prop:405a} that the limit $(M_{\infty},g_{\infty})$ splits isometrically as $(\Sigma_{\infty} \times \R,g'_{\infty} \times g_E )$. From \eqref{Ex419}, it is clear that $(\Sigma_{\infty},g'_{\infty})$ is isometric to $(S^{n-1},g_S)$ up to scaling, where $g_S$ is the standard metric on $S^{n-1}$ such that the scalar curvature is $\frac{n-1}{2}$. In particular, if the limit $(M_{\infty},g_{\infty})$ is a Ricci shrinker, then $(\Sigma_{\infty},g'_{\infty})$ is exactly isometric to $(S^{n-1},g_S)$.

Therefore, we conclude from \cite[Proposition 5.2]{Naber} that for any $x_i \to \infty$
\begin{align*}
(M_i, x_i, g_i) \longright{C^{\infty}-Cheeger-Gromov} (S^{n-1}\times \R,p_c,g_c)
\end{align*}
and the proof is complete.
\end{proof}

If we denote the induced metric on $\Sigma(t)$ by $\tilde g(t)$, then it follows from \eqref{Ex419}, Proposition \ref{prop:406} and the same proof as in \cite[Theorem $4.1$, Page 919-922]{MW19} that
\begin{align}
\tilde g(t)=g_S+O(t^{-\tau}), \label{E422x}
\end{align}
where $\tau=\tau(n)>0$. Similar estimates also hold for the higher derivatives of $\tilde g(t)$. Now we define a map
\begin{align*}
\Psi: S^{n-1} \times [t_0,\infty) \longrightarrow E
\end{align*}
by $\Psi(z,t)=\varphi_t(z)$. Here $\varphi_t$ is defined in \eqref{E412} and we have identified $S^{n-1}$ with $\Sigma(t_0)$ by a diffeomorphism.

Therefore, it follows from \eqref{E422x} that on $S^{n-1} \times \{t\}$,
\begin{align*}
|\Psi^*g-g_c|=O(t^{-\tau}). 
\end{align*}
If we consider the estimates of the second fundamental form, then the estimates for all higher derivatives can also be derived. More precisely, we have on $S^{n-1} \times \{t\}$ that
\begin{align*}
|\na^k_{g_c}(\Psi^*g)|=O(t^{-\tau_k})
\end{align*}
for some $\tau_k=\tau_k(n)>0$. In other words, the end $E$ is smoothly asymptotic to $S^{n-1} \times \R$.

\begin{rem}
In \emph{\cite[Theorem $1.6$]{MW19}}, Munteanu and Wang have obtained a similar result with an extra assumption that the Ricci shrinker considered has uniformly bounded curvature. This assumption plays an important role in \emph{\cite[Theorem $3.1$]{MW19}} for getting the higher-order estimates $|\na^k Rm| \le B_k R^{\frac{k}{2}+1}$ for any $k \ge 0$. Notice that we drop the uniformly bounded curvature condition by applying the pseudolocality theorem (cf. \emph{\cite[Corollary 10.6]{LW20}, ~\cite[Theorem 10.1]{Pe1},~\cite[Theorem 1.2 ]{BWang20}}).
\end{rem}

\subsection*{Conical end}

In this subsection, we assume that \eqref{E416ab} is violated in one direction and hence there exists a $t_1 \in ( t_0,\infty)$ such that 
\begin{align} \label{E450a}
\left| R -\frac{n-1}{2} \right| < \ep_2
\end{align}
on $\Sigma(t_0,t_1)$ and
\begin{align} \label{E451a}
\min_{\Sigma(t_1)} R = \frac{n-1}{2}-\ep_2.
\end{align}
We define
\begin{align}
E^1 \coloneqq \Sigma(t_1,s) \subset E. \label{eqn:OH05_4}
\end{align}

\begin{prop} \label{prop:402}
With the assumptions above, there exist positive constants $\hat c_1$ and $\hat c_2$ depending only on $n,A$ and $B$ such that if $\ep_0$ is sufficiently small, then
\begin{align}
\frac{1}{\frac{2}{n-1}+\hat c_2 \ep_2 t_1^{-1}t} \le R(x) \le \frac{1}{\frac{2}{n-1}+\hat c_1 \ep_2 t_1^{-1}t} \label{E417}
\end{align}
for any $x \in E^1$, where $t=f(x)$.
\end{prop}

\begin{proof}
In the proof, all constants $c_i,i=1,2,\cdots$ depend only on $n,A$ and $B$.

By our assumption, there exists $x \in \Sigma(t_1)$ such that 
\begin{align*}
R(x)=\frac{n-1}{2}-\ep_2.
\end{align*}
From Lemma \ref{L404}, we may assume on $\Sigma(t_1)$,
\begin{align}
\frac{n-1}{2}-\frac{3}{2}\ep_2 \le R \le \frac{n-1}{2}-\frac{1}{2}\ep_2. \label{E417a}
\end{align}
In addition, it follows from \eqref{E413} and \eqref{E413a} that
\begin{align}
\lc \frac{2}{n-1}-c_1 \ep_0^{\frac{1}{3}}\rc R^2-R \le tR_t \le \lc \frac{2}{n-1}+c_1 \ep_0^{\frac{1}{3}}\rc R^2-R \label{E417b}
\end{align}
for some constant $c_1>0$. By considering the ODE with the given initial values in \eqref{E417a}, we obtain
\begin{align}
\frac{1}{\frac{2}{n-1}-c_1\ep_0^{\frac{1}{3}}+c_3 \ep_2 t_1^{-1}t} \le R(t) \le \frac{1}{\frac{2}{n-1}+c_1\ep_0^{\frac{1}{3}}+c_2 \ep_2 t_1^{-1}t} \label{E417c}
\end{align}
for any $t \ge t_1$. Here, the constants $c_2$ and $c_3$ are determined by
\begin{align}
\frac{1}{\frac{2}{n-1}+c_1\ep_0^{\frac{1}{3}}+c_2 \ep_2}=\frac{n-1}{2}-\frac{1}{2}\ep_2 \label{E417d}
\end{align}
and
\begin{align}
\frac{1}{\frac{2}{n-1}-c_1\ep_0^{\frac{1}{3}}+c_3 \ep_2}=\frac{n-1}{2}-\frac{3}{2}\ep_2. \label{E417e}
\end{align}

It is clear from \eqref{E417c}, \eqref{E417d}, \eqref{E417e} and the definition of $\ep_2$ that \eqref{E417} holds if $\ep_0$ is sufficiently small.
\end{proof}

Similar to Proposition \eqref{prop:405a}, we have the following result. 

\begin{prop} \label{prop:406a}
Under the assumptions \eqref{E401}, \eqref{E450a} and \eqref{E451a}, there exists $\sigma_3=\sigma_3(n,\sigma_0,A,B)>0$ such that if $\ep_0 \le \sigma_3$, then $s=\infty$.
\end{prop}

\begin{proof}
We prove it by contradiction. Similar to Proposition \eqref{prop:405a}, we assume there exists a sequence of Ricci shrinkers $(M_i, g_i, f_i) $ with the end $E_i=\Sigma(t_{0,i},s_i)$ satisfying \eqref{E401}, \eqref{E450a} and \eqref{E451a} with $\ep_{0,i} \to 0$.
Moreover, there exists a point $q_i \in \Sigma(s_i)$ at which
\begin{align*}
|\na^k Rm_i|= BR_{g_i}^{\frac{k}{2}+1}
\end{align*}
for a fixed $0 \le k \le 4$.

\textbf{Claim 1.} 
\begin{align}\label{E452b}
\liminf_{i \to \infty} \ep_{2,i}\frac{s_i}{t_{1,i}}>0.
\end{align}
\emph{Proof of Claim 1:} Otherwise, it follows from the assumption \eqref{E450a} and \eqref{E417} that
\begin{align*}
\limsup_{i \to \infty} \max_{\Sigma(t_{0,i},s_i)} \left| R_i-\frac{n-1}{2} \right|=0.
\end{align*}
Therefore, the same proof of Proposition \ref{prop:405a} yields a contradiction.

Now we set $r^{-2}_i=\frac{2}{n-1}R_i(q_i)$, then it follows from \eqref{E417} that 
\begin{align*}
r_i=\sqrt{ 1+\frac{n-1}{2}\sigma_i s_i} \ge 1
\end{align*}
where 
\begin{align*}
\hat c_1 \ep_{2,i} t_{1,i}^{-1} \le \sigma_i \le \hat c_2 \ep_{2,i} t_{1,i}^{-1}.
\end{align*}

Next we fix a constant $\delta=\delta(n) \in (0,1/20)$ determined later and prove the following result.

\textbf{Claim 2.} There exists a constant $c_1=c_1(n)>0$ such that for any $t \in [-2\delta r_i^2,0]$,
\begin{align}\label{E452c}
\frac{s_i}{1+|t|} \le f_i(\psi_i^t(q_i)) \le \frac{s_i-c_1}{1+|t|}+c_1
\end{align}
\emph{Proof of Claim 2:} For any $t \in [-2\delta r_i^2,0]$ we have
\begin{align*}
\frac{df_i(\psi_i^t(q_i))}{dt}=\frac{|\na_{g_i} f_i|_{g_i}^2}{1-t}(\psi_i^t(q_i))=\frac{f_i-R_i}{1-t}(\psi_i^t(q_i))
\end{align*}
Since $0 \le R_i(\psi_i^t(q_i)) \le c_1$, it is clear that
\begin{align*}
\frac{f_i(\psi_i^t(q_i))-c_1}{1-t} \le \frac{df_i(\psi_i^t(q_i))}{dt} \le \frac{f_i(\psi_i^t(q_i))}{1-t}.
\end{align*}
By solving the above ODE, we immediately obtain \eqref{E452c}.

Now we set $q_i^t:=\psi_i^t(q_i)$ and $s_i^t:=f_i(\psi_i^t(q_i))$. Moreover, for any fixed number $L>1$ and $t \in [-2\delta r_i^2,0]$, we consider $y \in B_{g_i}\lc q^t_i,L(1+|t|)^{-\frac1 2}r_i \rc$ and obtain
\begin{align*}
|2\sqrt{f_i(y)}- 2\sqrt{s^t_i}|=|2\sqrt{f_i(y)}- 2\sqrt{f_i(q^t_i)}| \le \frac{Lr_i}{\sqrt{1+|t|}}=L \sqrt{\frac{ 1+\frac{n-1}{2}\sigma_i s_i}{1+|t|}}. 
\end{align*}
From \eqref{E452c} and the fact that $\lim_{i \to \infty} \sigma_i=0$, we obtain
\begin{align}
\lim_{i \to \infty} \frac{f_i(y)}{s^t_i}=1 \label{E453c}
\end{align}
uniformly for $(y,t) \in B_{g_i}\lc q^t_i,L(1+|t|)^{-\frac1 2}r_i \rc \times [-2\delta r_i^2,0]$. 

In particular, for $(y,t) \in B_{g_i}\lc q^t_i,L(1+|t|)^{-\frac1 2}r_i \rc \times [-2\delta r_i^2,-L^{-1}r_i^2]$, we obtain
\begin{align}
s_i^t \ge \frac{s_i}{1+|t|} \ge \frac{s_i}{1+2\delta r_i^2} \ge t_{1,i}, \label{E453d}
\end{align}
where we have used \eqref{E452b} for the last inequality. Moreover, since $r_i \ge 1$, we have
\begin{align*}
s_i^t \le \frac{s_i-c_1}{1+|t|}+c_1 \le \frac{s_i-c_1}{1+L^{-1} r_i^2}+c_1 \le \frac{s_i-c_1}{1+L^{-1}}+c_1.
\end{align*}
Therefore, it follows from \eqref{E453c} that
\begin{align*}
t_{1,i} \le f_i(y) \le s_i 
\end{align*}
for sufficiently large $i$, which implies that $y \in E^1_i$.

\textbf{Claim 3.} If we set $\tilde g_i(t)=r_i^{-2}g_i(r_i^2 t)$, then
\begin{align*}
(M_i, q_i, \tilde g_i(t))_{t\in [-\delta,0)} \longright{C^{\infty}-Cheeger-Gromov} (M^n_{\infty}, q_{\infty},g_{\infty}(t))_{t\in [-\delta,0)}
\end{align*}
for a complete smooth Ricci flow $(M^n_{\infty}, g_{\infty}(t))_{t\in [-\delta,0)}$ with uniformly bounded curvature.

\emph{Proof of Claim 3:} For any $L>1$ and $(y,t) \in B_{g_i}\lc q^t_i,L(1+|t|)^{-\frac1 2}r_i \rc \times [-2\delta r_i^2,-L^{-1}r_i^2]$ it follows from \eqref{E417} and \eqref{E453c} that
\begin{align*}
R_i(y) \le \frac{1}{\frac{2}{n-1}+\hat c_1\frac{\ep_{2,i}}{t_{1,i}} f_i(y) } \le \frac{2}{\frac{2}{n-1}+\hat c_1\frac{\ep_{2,i}}{t_{1,i}} s_i^t } .
\end{align*}
Therefore, we obtain
\begin{align}
\frac{1}{1+|t|} R_i(y) \le \frac{s_i^t}{s_i}R_i(y) \le \frac{2}{s_i} \frac{1}{\frac{2}{n-1} (s^t_i)^{-1}+\hat c_1\frac{\ep_{2,i}}{t_{1,i}} } \le \frac{K}{\frac{2}{n-1} +\hat c_2\frac{\ep_{2,i}}{t_{1,i}}s_i } \le K R_i(q_i)=\frac{n-1}{2}Kr_i^{-2}. \label{E454ab}
\end{align}
Here, $K:=2(1+\hat c_2/\hat c_1)$ and the third inequality follows from \eqref{E452b}.

From \eqref{E454ab}, it is clear that $R_{\tilde g_i}(y,t)$ is uniformly bounded for any $t \in [-2\delta, -L^{-1}]$ and $y \in B_{\tilde g_i(t)}(q_i,L)$. Since $|Rm_{\tilde g_i}|$ is controlled by $R_{\tilde g_i}$ by our assumption \eqref{E401}, we immediately conclude from the Shi's local estimates and Theorem \ref{thm:nonc} that
\begin{align*}
(M_i, q_i, \tilde g_i(t))_{t\in [-\delta,0)} \longright{C^{\infty}-Cheeger-Gromov} (M^n_{\infty}, q_{\infty},g_{\infty}(t))_{t\in [-\delta,0)}
\end{align*}
for a complete smooth Ricci flow $(M^n_{\infty}, g_{\infty}(t))_{t\in [-\delta,0)}$ with uniformly bounded curvature. 

We have

\textbf{Claim 4.} There exists a constant $\lambda \ge K^{-1}$ such that
\begin{align*}
(M_{\infty}, g_{\infty}(-\delta))=(S^{n-1} \times \R, \lambda g_c),
\end{align*}
where $K$ is the same constant in \eqref{E454ab}.

\emph{Proof of Claim 4:} We set 
\begin{align*}
\tilde f_i:= r_i^{-1} t_i^{-\frac{1}{2}} (F_i(-\delta r_i^2)-t_i),
\end{align*}
where $F_i(t):=(1-t)f_i(t)$ and $t_i:=F_i(-\delta r_i^2)(q_i)$. From a direct calculation, 
\begin{align}
|\na_{\tilde g_i(-\delta)} \tilde f_i|^2_{\tilde g_i(-\delta)}=\frac{|\na_{g_i(-\delta r_i^2)} F_i(-\delta r_i^2)|_{g_i(-\delta r_i^2)}^2}{ t_i}=\frac{F_i(-\delta r_i^2)-(1+\delta r_i^2)^2R_{g_i(-\delta r_i^2)}}{t_i}, \label{E455a}
\end{align}
where we have used the identity $(1-t)^2R_i+|\na F_i|^2=F_i$. From \eqref{E454ab}, we obtain $(1+\delta r_i^2)R_{g_i(-\delta r_i^2)}$ is uniformly bounded on $B_{\tilde g_i(-\delta)}(q_i,L)$. Moreover, 
\begin{align}
t_i=(1+\delta r_i^2) f_i(q^{-\delta r_i^2}_i)=(1+\delta r_i^2) s_i^{-\delta r_i^2}, \label{E455b}
\end{align}

Therefore, it follows from \eqref{E453d} that
\begin{align}
\lim_{i \to \infty} \frac{(1+\delta r_i^2)^2R_{g_i(-\delta r_i^2)}}{t_i}=0 \label{E455c}
\end{align}
uniformly on $B_{\tilde g_i(-\delta)}(q_i,L)$. On the other hand, it follows from \eqref{E455a} that $F_i(-\delta r_i^2)$ is $r_i/2$-Lipschitz with respect to $\tilde g_i(-\delta)$. Therefore, for any $y \in B_{\tilde g_i(-\delta)}(q_i,L)$,
\begin{align}
\tilde f_i(y) =& r_i^{-1} t_i^{-\frac{1}{2}}(F_i(y,-\delta r_i^2)-t_i) \notag\\
=& r_i^{-1} t_i^{-\frac{1}{2}} (\sqrt{F_i(y,-\delta r_i^2)}-\sqrt{t_i})(\sqrt{F_i(y,-\delta r_i^2)}+\sqrt{t_i}) \notag\\
\le & r_i^{-1} t_i^{-\frac{1}{2}}\frac{L}{2}r_i 3\sqrt{t_i} \le \frac{3}{2}L, \label{E455d}
\end{align}
where we have used the fact that $\lim_{i \to \infty} t_i^{-1}r_i^{2}=0$ from \eqref{E455b}. Then it is clear that 
\begin{align*}
\lim_{i \to \infty} \frac{F_i(-\delta r_i^2)}{t_i}=\lim_{i \to \infty} \frac{\tilde f_i r_i t_i^{\frac 1 2}+t_i}{t_i}=1 
\end{align*}
uniformly on $B_{\tilde g_i(-\delta)}(q_i,L)$.
In addition, we compute
\begin{align}
&\text{Hess}_{\tilde g_i(-\delta)} \tilde f_i \notag\\ 
=& r_i^{-1} t_i^{-\frac{1}{2}} \text{Hess}_{g_i(-\delta r_i^2)} F_i \notag\\
=&\frac{1}{2} r^{-1}_it_i^{-\frac 1 2} g_i(-\delta r_i^2)-(1+\delta r_i^2) r_i^{-1} t_i^{-\frac{1}{2}} Rc_{g_i(-\delta r_i^2)} \notag\\
=&\frac{1}{2} r_i t_i^{-\frac 1 2}\tilde g_i(-\delta)-(1+\delta r_i^2) r_i^{-1} t_i^{-\frac{1}{2}} Rc_{\tilde g_i(-\delta)}. \label{E455f}
\end{align}
Therefore, it follows from \eqref{E455d}, \eqref{E455f} and the standard regularity theorem for elliptic equations that $\tilde f_i$ converges smoothly to a smooth function $f_{\infty}$ on $M_{\infty}$ and
\begin{align}
\text{Hess}_{g_{\infty}} f_{\infty}=0, \label{E455g}
\end{align}
since $\displaystyle \lim_{i \to \infty} t_i^{-1}r_i^{2}=0$. Furthermore, it follows from the combination of \eqref{E455a}, \eqref{E455b} and \eqref{E455c} that 
\begin{align}
|\na_{g_{\infty}} f_{\infty}|^2=1.\label{E455ff}
\end{align}
Therefore, $f_{\infty}$ is not a constant. 
From \eqref{E455g} and \eqref{E455ff}, we conclude that $(M_{\infty},g_{\infty})$ is isometric to $(\Sigma_{\infty} \times \R,g'_{\infty} \times g_E )$, where $\Sigma_{\infty}$ is the level set of $f_{\infty}$. 
It follows immediately from Lemma \ref{L404} that $(M_{\infty}, g_{\infty}) =(S^{n-1} \times \R, \lambda g_c)$ for some constant $\lambda >0$. Notice that $(M_{\infty}, g_{\infty})$ cannot be flat $\R^n$ since by \eqref{E417} and \eqref{E452c} that $R_i(q_i^{-\delta})/R_i(q_i) \ge C>0$ uniformly. From the estimate \eqref{E454ab}, it is clear that $\lambda \ge K^{-1}$.

Now we fix the parameter $\delta=\delta(n)>0$ small enough such that the pseudolocality Theorem \cite[Theorem 24]{LW20} can be applied. Combining with \cite[Theorem $3.1$]{CBL07}, we conclude that for any $L>0$,
\begin{align*}
|Rm_{\tilde g_i}| \le C(n)
\end{align*}
on $B_{\tilde g_i}(q_i,L) \times [-\delta,0]$, if $i$ is sufficiently large. Therefore, we conclude that 
\begin{align*}
(M_i, q_i, \tilde g_i(t))_{t \in [-\delta,0]} \longright{C^{\infty}-Cheeger-Gromov} (S^{n-1} \times \R, p_c, g_c( t))_{t \in [-\delta,0]}
\end{align*}
since $R_{\tilde g_i}(q_i,0)=\frac{n-1}{2}$. Therefore, we can derive the contradiction as in the proof of Proposition~\ref{prop:405a}. 
\end{proof}

Proposition \ref{prop:406a} implies that $E$ is an end of the Ricci shrinker. Moreover, it follows from \eqref{E417} and \eqref{E401} that $|Rm|$ is quadratically decaying. 
It is well-known (e.g., \cite{KW15}) that the tangent cone at infinity concerning this end is a metric cone over a smooth cross-section.

Next, we show that any point in $E$ far away from $\Sigma(t_0)$ is the center of an evolving $\ep$-neck. We only sketch the proof as it is similar to the proofs of Proposition \ref{prop:405b} and Proposition \ref{prop:406a}.

\begin{prop} \label{prop:405c}
Under the assumptions \eqref{E401}, \eqref{E450a} and \eqref{E451a}, for any $\ep>0$, there exists a constant $\eta_4=\eta_4(\ep,n,\delta_0,A,B)>0$ such that if $\ep_0\le \eta_4$, then any point $x \in \Sigma(\eta_4^{-1} t_0,\infty) \subset E^1$ is the center of an evolving $\ep$-neck.
\end{prop}

\begin{proof}
We prove it by contradiction. Suppose for fixed $\ep,n$ and $A$ the conclusion does not hold. Then there exists a sequence of Ricci shrinkers $(M^n_i,p_i,g_i,f_i)$ satisfying all the assumptions and $x_i \in \Sigma(i^2 t_{0,i},\infty)$ which is not the center of an evolving $\ep$-neck. 
Now we set $t_i=f_i(x_i)$. We claim that
\begin{align}\label{E457aaa}
\liminf_{i \to \infty} \ep_{2,i}\frac{t_i}{t_{1,i}}>0.
\end{align}
Otherwise, it follows from the assumption \eqref{E450a} and \eqref{E417} that
\begin{align*}
\limsup_{i \to \infty} \max_{\Sigma(t_{0,i},t_i)} \left| R_i-\frac{n-1}{2} \right|=0.
\end{align*}
Therefore, the same proof of Proposition \eqref{prop:405b} shows that $x_i$ is the center of an evolving normalized $\ep$-neck for sufficiently large $i$.

By the same proof as in that of Proposition \ref{prop:406a} we have
\begin{align}
(M_i, x_i, \tilde g_i) \longright{C^{\infty}-Cheeger-Gromov} \lc S^{n-1} \times \R,p_c,g_c \rc \label{E457a}
\end{align}
where $\tilde g_i=r_i^{-2} g_i$ and $R_i(x_i)=\frac{n-1}{2}r_i^{-2}$. For a fixed number $L>1$ and $y \in B_{g_i}(x_i,Lr_i)$, one can prove as \eqref{E453c} that
\begin{align*}
\lim_{i \to \infty} \frac{f_i(y)}{t_i}=1
\end{align*}
uniformly for $y$, where $t_i=f_i(x_i)$. If we set $\psi_i^t(y)=y^t$, then we can prove as \eqref{E452c} that
\begin{align}
\frac{9}{10}\frac{t_i}{1+|t|} \le f_i(y^t) \le \frac{11}{10}\frac{t_i}{1+|t|} \label{E457c}
\end{align}
for any $(y,t) \in B_{g_i}(x_i,Lr_i^2) \times [-Lr_i^2,0]$. Since
\begin{align*}
r_i=\sqrt{ 1+\frac{n-1}{2}\sigma_i t_i} 
\end{align*}
for some 
\begin{align*}
\hat c_1 \ep_{2,i} t_{1,i}^{-1} \le \sigma_i \le \hat c_2 \ep_{2,i} t_{1,i}^{-1},
\end{align*}
it is clear from \eqref{E457aaa} and \eqref{E457c} that
\begin{align*}
f_i(y^t)\ge \frac{9}{10}\frac{t_i}{1+Lr_i^2}= \frac{9}{10}\frac{t_i}{1+L(1+\frac{n-2}{2}\sigma_i t_i)} \ge t_{1,i}.
\end{align*}
In particular, $y^t \in E^1_i$. Therefore, we can apply \eqref{E417} and prove exactly as \eqref{E454ab} that
\begin{align*}
\frac{1}{1+|t|}R_i(y^t) \le \frac{f_i(y^t)}{f_i(y)} R_i(y^t) \le \frac{2}{t_i} \frac{1}{\frac{2}{n-1} u_i^{-1}+\hat c_1\frac{\ep_{2,i}}{t_{1,i}} } \le \frac{K}{\frac{2}{n-1} +\hat c_2\frac{\ep_{2,i}}{t_{1,i}}t_i } \le K r_i^{-2}
\end{align*}
for $(y,t) \in B_{g_i}(x_i,Lr_i^2) \times [-Lr_i^2,0]$, where $u_i=f_i(y^t)$. Therefore, we conclude that for $\tilde g_i(t):=r_i^{-2}g_i(r_i^2t)$, $R_{\tilde g_i}(y,t)$ is uniformly bounded on $B_{\tilde g_i(0)}(y,L) \times [-L,0]$ since 
\begin{align*}
R_{\tilde g_i}(y,t)=\frac{r_i^2}{1+|t|}R_i(y^t) \le K.
\end{align*}

Then it is clear as before that 
\begin{align*}
(M_i, x_i, \tilde g_i(t))_{t \le 0} \longright{C^{\infty}-Cheeger-Gromov} (M^n_{\infty}, x_{\infty}, g_{\infty}(t))_{t \le 0}
\end{align*}
for a complete smooth Ricci flow solution $(M_{\infty}, g_{\infty}(t))_{t \le 0}$ with uniformly bounded curvature. 
From \eqref{E457a} and the backward uniqueness of the Ricci flow \cite[Theorem $1.1$]{Kot10}, we conclude that 
\begin{align*}
(M_{\infty}, g_{\infty}(t))_{t \le 0}=(S^{n-1} \times \R ,g_c(t))_{t \le 0}. 
\end{align*}
However, this implies that $x_i$ is the center of an evolving $\ep$-neck for large $i$, which is a contradiction.
\end{proof}

\subsection*{Cap region}

In this subsection, we assume that \eqref{E416ab} is violated in the other direction and hence there exists a $t_2 \in (t_0,\infty)$ such that on $\Sigma(t_0,t_2)$
\begin{align} \label{E460a}
\left| R -\frac{n-1}{2} \right| < \ep_2
\end{align}
and
\begin{align} \label{E460b}
\max_{\Sigma(t_2)} R = \frac{n-1}{2}+\ep_2.
\end{align}

Now we set $E^2=\Sigma(t_2,s) \subset E$ and prove the following result.

\begin{prop} \label{prop:460}
With the assumptions above, there exists a positive constants $\hat c_3,\hat c_4$ depending only on $n,A$ and $B$ such that 
\begin{align}
\frac{1}{\frac{2}{n-1}-\hat c_3 \ep_2 t_2^{-1}t} \le R(x) \le \frac{1}{\frac{2}{n-1}-\hat c_4 \ep_2 t_2^{-1}t} \label{E461}
\end{align}
for any $x \in E^2$, where $t=f(x)$. Moreover, $R(\psi^t(x))$ is increasing for $t \ge 0$ as long as $\psi^t(x)$ stays in $E^2$.
\end{prop}

\begin{proof}
By our assumption, there exists $z \in \Sigma(t_2)$ such that 
\begin{align*}
R(z)=\frac{n-1}{2}+\ep_2.
\end{align*}
From Lemma \ref{L404}, we may assume on $\Sigma(t_2)$ that
\begin{align}
\frac{n-1}{2}+\frac{1}{2}\ep_2 \le R \le \frac{n-1}{2}+\frac{3}{2}\ep_2. \label{E461a}
\end{align}
Recall from \eqref{E417b} we have
\begin{align}
\lc \frac{2}{n-1}-c_1 \ep_0^{\frac{1}{3}}\rc R^2-R \le tR_t \le \lc \frac{2}{n-1}+c_1 \ep_0^{\frac{1}{3}}\rc R^2-R \label{E461b}
\end{align}
for some constant $c_1=c_1(n)>0$. By considering the ODE with the given initial values in \eqref{E461a}, we obtain
\begin{align}
\frac{1}{\frac{2}{n-1}-c_1\ep_0^{\frac{1}{3}}-c_2 \ep_2 t_2^{-1}t} \le R(t) \le \frac{1}{\frac{2}{n-1}+c_1\ep_0^{\frac{1}{3}}-c_3 \ep_2 t_2^{-1}t} \label{E461c}
\end{align}
for any $t \ge t_2$. Here, the constants $c_2$ and $c_3$ are determined by
\begin{align}
\frac{1}{\frac{2}{n-1}-c_1\ep_0^{\frac{1}{3}}-c_2 \ep_2}=\frac{n-1}{2}+\frac{1}{2}\ep_2. \label{E461d}
\end{align}
and 
\begin{align}
\frac{1}{\frac{2}{n-1}+c_1\ep_0^{\frac{1}{3}}-c_3 \ep_2}=\frac{n-1}{2}+\frac{3}{2}\ep_2. \label{E461e}
\end{align}

It is clear from \eqref{E461c}, \eqref{E461d}, \eqref{E461e} and the definition of $\ep_2$ that \eqref{E461} holds if $\ep_0$ is sufficiently small. Now the last statement follows from \eqref{E461} and \eqref{E461b} since on $E^2$,
\begin{align*}
\lc \frac{2}{n-1}-c_1 \ep_0^{\frac{1}{3}}\rc R^2-R>0.
\end{align*}
\end{proof}

Parallel to Proposition \eqref{prop:405a} and Proposition \ref{prop:406a}, we have the following result. 

\begin{prop} \label{prop:462}
Under the assumptions \eqref{E401}, \eqref{E460a} and \eqref{E460b}, there exists $\sigma_4=\sigma_4(n,\delta_0,A,B)>0$ such that if $\ep_0 \le \sigma_4$, then $s<\infty$ and there exists a point $q \in \Sigma(s)$ such that
\begin{align}
R(q)=\ep_1 f(q), \label{eqn:OH05_5}
\end{align}
if $\ep_0$ is sufficiently small.
\end{prop}

\begin{proof}
By \eqref{E461}, the scalar curvature blowup before the time $t=\frac{2t_2}{(n-1)\ep_2 \hat c_4}$. Therefore, we must have $s<\infty$. 
We move on to show (\ref{eqn:OH05_5}). 

From our definition of $E$ in \eqref{E401}, there exists a point $q \in \Sigma(s)$ such that at $q$ either
\begin{align*}
R=\ep_1 f \quad \text{or} \quad |\na^k Rm|=BR^{\frac{k}{2}+1}
\end{align*}
for some $0 \le k \le 4$. Notice that the second possibility can be excluded exactly as Proposition \ref{prop:406a}, if $\ep_0$ is sufficiently small and we sketch the proof below.

Assume there exists a sequence of Ricci shrinkers $(M_i, g_i, f_i) $ with the end $E_i=\Sigma(t_{0,i},s_i)$ satisfying \eqref{E401}, \eqref{E460a} and \eqref{E460b} with $\ep_{0,i} \to 0$. Moreover, there exists a $q_i \in \Sigma(s_i)$ at which
\begin{align*}
|\na^k Rm_i|=B R_i^{\frac{k}{2}+1}
\end{align*}
for a fixed $0 \le k \le 4$.

\textbf{Claim 1.} 
\begin{align}\label{E463b}
\liminf_{i \to \infty} \ep_{2,i}\frac{s_i}{t_{2,i}}>0.
\end{align}
\emph{Proof of Claim 1:} Otherwise, it follows from the assumption \eqref{E460a} and \eqref{E461} that
\begin{align*}
\limsup_{i \to \infty} \max_{\Sigma(t_{0,i},s_i)} \left| R_i-\frac{n-1}{2} \right|=0.
\end{align*}
Therefore, the same proof of Proposition \eqref{prop:405a} yields a contradiction.

Now we set $r^{-2}_i=\frac{2}{n-1}R_i(q_i)$. Then it is clear from \eqref{E461} that $r_i \le 1$. Moreover, we define $q_i^t=\psi_i^t(q_i)$, $s_i^t=f_i(q_i^t)$ and fix a constant $\delta=\delta(n) \in (0,1/20)$ determined later.

\textbf{Claim 2.} For any $t \in [-2\delta r_i^2,0]$, we have
\begin{align}\label{E463c}
\frac{s_i}{1+|t|} \le s_i^t \le \frac{s_i}{(1+|t|)^{1-\ep_{1,i}}}.
\end{align}
\emph{Proof of Claim 2:} For any $t \in [-2\delta r_i^2,0]$ we have
\begin{align*}
\frac{df_i(q_i^t)}{dt}=\frac{|\na_{g_i} f_i|_{g_i}^2}{1-t}(q_i^t)=\frac{f_i-R_i}{1-t}(q_i^t).
\end{align*}
Since $0 \le R_i(q_i^t) \le \ep_{1,i} f_i(q_i^t)$ by our assumptions, it is clear that
\begin{align*}
(1-\ep_{1,i}) \frac{f_i(q_i^t)}{1-t} \le \frac{df_i(q_i^t)}{dt} \le \frac{f_i(q_i^t)}{1-t}.
\end{align*}
By solving the above ODE, we immediately obtain \eqref{E463c}.

\textbf{Claim 3.} For any fixed number $L>1$, $t \in [-2\delta r_i^2,-L^{-1}r_i^2]$ and $y \in B_{g_i}\lc q^t_i,L(1+|t|)^{-\frac1 2}r_i \rc$, we have
\begin{align}\label{E463d}
\frac{s_i}{1+2\delta} \le f_i(y) \le s_i
\end{align}
for sufficiently large $i$.

\emph{Proof of Claim 3:} We compute 
\begin{align*}
|2\sqrt{f_i(y)}- 2\sqrt{s^t_i}|=|2\sqrt{f_i(y)}- 2\sqrt{f_i(q^t_i)}| \le \frac{Lr_i}{\sqrt{1+|t|}}. 
\end{align*}

Therefore, it follows from \eqref{E463c} that
\begin{align}
2\sqrt{f_i(y)} \le 2\sqrt{s^t_i}+\frac{Lr_i}{\sqrt{1+|t|}} \le \frac{2\sqrt{s_i}}{(1+|t|)^{\frac{1-\ep_{1,i}}{2}}}+\frac{Lr_i}{\sqrt{1+|t|}}. \label{E463f}
\end{align}
Since $L^{-1} r_i^2 \le |t| \le \delta r_i^2 \le \delta$, there exists a universal constant $c_1>0$ such that
\begin{align}
2\sqrt{s_i}-\frac{2\sqrt{s_i}}{(1+|t|)^{\frac{1-\ep_{1,i}}{2}}} \ge c_1 |t| \sqrt{s_i} \ge c_1 L^{-1} r_i^2 \sqrt{s_i} \ge \frac{Lr_i}{\sqrt{1+|t|}}. \label{E463g}
\end{align}
Here, the last inequality follows from the fact that $\limsup_{i \to \infty} s_i^{-1}r^{-2}_i =0$ since $R_i \le \ep_{1,i} f_i$ on $E_i$. Therefore, the second inequality in Claim 3 follows from \eqref{E463f} and \eqref{E463g}. Similarly, we compute
\begin{align*}
2\sqrt{f_i(y)} \ge 2\sqrt{s^t_i}-\frac{Lr_i}{\sqrt{1+|t|}} \ge \frac{2\sqrt{s_i}-L}{\sqrt{1+|t|}} \ge \frac{2\sqrt{s_i}}{\sqrt{1+2\delta}}
\end{align*}
since $\lim_{i \to \infty} s_i=\infty$ and $|t| \le \delta$. Therefore, the first inequality in Claim 3 follows.

Combining \eqref{E463b} and \eqref{E463d}, we conclude that $y \in E^2_i$.

\textbf{Claim 4.} If we set $\tilde g_i(t)=r_i^{-2}g_i(r_i^2 t)$, then
\begin{align*}
(M_i, q_i, \tilde g_i(t))_{t\in [-\delta,0]} \longright{C^{\infty}-Cheeger-Gromov} (S^{n-1} \times \R, p_c, g_c( t))_{t \in [-\delta,0]}
\end{align*}
for a complete smooth Ricci flow $(M^n_{\infty}, g_{\infty}(t))_{t\in [-\delta,0)}$ with uniformly bounded curvature.

\emph{Proof of Claim 4:} For any $L>1$ and $(y,t) \in B_{g_i}\lc q^t_i,L(1+|t|)^{-\frac1 2}r_i \rc \times [-2\delta r_i^2,-L^{-1}r_i^2]$, it follows from Proposition \ref{prop:460} that $R_i(y)$ is bounded above by $R_i(\phi^u(y))$, 
where $u=s-f(y)$. Since the scalar curvature is almost constant on $\Sigma(s)$ by Lemma \ref{L404}, we have
\begin{align*}
\frac{1}{1+|t|} R_i(y) \le (n-1)r_i^{-2}. 
\end{align*}
The rest of the proof follows exactly like that of Proposition \ref{prop:406a}, and we can derive a contradiction as before.

In sum, the proof is complete.
\end{proof}

\begin{cor} \label{cor:sest}
Under the assumptions of Proposition \ref{prop:462}, there exists $\hat c_5=\hat c_5(n,A,B)>0$ such that
\begin{align} \label{Ex:400}
s \ge \frac{\hat c_5}{\ep_2} t_2 \quad \text{and} \quad R(q) \ge \frac{\hat c_5}{ \ep_2}. 
\end{align}
\end{cor}

\begin{proof}
By our assumption \eqref{E401}, $R(q)=\ep_1f(q) \ge n$. Therefore, it follows from \eqref{E461} that
\begin{align*}
s \ge \frac{n+1}{n(n-1) \hat c_4} \frac{t_2}{\ep_2}.
\end{align*}
Consequently, 
\begin{align*}
R(q)=\ep_1 s \ge \frac{n+1}{n(n-1) \hat c_4} \frac{\ep_1 t_2}{\ep_2} \ge \frac{n+1}{(n-1) \hat c_4} \frac{1}{\ep_2}.
\end{align*}
\end{proof}

Next, we show that any point in $E$ far away from $\Sigma(t_0)$ is the center of an evolving $\ep$-neck. We omit the proof as it is similar to the proofs of Proposition \ref{prop:405c} and Proposition \ref{prop:462}.

\begin{prop} \label{prop:463}
Under the assumptions \eqref{E401}, \eqref{E450a} and \eqref{E451a}, for any $\ep>0$, there exists $\eta_5=\eta_5(\ep,n,\delta_0,A,B)>0$ such that if $\ep_0\le \eta_5$, then any point $x \in \Sigma(\eta_5^{-1} t_0,s) \subset E^2$ is the center of an evolving $\ep$-neck.
\end{prop}

For later applications, we show that any point in $E^2$ has a positive curvature operator.

\begin{prop} \label{prop:464}
Under the assumptions \eqref{E401}, \eqref{E450a} and \eqref{E451a}, there exists $\sigma_5=\sigma_5(n,A,B)>0$ such that if $\ep_0 \le \sigma_5$, then
\begin{align*}
Rm \ge \delta \ep_2 \frac{R^2}{f}
\end{align*}
for $\delta=\delta(n)>0$ on $E^2$.
\end{prop}

\begin{proof}
From \eqref{E406a}, we have
\begin{align*}
R_{abcd}=\frac{R}{(n-1)(n-2)}(g_{ac}g_{bd}-g_{ad}g_{bc})+O(\ep_0 R).
\end{align*}

Moreover, it follows from \eqref{E202b} that
\begin{align*}
R_{abcn}=\frac{R_{abcl}f_l}{|\na f|}=\frac{\na_b R_{ac}-\na_a R_{bc} }{|\na f|}=\frac{\na_b \overset{\circ}{R}_{ac}-\na_a \overset{\circ}{R}_{bc}+\frac{\na_b R g_{ac}}{n-1}-\frac{\na_a R g_{bc}}{n-1} }{|\na f|}.
\end{align*}

Therefore, it is clear from \eqref{E408a} and Lemma \ref{L403} that
\begin{align*}
R_{abcn}=O(\ep_0^{\frac 1 2}f^{-\frac{1}{2}}R^{\frac 3 2}).
\end{align*}

In addition, from the computation in the proof of Proposition \ref{P201}, we have
\begin{align}
R_{ancn}=\frac{2R_{akcl}R_{kl}-R_{kc}R_{ka}-R_{ac}/2+\Delta R_{ac}-\na^2_{ac} R/2}{|\na f|^2}. \label{E465aa}
\end{align}

From direct calculations,
\begin{align}
2R_{akcl}R_{kl}=&2R_{abcd}R_{bd}+2R_{abcn}R_{bn}+2R_{ancn}R_{nn} \notag \\
=& 2R_{abcd}R_{bd}+O(f^{-1}R^3)+O(f^{-2}R^4)\notag \\
=& \frac{2R}{(n-1)(n-2)}(g_{ac}g_{bd}-g_{ad}g_{bc}) R_{bd}+O(\ep_0 R^2)+O(\ep_1 R^2)\notag \\
=& \frac{2R^2g_{ac}}{(n-1)(n-2)}- \frac{2RR_{ac}}{(n-1)(n-2)}- \frac{2RR_{nn}}{(n-1)(n-2)}+O(\ep_0 R^2) \notag \\
=& \frac{2R^2g_{ac}}{(n-1)^2}+O(\ep_0 R^2), \label{E465c}
\end{align}
where we have used \eqref{E401ab} and \eqref{E401ac}. In addition,
\begin{align}
R_{kc}R_{ka}=R_{bc}R_{ab}+R_{cn}R_{an}=\frac{R^2}{(n-1)^2}+O(\ep_0 R^2). \label{E465d}
\end{align}

Combining \eqref{E465c}, \eqref{E465d}, \eqref{E411b} and Lemma \ref{L406}, it follows from \eqref{E465aa} that
\begin{align}
R_{ancn}=\frac{1}{|\na f|^2} \lc \frac{R^2}{(n-1)^2}-\frac{R}{2(n-1)} \rc g_{ac}+O(\ep_0^{\frac 1 3}f^{-1}R^2). \label{E465e}
\end{align}

Since $R \ge \frac{n-1}{2}+\frac{\ep_2}{2}$ on $E^2$, we obtain from \eqref{E465e} that
\begin{align*}
R_{anan} \ge c_1 \ep_2 \frac{R^2}{f}
\end{align*}
for some constant $c_1=c_1(n)>0$. For any two-form $w$, we decompose it as
\begin{align*}
w=\tilde w+x_a e^a \wedge e^n,
\end{align*}
where $\tilde w$ is the two-form on $\Sigma$. If we set $|x|=\sqrt{\sum_{a=1}^{n-1}x_a^2}$, then we compute
\begin{align*}
Rm(w,w)=&Rm(\tilde w, \tilde w)+2x_a Rm(\tilde w,e^a \wedge e^n )+x_ax_b R_{anbn} \\
\ge & \frac{R}{(n-1)(n-2)} |\tilde w|^2-c_2 \ep_0 R|\tilde w|^2-c_2|x| |\tilde w| \ep_0^{\frac 1 2}f^{-\frac 1 2} R^{\frac 3 2}+x_ax_b R_{anbn} \\
\ge & \frac{R}{(n-1)(n-2)} |\tilde w|^2-c_2 \ep_0 R|\tilde w|^2-c_2|x| |\tilde w| \ep_0^{\frac 1 2}f^{-\frac 1 2} R^{\frac 3 2}+c_1|x|^2 \ep_2 f^{-1}R^2-c_3|x|^2\ep_0^{\frac 1 3}f^{-1}R^2
\end{align*}
for some constants $c_2$ and $c_3$ depending only on $n,A$ and $B$. Since $\ep_2=\ep_0^{\frac 1 4}$, it follows from the mean value inequality that
\begin{align*}
Rm(w,w) \ge \frac{c_1}{2} \ep_2 f^{-1}R^2 (|x|^2+|\tilde w|^2)= \frac{c_1}{2} \ep_2 f^{-1}R^2 |w|^2,
\end{align*}
if $\ep_0$ is sufficiently small.
\end{proof}

\section{Estimates on the cap region}

Throughout this section, we assume \eqref{E401}, \eqref{E460a}, \eqref{E460b} and $\ep_0$ is small enough such that Proposition \ref{prop:462} holds. 
We will fix $\ep_1$ and choose $\ep_1'$ to be a small parameter which may depend on $\ep_1$. In particular, we assume $\ep_1' \le \ep_1^{\frac 1 2}$ and hence $\ep_0=\ep_1^{\frac{1}{2}}$ which is also fixed. 

Define
\begin{align}
D \coloneqq \textrm{the connected component of $M \backslash \text{Int}\,\Sigma(t_0,s)$ containing $\Sigma(s)$.}
\label{eqn:OH05_6} 
\end{align}
From Proposition \ref{prop:460} there exists a point $q \in \Sigma(s)$ such that
\begin{align*}
R(q)=\ep_1 f(q)=\ep_1 s. 
\end{align*}

First we prove
\begin{prop} \label{prop:500}
Under the above assumptions, there exist positive constants $m_1,m_2$ depending only on $n$ such that
\begin{align}
\frac{0.9}{\ep_1^{-1}+m_2 L} \le \frac{R(y)}{s} \le \frac{1.1}{\ep_1^{-1}+m_1 L} \label{E500a}
\end{align}
for any $L \in [0,s/2]$ and $y \in \Sigma(s-L)$.
\end{prop}

\begin{proof}
We first prove $\Sigma(s/2,s) \subset E^2$. Otherwise, there exist $s_1 \in [s/2,s]$ and a point $x \in \Sigma(s_1)$ such that $\Sigma(s_1,s) \subset E^2$ and 
\begin{align}
R(x)=\frac{n-1}{2}+\ep_2. \label{E500b}
\end{align}
From \eqref{E417b} we have on $E$ that
\begin{align*}
tR_t \ge \lc \frac{2}{n-1}-c_1 \ep_0^{\frac{1}{3}}\rc R^2-R. 
\end{align*}
Consequently, solving the ODE yields that 
\begin{align}
R(q^t) \ge \frac{1}{\frac{2}{n-1}-c_1 \ep_0^{\frac{1}{3}}-\delta_1 t} \label{E500d}
\end{align}
for any $t \in [s_1,s]$, where $q^t=\phi^{t-s}(q) \in \Sigma(t)$ and $\delta_1$ is determined by
\begin{align}
\ep_1s=R(q) =\frac{1}{\frac{2}{n-1}-c_1 \ep_0^{\frac{1}{3}}-\delta_1 s}. \label{E500e}
\end{align}
Since $R(q) \gg 1$ by \eqref{Ex:400}, it follows from \eqref{E500d} and \eqref{E500e} that
\begin{align}
R(q^{s_1}) \ge \frac{n}{2}. \label{E500f}
\end{align}
However, \eqref{E500f} contradicts \eqref{E500b} by Lemma \ref{L404}.

Now it is clear that \eqref{E500d} holds for any $t \in [s/2,s]$. Therefore, it follows from \eqref{E500d} and \eqref{E500e} that for $L \in [0,s/2]$,
\begin{align*}
R(q^{s-L}) \ge \frac{1}{\frac{2}{n-1}-c_1 \ep_0^{\frac{1}{3}}-\delta_1 (s-L)}=\frac{1}{\frac{1}{\ep_1 s}+\delta_1s-\delta_1 (s-L)}=\frac{1}{\frac{1}{\ep_1 s}+\delta_1L}
\end{align*}
and hence
\begin{align}
\frac{R(q^{s-L})}{R(q)} \ge \frac{1}{1+\delta_1s \ep_1 L}. \label{E500g}
\end{align}
From \eqref{E500e} and the fact that $R(q) \gg 1$, we know that $\delta_1s \approx \frac{2}{n-1}$. Combining \eqref{E500g} and Lemma \ref{L404}, we conclude that
\begin{align*}
\min_{\Sigma(s-L)} R\ge \frac{0.9R(q)}{1+m_2 \ep_1 L} 
\end{align*}
for some constant $m_2=m_2(n) \approx \frac{2}{n-1}$. Therefore, 
\begin{align*}
\min_{\Sigma(s-L)} R\ge \frac{0.9s}{\ep_1^{-1}+m_2 L}. 
\end{align*}
Now the other inequality in \eqref{E500a} can be proved similarly by considering
\begin{align*}
tR_t \le \lc \frac{2}{n-1}+c_1 \ep_0^{\frac{1}{3}}\rc R^2-R,
\end{align*}
which follows from \eqref{E417b}.
\end{proof}

Next, we prove

\begin{prop} \label{prop:501}
Under the above assumptions, for any small $\ep_1>0$, there exist $\tau_1=\tau_1(\ep_1,n,\delta_0,A,B)>0$ and $L_1=L_1(\ep_1,n,\delta_0,A,B)>0$ such that if $\ep_1' \le \tau_1$, then
\begin{align*}
\emph{diam}_g D \le \frac{L_1}{\sqrt{s}} \quad \text{and} \quad \sup_{D} |f-s| \le L_1.
\end{align*}
In particular, $D$ is compact and $\partial D=\Sigma(s)$.
\end{prop}

\begin{proof}
We prove it by contradiction. Suppose for fixed $\ep_1$ the conclusion does not hold. Then there exists a sequence of Ricci shrinkers $(M^n_i,g_i,f_i)$ satisfying all the assumptions with $\ep'_{1,i} \to 0$ and 
\begin{align}
\textrm{diam}_{g_i} D_i \geq \frac{i^2}{\sqrt{s_i}}, \quad \textrm{or} \quad \sup_{D_i} |f_i-s_i| \ge i^2. \label{eqn:OH04_1}
\end{align}
From Proposition \ref{prop:460}, there exists a point $q_i \in \Sigma(s_i)$ satisfying
\begin{align*}
R_i(q_i)=\ep_1 f_i(q_i)=\ep_1 s_i. 
\end{align*}

Now we define $\tilde f_i:=f_i-s_i$ and $\tilde g_i:=s_i g_i$. From direct calculations, we have 
\begin{align}
R_{\tilde g_i}+|\na_{\tilde g_i} \tilde f_i|^2_{\tilde g_i}=\frac{R_i+|\na_{g_i} f_i|^2_{g_i}}{s_i}=\frac{f_i}{s_i}=1+\frac{\tilde f_i}{s_i} \label{E501b}
\end{align}
and
\begin{align}
\text{Hess}_{\tilde g_i} \tilde f_i+Rc_{\tilde g_i}=\text{Hess}_{g_i} f_i+Rc_i=\frac{g_i}{2}=\frac{\tilde g_i}{2s_i}.
\label{E501c}
\end{align}

For any $L>1$ and $y \in B_{\tilde g_i}(q_i,L)$, it follows from \eqref{E501b} that 
\begin{align}
|\tilde f_i(y)| \le 2L \quad \text{and} \quad R_{\tilde g_i}(y) \le 2
\label{E501d}
\end{align}
if $i$ is sufficiently large. Therefore, it follows from \eqref{E501d} and Theorem \ref{thm:nonc} that
\begin{align}
|B_{\tilde g_i}(q_i,1)|_{\mu_i} \ge v_0
\label{E501e}
\end{align}
for some $v_0=v_0(n,A)>0$, where $\mu_i:=e^{-\tilde f_i}dV_{\tilde g_i}$ is the weighted measure. Combining \eqref{E501b}-\eqref{E501e}, it follows from the main result \cite[Theorem $1.1$]{LLW21} and its variant \cite[Theorem $10.2$]{LLW21} that by taking a subsequence if necessary,
\begin{align}
(M_i, q_i, \tilde g_i, \tilde f_i) \longright{pointed-Gromov-Hausdorff} \left(X_{\infty}, q_{\infty}, d_{\infty}, f_{\infty} \right), \label{E502a}
\end{align}
where $(X_{\infty}, d_{\infty})$ is a length space, $f_{\infty}$ is a Lipschitz function on $(X_{\infty}, d_{\infty})$. The space $X_{\infty}$ has a natural regular-singular decomposition $X_{\infty}=\mathcal{R} \cup \mathcal{S}$ satisfying
\begin{itemize}
\item[(a).] The singular part $\mathcal{S}$ is a closed set of Minkowski codimension at least $4$. Namely, we have
\begin{align}
\dim_{\mathcal{M}} \mathcal{S} \leq n-4. \label{eqn:OH06_5}
\end{align}
\item[(b).] The regular part $\mathcal{R}$ is an open manifold with smooth metric $g_{\infty}$ satisfying Ricci steady soliton equation
\begin{align}
Rc_{\infty} + \text{Hess}_{g_{\infty}} f_{\infty}=0. \label{E502b}
\end{align}
\item[(c).] The convergence \eqref{E502a} can be improved to 
\begin{align}
(M_i, q_i, \tilde g_i, \tilde f_i) \longright{pointed-\hat{C}^{\infty}-Cheeger-Gromov} \left(X_{\infty}, q_{\infty}, g_{\infty}, f_{\infty} \right). \label{E502c}
\end{align}

\item[(d).] On the regular part $\mathcal R$,
\begin{align}
R_{\infty}+|\na f_{\infty}|^2=1. \label{E502d}
\end{align}
\end{itemize}

Note that \eqref{E502c} implies that away from $\mathcal{S}$, the convergence is smooth. 
This follows from the bootstrapping argument based on \eqref{E501b} and \eqref{E501c}, see \cite{LLW21}. Moreover, \eqref{E502b} and \eqref{E502d} follow from \eqref{E501c} and \eqref{E501b}, respectively, by taking the limit.

In addition, it follows from \cite[Theorem $4.10$]{HLW21} that the regular part $\mathcal{R}$ is geodesically convex. Therefore, the limit $(X_{\infty}, q_{\infty}, d_{\infty}, f_{\infty} )$ is a Ricci steady soliton conifold in the sense of Definition \ref{dfn:B001}.

By our assumptions, any point $y \in B_{\tilde g_i}(q_i,L) \cap E^2_i$ satisfies \eqref{E401} and \eqref{E500a}. Furthermore, $y$ is an evolving $\ep$-neck by Proposition \ref{prop:463}. Therefore, the corresponding limit set $E_{\infty}$ of $E^2_i$ under the convergence \eqref{E502c} is an end of $(X_{\infty},d_{\infty})$. On the other hand, it follows from Theorem \ref{thm:oneend} that $(X_{\infty},d_{\infty})$ has only one end. In other words, $D_{\infty}:=X_{\infty} \backslash E_{\infty}$ is a compact set, which fact in turn yields that $\displaystyle \text{diam}_{d_{\infty}} D_{\infty} < \infty$. Since $f_{\infty}$ is $1$-Lipschitz and $f_{\infty}(q_{\infty})=0$, it is clear that $\displaystyle \sup_{D_{\infty}}|f_{\infty}| <\infty$. 
Therefore, for some $L_1>0$ and sufficiently large $i$, we have
\begin{align*}
\text{diam}_{\tilde g_i} D_i \le L_1, \quad \sup_{D_i}|f_i-s_i| \le L_1, 
\end{align*}
which contradicts (\ref{eqn:OH04_1}). The proof is established by this contradiction. 
\end{proof}

We have the following definition, which measures how close the neighborhood of a point in a Ricci shrinker is to a Ricci steady soliton conifold.

\begin{defn}[Center of an $\ep$-steady soliton conifold]\label{def:soli}
Let $(M^n,g,f)$ be a Ricci shrinker and let $\bar x \in M$ such that $f(\bar x)=s$ . We say that $\bar x$ is the center of an $\ep$-steady soliton conifold if, after rescaling the metric by the factor $s$, 
the neighborhood $B_{g}(\bar{x},\ep^{-1} s^{-\frac 1 2}) $ is $\ep$-close in the Gromov-Hausdorff sense to a nontrivial $(X,q,d,f) \in \mathfrak S(n)$.
\end{defn}

From the proof of Proposition \ref{prop:501}, we immediately have

\begin{cor} \label{cor:501}
Under the above assumptions, for any $\ep,\ep_1>0$, there exists $\tau_2=\tau_2(\ep,\ep_1,n,\delta_0,A,B)>0$ such that if $\ep_1' \le \tau_2$, then any point in $D$ is the center of an $\ep$-steady soliton conifold. 
\end{cor}

Next, we estimate the scalar curvature on the cap region. Before that, we recall the following Sobolev inequality proved in \cite{LW20}.

\begin{thm}\emph{(\cite[Theorem 1]{LW20})}\label{thm:sobo}
For any Ricci shrinker $(M^n,g,f)$ with $\boldsymbol{\mu}(g) \ge -A$, there exists a $C=C(n,A)>0$ such that for any compactly supported locally Lipschitz function $u$, we have
\begin{align*} 
\lc \int u^{\frac{n}{n-2}} \, dV \rc^{\frac{n-2}{n}} \le C\int 4|\na u|^2+R u^2 \,dV.
\end{align*}
\end{thm}

Notice that the Sobolev inequality is scaling-invariant and also holds at different time slices of the Ricci flow associated with a Ricci shrinker. Now we have

\begin{prop} \label{prop:502}
Under the above assumptions, for any small $\ep_1>0$, there exist positive numbers $\tau_3=\tau_3(\ep_1,n,\delta_0,A,B)$ and $L_2=L_2(\ep_1,n,\delta_0,A,B)$ such that if $\ep_1' \le \tau_3$, 
then
\begin{align}
L_2^{-1} s \le R \le s+L_1 \label{E503a}
\end{align}
on $D$, where $L_1$ is the same constant in Proposition \ref{prop:501}.
\end{prop}

\begin{proof}
Without loss of generality, we assume $\tau_3 \le \tau_1$ and the conclusion of Proposition \ref{prop:501} holds. In particular, we have
\begin{align}
R(y) \le f(y) \le s+L_1, \quad \forall \; y \in D. \label{E503b}
\end{align}
For the other direction, we assume the inequality does not hold. Then there exists a sequence of Ricci shrinkers $(M^n_i,q_i, g_i,f_i)$ satisfying all the assumptions with $\ep'_{1,i} \to 0$ such that 
\begin{align}
R_i(x_i) < i^{-2} s_i \label{E503c}
\end{align}
for some $x_i \in D_i$. For $t \le 0$, we define 
\begin{align*}
\tilde g_i(t) \coloneqq s_i g_i(s_i^{-1}t), \quad \tilde F_i(t) \coloneqq F_i(s_i^{-1}t)-s_i, \quad \tilde \square_i \coloneqq \partial_t-\Delta_{\tilde g_i(t)}, 
\end{align*}
where $F_i(t)=(1-t)f_i(t)$. Direct calculation shows that 
\begin{align*}
|\na_{\tilde g_i(t)} \tilde F_i(t)|^2_{\tilde g_i(t)} = & s_i^{-1} |\na_{ g_i(s_i^{-1} t)} F_i(s_i^{-1} t)|^2_{ g_i(s_i^{-1} t)}= \frac{F_i(s_i^{-1} t)-(1-s_i^{-1}t)^2 R_{g_i(s_i^{-1} t)} }{s_i}.
\end{align*}
Thus we obtain
\begin{align*}
|\na_{\tilde g_i(t)} \tilde F_i(t)|^2_{\tilde g_i(t)}+(1-s_i^{-1}t)^2R_{\tilde g_i(t)}=1+ \frac{\tilde F_i(t)}{s_i}.
\end{align*}

Moreover, we compute
\begin{align}
\partial_t \tilde F_i(t)=s_i^{-1} \partial_t F_i(s_i^{-1}t)=-s_i^{-1}(1-s_i^{-1}t) R_{g_i(s_i^{-1} t)}=-(1-s_i^{-1}t) R_{\tilde g_i(t)} \label{E503e}
\end{align}
and 
\begin{align}
\Delta_{\tilde g_i(t)} \tilde F_i(t)=s_i^{-1}\Delta_{g_i(s_i^{-1} t)} F_i(s_i^{-1}t)=s_i^{-1} \lc \frac{n}{2}-(1-s_i^{-1}t) R_{g_i(s_i^{-1} t)} \rc=\frac{n}{2s_i}-(1-s_i^{-1}t)R_{\tilde g_i(t)}. \label{E503f}
\end{align}
Combining \eqref{E503e} and \eqref{E503f}, we obtain
\begin{align}
\tilde \square_i \tilde F_i(t)=(\partial_t-\Delta_{\tilde g_i(t)}) \tilde F_i(t)=-\frac{n}{2s_i}. \label{E503g}
\end{align}

Now we define the following parabolic balls
\begin{align*}
P^1_i& \coloneqq \{(y,t) \,\mid |\tilde F_i(y,t)| \le 2L_1, \, -2L_1 \le t \le 0\}, \\
P^2_i& \coloneqq \{(y,t) \,\mid |\tilde F_i(y,t)| \le 4L_1, \, -4L_1 \le t \le 0\} . 
\end{align*}

In the following proof, all positive constants $C_i$ in the following proof depend only on $n,\delta_0,A,B$ and $\ep_1$ and the corresponding inequalities hold for sufficiently large $i$. It follows from \eqref{E500a} and \eqref{E503b} that 
\begin{align}
0 \le R_{\tilde g_i(t)} \le C_1 \label{E504a}
\end{align}
on $ P^2_i$, since by the definition we know 
\begin{align*}
R_{\tilde g_i(t)}(z)=s_i^{-1} R_{g_i(s_i^{-1}t)}(z)=s_i^{-1}(1-s_i^{-1}t)^{-1} R_{g_i(0)}(\psi_i^{s_i^{-1}t}(z))
\end{align*}
for any point $z$. 
Since $\tilde F_i(q_i,0)=0$, it follow from \eqref{E503e} and \eqref{E504a} that
\begin{align*}
0 \le \tilde F_i(q_i,t) \le C_2
\end{align*}
for $t \in [-4L_1,0]$.

\textbf{Claim 1.} Suppose $u$ is a nonnegative function such that 
\begin{align*}
\tilde \square_i u \le 0, \quad \textrm{on} \; P^2_i. 
\end{align*}
Then we have
\begin{align}\label{E504c}
\max_{P^1_i} u^2 \le C_3 \iint_{P^2_i} u^2 \,dV_{\tilde g_i(t)} dt
\end{align}
for some $C_3>0$.

\emph{Proof of Claim 1:} The proof follows verbatim as \cite[Lemma $9.7$]{LLW21} by using the Moser iteration. Notice that the Sobolev inequality and the control of $R_{\tilde g_i(t)}$ on $P^2_i$ are guaranteed by Theorem \ref{thm:sobo} and \eqref{E504a} respectively. 
Moreover, the cutoff functions $\eta_k$ in the proof of \cite[Lemma $9.7$]{LLW21} can be defined similarly by using $\tilde F_i$ and $\eta_k$ can be estimated similarly by using \eqref{E503g}. We omit the details.

Next we define for any $\tau>0$,
\begin{align*}
P_i(\tau)=P^2_i\cap \{(z,t) \,\mid R_{\tilde g_i(t)}(z) <\tau \}.
\end{align*}

\textbf{Claim 2.} There exists a constant $C_4>0$ such that 
\begin{align} \label{E504d}
|P_i( 2i^{-2})| \ge C_4. 
\end{align}
Here the volume is with respect to $dV_{\tilde g_i(t)} dt$.

\emph{Proof of Claim 2:} It follows from \eqref{E503c} that $(x_i,0) \in P^1_i$. Then the proof follows exactly as that of \cite[Lemma $9.8$]{LLW21} by using \eqref{E504c}.

It is clear from \eqref{E504d} that for any $i$ sufficiently large, there exists $t_i \in [-2L_1,0]$ such that 
\begin{align*}
|\Sigma_i|_{dV_{\tilde g_i(t_i)}} \ge C_5,
\end{align*}
where $\Sigma_i=M_i \times \{t_i\} \cap P_i( 2i^{-2})$ and $C_5=\frac{C_4}{2L_1}$. Now, we define 
\begin{align*}
\Omega_i:=\psi_i^{s_i^{-1}t_i}\lc \text{pr}(\Sigma_i) \rc
\end{align*}
where $\text{pr}$ is the projection onto $M_i$. Then we compute
\begin{align} \label{E504e}
|\Omega_i|_{dV_{\tilde g_i(0)}}\ge \frac{1}{2}(1-s_i^{-1}t_i)^{\frac n 2} |\Omega_i|_{dV_{\tilde g_i(0)}}= \frac{1}{2} |\Sigma_i|_{dV_{\tilde g_i(t_i)}} \ge C_6,
\end{align}
where $C_6=C_5/2$. On the other hand, for any $(z,t_i) \in \Sigma_i$, we have
\begin{align*}
\tilde F_i(z,t_i)=F_i(z,t)-s_i=(1-s_i^{-1}t_i) f_i(z',0)-s_i
\end{align*}
where $z'=\psi_i^{s_i^{-1}t_i}(z)$. Therefore, it follows from the definition of $P^2_i$ and \eqref{E503a} that
\begin{align} \label{E504f}
|\tilde f_i(z')|=|f_i(z',0)-s_i| \le 4L_1+1.
\end{align}

As in the proof of Proposition \ref{prop:501}, by taking a subsequence if necessary, we have
\begin{align*}
(M_i, q_i, \tilde g_i, \tilde f_i) \longright{pointed-\hat{C}^{\infty}-Cheeger-Gromov} \left(M_{\infty}, q_{\infty}, g_{\infty}, f_{\infty} \right), 
\end{align*}
where $(M_{\infty}, q_{\infty}, g_{\infty}, f_{\infty})$ is a steady soliton conifold. Notice that $\lim_{z \to \infty} f_{\infty}(z) =-\infty$ by our assumptions. Combining \eqref{E504e} and \eqref{E504f}, we conclude that there exists a point $z \in \mathcal R$ such that
\begin{align*}
R_{g_{\infty}}(z)=0.
\end{align*}

By applying the strong maximum principle for $\Delta_{f_{\infty}} R_{\infty}=-2|Rc_{\infty}|^2 \le 0$, we conclude that $R_{\infty} \equiv 0$ on $\mathcal R$. However, this contradicts the fact that $R_{\infty}(q_{\infty})=\ep_1$.

In sum, the proof is complete.
\end{proof}

As observed above, the cap region is modeled on a steady soliton conifold if $\ep_1'$ is sufficiently small. In general, it is difficult to classify all steady soliton conifolds even though there are no singularities. However, under some special conditions, we can classify all such steady soliton conifolds. For the definition of PIC2, see \cite[Section $7.5$]{Bre10}.

\begin{thm} \label{thm:Bryant}
Let $(X,p,d,f) \in \mathfrak S(n)$ be non-compact such that $(\mathcal R,g)$ has \emph{PIC2}. Suppose there exists a compact set $C$ satisfying
\begin{itemize}
\item[(a).] The singular part $\mathcal S \subset C$.
\item[(b).] For any $x \in X\backslash C$, the scalar curvature satisfies $\dfrac{c_1}{d(p,x)} \le R(x) \le \dfrac{c_2}{d(p,x)}$, where $c_1,c_2$ are two positive constants.
\item[(c).] For any sequence $q_i \to \infty$, $q_i$ is the center of an evolving $\ep_i$-neck with $\ep_i \to 0$. 
\end{itemize}
Then $(X,d)$ is isometric to the Bryant soliton up to scaling.
\end{thm}

\begin{proof}
Since $(R,g)$ has PIC2 on $\mathcal R$, it implies that $sec \ge 0$ on $\mathcal R$. Combined with the identity $R+|\na f|^2=1$ as $(X,d)$ is nontrivial, we conclude that $|Rm|$ is uniformly bounded on $\mathcal R$. 
For any $x \in \mathcal R$, we define $\lambda=\lambda(x)$ to be the largest constant such that $S:=Rm-\lambda I$ has weakly PIC2. By our assumption, $\lambda>0$ on $\mathcal R$. 

\textbf{Claim 1.} $\lambda$ satisfies the following inequality on $\mathcal R$.
\begin{align*}
\Delta_f \lambda \le 0
\end{align*}

\emph{Proof of the Claim 1:} By our definition, $S$ is contained in the boundary of PIC2 cone. For fixed $x \in \mathcal R$, there exists a four-frame $\{e_1,e_2,e_3,e_4\} \in T_x M$ and constants $u,v \in [0,1]$ such that
\begin{align*}
S_{1313}+u^2 S_{1414}+v^2 S_{2323}+u^2v^2 S_{2424}-2uv S_{1234}=0.
\end{align*}

Therefore, it follows from \cite[Proposition $7.21$]{Bre10} that
\begin{align}\label{E505ab}
Q(S)_{1313}+u^2 Q(S)_{1414}+v^2 Q(S)_{2323}+u^2v^2 Q(S)_{2424}-2uv Q(S)_{1234}\ge 0,
\end{align}
where $Q$ is a quadratic term of the curvature tensor, see \cite{Ham86} for its precise definition.

In addition, we have
\begin{align}\label{E505ac}
Q(S)=Q(Rm)-2\lambda Rc \owedge \text{id}+2(n-1)\lambda^2I,
\end{align}
where $\owedge$ is the Kulkarni-Nomizu product. Therefore, it follows from \eqref{E505ab} and \eqref{E505ac} that
\begin{align*}
& Q(Rm)_{1313}+u^2 Q(Rm)_{1414}+v^2 Q(Rm)_{2323}+u^2v^2 Q(Rm)_{2424}-2uv Q(Rm)_{1234} \\
\ge & 2\lambda \lc R_{11}+R_{33}+u^2(R_{11}+R_{44})+v^2(R_{22}+R_{33})+u^2v^2(R_{22}+R_{44}) \rc \\
& -2(n-1)\lambda^2 (1+u^2+v^2+u^2v^2) \ge 0,
\end{align*}
where for the last inequality, we have used $R_{ii} \ge (n-1)\lambda$ since $S$ has weakly PIC2. Therefore, we conclude that
\begin{align*}
\Delta_f \lambda \le& \frac{\Delta_f(Rm)_{1313}+u^2 \Delta_f(Rm)_{1414}+v^2 \Delta_f(Rm)_{2323}+u^2v^2 \Delta_f(Rm)_{2424}-2uv \Delta_f(Rm)_{1234}}{(1+u^2)(1+v^2)} \\
=& -\frac{Q(Rm)_{1313}+u^2 Q(Rm)_{1414}+v^2 Q(Rm)_{2323}+u^2v^2 Q(Rm)_{2424}-2uv Q(Rm)_{1234} }{(1+u^2)(1+v^2)} \le 0
\end{align*}
at $x$, where we have used the equation $\Delta_f Rm=-Q(Rm)$.

Hence, we have on $\mathcal R$,
\begin{align*}
\Delta_f \lambda \le 0.
\end{align*}

Since $\lambda$ is positive on $\mathcal R$, it follows from Proposition \ref{prop:B006} that
\begin{align} \label{E505a}
\inf_{K \cap \mathcal R} \lambda>0
\end{align}
for any compact set $K$. Since the regular part $\mathcal R$ is geodesically convex and $(X,d)$ is the completion of $(\mathcal R, g)$, it follows from \cite{Pet16} that $(X,d)$ is an Alexandrov space with $sec \ge 0$.

\textbf{Claim 2.} $f$ is a concave function on $(X,d)$ with a unique critical point.

\emph{Proof of the Claim 2:} Since $\text{Hess}\, f=-Rc <0$, it is clear that $f$ is concave on $(\mathcal R,g)$ in the sense that for any geodesic $\gamma(t)$ contained in $\mathcal R$, $f(\gamma(t))$ is concave. Now we fix a geodesic $\gamma(t)_{t \in [0,L]}$ in $X$ 
such that $\gamma(0)=x$ and $\gamma(L)=y$. For any small number $\ep>0$, we set $x'=\gamma(\ep)$ and $y'=\gamma(L-\ep)$. Since $\mathcal R$ is dense, we can take two sequences $x_n \to x'$ and $y_n \to y'$ such that $x_n$ and $y_n$ are contained in $\mathcal R$. From the geodesical convexity of $\mathcal R$, there exists a geodesic $\gamma_n \subset \mathcal R$ connecting $x_n$ and $y_n$. Since the geodesic in an Alexandrov space is not branching, $\gamma_n$ converges to $\gamma(t)_{t \in [\ep,L-\ep]}$. 
By taking the limit, we immediately conclude that $f(\gamma(t))$ is concave for $t \in [\ep,L-\ep] $. Since $\ep>0$ is arbitrary, $f(\gamma(t))$ is concave for any $t\in [0,L]$. Therefore, $f$ is concave in the sense of \cite[Definition $1.1$]{Pet07}.

We fix a point $x \in X$. Starting from $x$, we denote the gradient flow of $f$ by $\alpha(t)$ for $t \ge 0$, see \cite[Section 2]{Pet07}. If $\alpha(t)_{t \in [0 ,\infty)}$ is contained in a compact set $K$ of $X$, then there exists a sequence $t_i \to \infty$ such that $|\na f|(\alpha(t_i)) \to 0$, since $f$ is bounded on $K$. By taking a subsequence, we may assume $\alpha(t_i)$ converges to a point $y \in K$. By the lower semicontinuity of $|\na f|$ \cite[Corollary $1.3.5$]{Pet07}, $y$ is a critical point of $f$. If $\alpha(t)_{t \in [0 ,\infty)}$ is not contained in any compact set, there exists a sequence $s_i$ such that $\alpha(s_i) \to \infty$. By our assumption (b), $R(\alpha(s_i)) \to 0$. However, this contradicts the fact that $R$ is increasing along $\alpha(t)$ as long as $\alpha(t)$ is outside $C$ since $\la \na R, \na f \ra=2Rc(\na f, \na f) \ge 0$. In any case, $f$ has a critical point. Without loss of generality, we set $p$ to be a critical point of $f$.

It follows from \eqref{E505a} that for any constant $c$, there exists a constant $\delta>0$ such that $(X_c,d)$ is an Alexandrov space with $sec \ge \delta$, where $X_c=\{x\in C\mid f(x) \ge c\}$. By this reason, $f$ is a strictly concave function on $X$. Therefore, it is a standard fact that $p$ is the unique critical point which is also the maximum point.

\textbf{Claim 3.} $X\backslash \{p\} \subset \mathcal R$.

\emph{Proof of the Claim 3:} Otherwise, we set $q \in \mathcal S$ and $q \ne p$ such that $f(q)=t$ is minimal. Notice that by our assumption (b), the choice of $t$ is possible. By our assumption, $|\na f|(q)>0$ since $p$ is the only critical point of $f$. From the lower semicontinuity of $|\na f|$ \cite[Corollary $1.3.5$]{Pet07}, we conclude that there exists a small $\delta>0$ such that 
\begin{align}
|\na f| \ge \delta \label{E505b}
\end{align}
on $B(q,\delta)$. Now we take any positive sequence $r_i \to 0$ and define $d_i=r_i^{-1}d$, $g_i=r_i^{-2}g$ and $f_i=r_i^{-1}(f-f(q))$. From the definition, it is clear that on $\mathcal R$
\begin{align} \label{E505c}
\text{Hess}_{g_i} f_i=-r_i^{-1}Rc_i.
\end{align}
Now we have
\begin{align*}
(X, q, d_i, f_i) \longright{pointed-Gromov-Hausdorff} \left( K_q, 0, d_{\infty}, f_{\infty} \right),
\end{align*}
where $(K_q,d_{\infty})$ is the tangent cone at $q$ and $f_{\infty}$ is a locally Lipschitz function. In addition, it follows from \eqref{E505c} and the same arguments in Claim 1 that $f_{\infty}$ is both concave and convex on $K_q$.

For any $s<t$, it is clear from our choice of $t$ that $f^{-1}(s)\cap B(q,\delta) \subset \mathcal R$ on which the $|Rm|$ is uniformly bounded. Moreover, for any $x \in f^{-1}(s)\cap B(q,\delta)$, it follows from \eqref{E505b} that the second fundamental form at $x$, which is $\frac{\text{Hess}\, f}{|\na f|}(x)$, is uniformly bounded. Therefore, we conclude that 
\begin{align*}
\left(f_i^{-1}((-\infty,0)), q, g_i,f_i \right) \longright{C^{\infty}-Cheeger-Gromov} \lc f_{\infty}^{-1}((-\infty,0)), 0,d_{\infty}, f_{\infty} \rc.
\end{align*}
Here, for the smooth convergence, we have used $\Delta_f Rm=-Q(Rm)$ to obtain the higher order estimates of $Rm$ on $f^{-1}(s)\cap B(q,\delta)$. Moreover, we have
\begin{align*} 
( f_{\infty}^{-1}((-\infty,0)), d_{\infty})=(\R^n \cap \{x_n<0\},g_E).
\end{align*}
In particular, it follows from the smooth convergence of $f_i$ to $f_{\infty}$ on $f_{\infty}^{-1}((-\infty,0))$ and \eqref{E505b} that $f_{\infty}$ is not a constant. Since $f_{\infty}$ is both convex and concave, it is clear (see \cite[Lemma $2.1.4$]{Pet07}) that $(K_q,d_{\infty})$ is isometric to $(Y \times \R, d' \times g_E)$, where $Y$ is a metric cone isometric to any level set of $f_{\infty}$. Therefore, it is immediate from \eqref{E505c} that $Y$ is smooth and hence isometric to $\R^{n-1}$. From this, we conclude that $(K_q,d_{\infty})$ is isometric to $(\R^n,g_E)$. However, this implies that $q \in \mathcal R$, which contradicts our assumption.

\textbf{Claim 4.} $p \in \mathcal R$.

\emph{Proof of the Claim 4:} From Claim 2, any point in $X \backslash \{p\}$ is regular with uniformly bounded $|Rm|$. Therefore, for any sequence $r_i \to 0$, the convergence
\begin{align*}
(X, p,r_i^{-1}d) \longright{pointed-Gromov-Hausdorff} \left(K_p, 0, d_{\infty} \right)
\end{align*}
is smooth away from the vertex. Moreover, 
\begin{align}\label{E505e}
(K_p,d)=(\R^n/\Gamma,g_E)
\end{align}
where $\Gamma \le O(n)$ is a finite subgroup acting freely on $S^{n-1}$. If we set the nontrivial level set of $f$ by $\Sigma$, then near $p$ we have two foliations by $\Sigma$ and $S^{n-1}/\Gamma$ respectively. 
Then it is easy to see that $\Sigma$ and $S^{n-1}/\Gamma$ are $h$-cobordant. 
Indeed, by our assumption, there exist a region in $M$ homeomorphic to $\Sigma \times [0,4]$ and an embedding of $N \times [1,3]$
 into the interior of $\Sigma \times [1,4]$ such that $N \times \{k\}$ separates $\Sigma \times \{k\}$ and $\Sigma \times \{k+1\}$ for $1 \le k \le 3$. By an easy topological argument, the compact manifold bounded 
 by $N \times \{1\}$ and $\Sigma \times \{3\}$ is a cobordism such that the boundary inclusions are homotopy equivalences. In particular, $\Sigma$ and $N$ are homotopic and hence $\Gamma$ must be trivial. 
 From \eqref{E505e}, we have
\begin{align*}
(K_p,d)=(\R^n,g_E)
\end{align*}
and hence $p \in \mathcal R$.

In sum, we have proved that $\mathcal S=\emptyset$ and hence $(X,p,d,f)$ is a smooth steady soliton. The rest of the proof follows from the classification theorem of Brendle \cite[Theorem $1.2$]{Bre14}.
\end{proof}

Combining the results in Section 3 and Section 4, we are able to prove Theorem~\ref{T101} now. 
Let us first recall and fix notations. 

For any constants $n,\delta_0,A, B$ and $\ep$, we define the following constants:
\begin{itemize}
\item Let $\bar \sigma \coloneqq \min\{\sigma_1(n,B), \sigma_2(n,\delta_0,A,B), \sigma_3(n,\delta_0,A,B), \sigma_4(n,\delta_0,A,B), \sigma_5(n,A,B) \}$, where $\sigma_1$, $\sigma_2$, $\sigma_3$, $\sigma_4$ and $\sigma_5$ are functions in Proposition \ref{P401}, Proposition \ref{prop:405a}, Proposition \ref{prop:406a}, Proposition \ref{prop:462} and Proposition \ref{prop:464}, respectively.

\item Let $\bar \eta \coloneqq \min\{ \eta_3(\ep,n,\delta_0,A,B), \eta_4(\ep,n,\delta_0,A,B), \eta_5(\ep,n,\delta_0,A,B) \}$, where $\eta_3$, $\eta_4$ and $\eta_5$ are functions in Proposition \ref{prop:405b}, Proposition \ref{prop:405c} and Proposition \ref{prop:463}, respectively.
\item Let $\sigma \coloneqq \min\{\bar \sigma^2,\bar \eta^2 \}$.
\item Let $\eta \coloneqq \min\{ \sigma^{\frac 1 2}, \bar \eta, \tau_1(\sigma,n,\delta_0,A,B),\tau_2(\sigma,n,\delta_0,A,B), \tau_3(\sigma,n,\delta_0,A,B) \}$, where $\tau_1$, $\tau_2$ and $\tau_3$ are functions in Proposition \ref{prop:501}, Corollary \ref{cor:501} and Proposition \ref{prop:502}, respectively.
\item Let $L \coloneqq \max \{ L_1(\sigma,n,\delta_0,A,B), L_2(\sigma,n,\delta_0,A,B) \}$, where $L_1$ and $L_2$ are functions in Proposition \ref{prop:501} and Proposition \ref{prop:502}, respectively.
\end{itemize}

We close this section with the proof of Theorem~\ref{T101}. 

\begin{proof}[Proof of Theorem \ref{T101}]
From our assumptions, the condition \eqref{E001} in particular holds for $\ep_1=\sigma$ and $\ep_1'=\eta$. With our choice of $\sigma, \eta$ and $L$, it is clear from the results of Section $3$ and Section $4$ that the conclusion holds. 
\end{proof}

\section{Proof of the main theorem}

The purpose of this section is to prove Theorem~\ref{T100}. As we discussed in the introduction, the key ingredient is to show each point in the Ricci shrinker is either in a cylinder-like neighborhood or in a Bryant-soliton-like neighborhood, 
then we apply the symmetry improvement argument. 

We recall the following definition of the curvature cone from \cite{BHS11}. Here $\mathscr C_{B}(\R^n)$ denotes the vector space of algebraic curvature tensor.
\begin{defn} \label{def:cone}
A cone $\hat C$ in $\mathscr C_{B}(\R^n)$ has the property $(*)$ if it satisfies the following conditions:
\begin{itemize}
\item $\hat C$ is closed, convex, $O(n)$-invariant, and of full-dimension.
\item $\hat C$ is transversally invariant under Hamilton's ODE: $\dfrac{d Rm}{dt}=Q(Rm)$.
\item Every algebraic curvature tensor $Rm \in \hat C \backslash \{0\}$ has positive scalar curvature.
\item The identity $I$ lies in the interior of $\hat C$.
\end{itemize}
\end{defn}

For later applications, we mainly consider the following example from \cite[Definition $2.2$]{Bre18}.
\begin{exmp} \label{exam:cone}
For any parameter $\sigma \in (0,2)$ and $\theta \ge 0$, we define the cone $C_{\sigma,\theta}$ as the set
\begin{align*}
\left\{ Rm=S+H\owedge \emph{id} \mid S \in \emph{PIC2}\,, Rc_0(S)=0,\, \emph{tr}(H)\emph{id}-(n-2\sigma)H\ge 0,\, \emph{tr}(H)-\theta\emph{scal}(S) \ge 0 \right\}, 
\end{align*}
where $H$ is a symmetric bilinear form, $\owedge$ is the Kulkarni-Nomizu product, $Rc_0$ is the traceless Ricci, and \emph{PIC2} denotes the cone consists of curvature tensor satisfying the weakly PIC2 condition. It follows from \emph{\cite[Theorem $2.5$]{Bre18}} and the definition that $C_{\sigma,\theta}$ has the property $(*)$ if $\sigma \in (1,2)$ and $\theta \in (0,\bar \theta(n))$ for a constant $\bar \theta$ depending only on $n$.

We denote $C_{1+\theta,\theta}$ by $\hat C(\theta)$ for $\theta \in (0, \min\{1/2,\bar \theta\})$. It is clear from the definition that
\begin{align} \label{ex:001}
Rc \ge -C(n) \theta R
\end{align}
for any $Rm \in \hat C(\theta)$. Furthermore, it is easy to see that the curvature tensor of $S^{n-1} \times \R$ lies in $\hat C(\theta)$. Indeed, we can write the curvature tensor as $Rm=H\owedge \emph{id}$, where $H_{ij}=\delta_{ij} $ for $i,j\ge 2$ and $H_{1i}=-\delta_{1i}$ and hence it is easy to check from the definition.
\end{exmp}

If a cone $\hat C$ has the property $(*)$, we have the following result (see, e.g., \cite[Lemma $4.2$]{Na19}).

\begin{lem}\label{lem:601}
Let $\hat C$ in $\mathscr C_{B}(\R^n)$ be a cone with property $(*)$. Then there exists two positive constants $\tau$ and $K$ depending only on the cone such that
\begin{align*}
Q(Rm)-\tau R^2 I \in T_{Rm}\hat C \quad \text{and} \quad |Rm| \le KR
\end{align*}
for any $Rm \in \hat C$.
\end{lem}

We prove the following theorem, which implies that locally the Ricci flow associated with a Ricci shrinker almost preserves the cone $\hat C$. Recall that on the Ricci flow associated with a Ricci shrinker, $F(x,t):=(1-t)f(x,t)$ for any $t<1$. Moreover, the constant $\tau$ and $K$ are from Lemma \ref{lem:601} which depends only on the cone $\hat C$.

\begin{thm} \label{thm:T601}
Let $\hat C$ in $\mathscr C_{B}(\R^n)$ be a cone with $(*)$. There exist constants $\delta=\delta(n,\tau,K)>0$ and $C=C(n)>0$ satisfying the following property. For any Ricci flow $(M^n,g(t))_{t<1}$ associated with a Ricci shrinker, we define $\lambda(x,t)$ be the minimal number such that $Rm(x,t)+\lambda(x,t)I \in \hat C$, for any $(x,t) \in M \times (-\infty,1)$. For any $r>0,\ep>0$, if $\ep+r^{-2} \le \delta R$ on $P_2:=\{(x,t) \mid F(x,t) \le 2r,\, 0\le t <1\}$ and $\lambda \le \ep$ on $\{(x,0) \mid F(x,0)\le 2r\}$, then
\begin{align*}
\lambda \le \ep+Cr^{-2}
\end{align*}
on $P_1:=\{(x,t)\,\mid F(x,t) \le r, 0\le t <1\}$.
\end{thm}

\begin{proof}
We fix a smooth function $\eta$ on $[0,\infty)$ such that $0 \le \eta \le 1$, $\eta=1$ on $[0,1]$ and $\eta=0$ on $[2,\infty)$. Furthermore, $|\eta'|^2 \le C \eta^{\frac 3 2}$ and $|\eta''| \le C \eta^{\frac 1 2}$ 
for a universal constant $C$. If we set $\phi=\phi_r=\eta \lc \frac{F}{r} \rc$, it follows from the identities of $F$ (see \cite[Section $3$]{LW20} for details) that
\begin{align}\label{E601b}
\quad |\na \phi|^2 \le C\phi^{\frac 3 2} r^{-1} \quad \text{and} \quad |\square \phi| \le C \phi^{\frac 1 2} r^{-1}
\end{align}
for some constant $C$ depending only on $n$.

From our definition, the curvature operator 
\begin{align}
S \coloneqq Rm+\lambda I \label{eqn:OH07_1}
\end{align}
lies on the boundary of $\hat{C}$. Therefore, it follows from Lemma \ref{lem:601} that
\begin{align}\label{E601c}
Q(S)-\tau \text{scal}^2(S)I \in T_S \hat C,
\end{align}
where $\text{scal}(S)$ is the scalar curvature of $S$. 
Direct calculation yields that 
\begin{align}
&Q(S)=Q(Rm)+2\lambda Rc \owedge \text{id}+2(n-1)\lambda^2I, \label{E601d}\\
&\text{scal}(S)=R+n(n-1) \lambda. \label{E601e}
\end{align}
Combining \eqref{E601c}, \eqref{E601d} and \eqref{E601e}, we have
\begin{align}\label{E601f}
Q(Rm)+2\lambda Rc \owedge \text{id}+2(n-1)\lambda^2I-\tau (R+n(n-1) \lambda)^2 I \in T_S \hat C.
\end{align}
Moreover, it follows from Lemma \ref{lem:601} again that
\begin{align*}
|S| \le K \text{scal}(S)=K\lc R+n(n-1) \lambda \rc, 
\end{align*}
which implies that
\begin{align}
\label{E601g}
|Rm| \le C_1(R+|\lambda|)
\end{align}
for some $C_1=C_1(n,K)>0$. 
In light of (\ref{eqn:OH07_1}), direct calculation(with Uhlenbeck's trick) yields that
\begin{align}
\label{E602a}
\square S=(\partial_t -\Delta) S=Q(Rm)+(\square \lambda) I= \left\{ Q(Rm)-\lambda^2 I \right\} +(\square \lambda+\lambda^2) I.
\end{align}

Note that $Q(Rm)-\lambda^2 I \in T_S \hat C$ if $0 \leq \lambda \leq \delta_1 R$ for some small constant $\delta_1=\delta_1(n,K,\tau)>0$. 
Actually, under this condition, it follows from \eqref{E601g} that
\begin{align*}
2\lambda Rc \owedge \text{id}+2(n-1)\lambda^2I \le C_2 \lambda (\lambda+R) I
\end{align*}
for some $C_2=C_2(n,K)>0$. Combining the above inequality with \eqref{E601f}, we have
\begin{align*}
Q(Rm)+ \lc C_2 \lambda (\lambda+R)-\tau (R+n(n-1) \lambda)^2\rc I \in T_S \hat C.
\end{align*}
If $\delta_1$ is sufficiently small, then $C_2 \lambda (\lambda+R)-\tau (R+n(n-1) \lambda)^2 \le -\lambda^2$. 
Consequently, we have proved that
\begin{align}
Q(Rm)-\lambda^2 I \in T_S \hat C \label{eqn:OH07_2}
\end{align}
if $0 \leq \lambda \leq \delta_1 R$. 

On the other hand, the choice of $\lambda$ in (\ref{eqn:OH07_1}) implies that $\square S$ is not in the interior of $T_S \hat C$, see \cite[Claim 3 on Page 11]{CL20} for details. 
Therefore, it follows from (\ref{eqn:OH07_2}) and (\ref{E602a}) that
\begin{align*}
\square \lambda \le -\lambda^2
\end{align*}
wherever $0 \le \lambda \le \delta_1 R$. Thus we have
\begin{align*}
\square (\phi \lambda)=&(\square \phi) \lambda+\phi(\square \lambda)-2\la \na \phi,\na \lambda \ra 
\le (\square \phi) \lambda-\phi \lambda^2-2\la \na \log \phi,\na( \phi \lambda) \ra+2\frac{|\na \phi|^2}{\phi}\lambda. 
\end{align*}
Plugging \eqref{E601b} into the above inequality, we obtain 
\begin{align}
\square (\phi \lambda) \le& -\phi \lambda^2+Cr^{-1} \sqrt \phi \lambda-2\la \na \log \phi,\na( \phi \lambda) \ra 
\le Cr^{-2}-2\la \na \log \phi,\na( \phi \lambda) \ra \label{E602b}
\end{align}
for some $C=C(n)$. Applying the maximum principle to \eqref{E602b}, we obtain
\begin{align*}
\phi \lambda \le \ep+Cr^{-2}
\end{align*}
as long as $\ep+Cr^{-2} \le \delta_1 R$. If we set $\delta=\frac{\delta_1}{C}$ in our assumption, then the conclusion follows.
\end{proof}

\begin{rem}
In the proof of Theorem \ref{thm:T601}, a key fact we use is that $\hat C$ is transversally invariant under Hamilton's ODE. This restriction excludes many well-known cones like $\{Rm \ge 0\}$, PIC1, and PIC2. Another critical fact we use is the existence of a good cutoff function from \eqref{E601b}, which is not available for general non-compact Ricci flows.
\end{rem}

The following lemma follows directly from Definition \ref{def:neck_2} and Definition \ref{neck_symmetry}.

\begin{lem}\label{L601}
For any $n$ and $\ep>0$, there exists a constant $\tau_4=\tau_4(\ep,n)>0$ satisfying the following property.

Suppose $(M^n, g(t))$ is an $n$-dimensional Ricci flow solution and $(\bar x, \bar t)$ is a spacetime point. Then $(\bar x, \bar t)$ is $\ep$-symmetric in the sense of Definition \ref{neck_symmetry} if $(\bar x, \bar t)$ is the center of an evolving $\tau_4$-neck with respect to $S^{n-1 }\times \R$.
\end{lem}

\begin{prop}
There exists a small constant $\hat{\epsilon}=\hat{\epsilon}(n)>0$ satisfying the following property.

Suppose $(M^n,p,g,f)$ is a Ricci shrinker and $(M^n, g(t))_{t \leq 0}$ is the associated ancient Ricci flow solution. 
If 
\begin{align}
d_{PGH} \left\{ (M^n,p,g), (S^{n-1}\times \R,p_c,g_c)\right\}<\hat{\epsilon}, \label{eqn:OH06_1}
\end{align} 
then one of the following conclusions hold:
\begin{itemize}
\item[(A).] The flow is $\delta_1$-symmetric of type A (cf. Definition~\ref{def:sym1} and Theorem~\ref{Tneck}). 
\item[(B).] The flow is $\delta_2$-symmetric of type B, and Assumption~\ref{assum1} and Assumption~\ref{assum2} are satisfied (cf. Definition~\ref{def:sym2} and Theorem~\ref{Thm:improve2}). 
\item[(C).] The flow is $\delta_4$-symmetric of type C, and Assumption~\ref{assum1} and Assumption~\ref{assum3} are satisfied (cf. Definition~\ref{def:sym3} and Theorem~\ref{Thm:improve3}). 
\end{itemize}

\label{prn:OH06_3}
\end{prop}

\begin{proof}
First, we fix some constants and parameters as follows:
\begin{align*}
A=-\boldsymbol{\mu}(S^{n-1} \times \R,g_c)+1,\quad B=&1, \quad \delta_0=\frac{1}{2}, \\
\ep_1= \sigma=\sigma(n,\delta_0,A,B), \quad \ep_0=&\ep_1^{\frac 1 2}, \quad \ep_2=\ep_0^{\frac 1 4},\\
L=L(n,\delta_0,A,B), \quad \ep_1'=&\eta=\eta(n,\delta_0,A,B,\ep), 
\end{align*}
where $\sigma,L$ and $\eta$ are functions defined in Theorem \ref{T101}.
Notice that $\ep_1'$ and $\eta$ depend on $n$ and $\ep$ and all other constants depend only on $n$.

We argue by contradiction.

If the Proposition were wrong, we could find a sequence of Ricci shrinkers $(M^n_i,p_i,g_i,f_i)$ which does not satisfy any situation of (A), (B), or (C) and satisfy
\begin{align*}
(M_i, p_i, g_i, f_i) \longright{pointed-Gromov-Hausdorff} (S^{n-1} \times \R, g_c,p_c).
\end{align*}
Note that the above convergence can be improved (cf. \cite[Theorem $1.1$]{LLW21}) as follows
\begin{align}
(M_i, p_i, g_i) \longright{C^{\infty}-Cheeger-Gromov} \lc S^{n-1} \times \R,p_c,g_c \rc. \label{E603a}
\end{align}
Furthermore, by \cite[Proposition $8.8$]{LLW21}, we know $\boldsymbol{\mu}(M_i,g_i) \to \boldsymbol{\mu}(S^{n-1} \times \R,g_c)$. 
Consequently, there exists a sequence $r_i \to +\infty$ satisfying the following properties if $i$ is sufficiently large.
\begin{enumerate}
\item $\{x \in M_i \mid r_i \le f_i(x) \le 2r_i \}$ consists of two components, denoted by $\Sigma_i^1(r_i,2r_i)$ and $\Sigma_i^2(r_i,2r_i)$, which satisfy the condition \eqref{E001} of Theorem \ref{T101}.
\item Each point in $\{x \in M_i \mid f_i(x) \le \eta^{-1}r_i \}$ is the center of an evolving normalized $\ep$-neck.
\end{enumerate}
From Theorem \ref{T101}, there exist sets $E_i^1$ and $E_i^2$ with $\partial E^1_i=\Sigma_i^1(r_i)$ and $\partial E^2_i=\Sigma_i^2(r_i)$ such that one of the following holds:
\begin{itemize}
\item[(A).] Both $E_i^1$ and $E_i^2$ are ends of $M_i$.
\item[(B).] $E_i^1$ is compact and $E_i^2$ is an end of $M_i$.
\item[(C).] Both $E_i^1$ and $E_i^2$ are compact. 
\end{itemize}
We shall discuss the above situations case by case. \\

\textit{Case A: If (A) happens, then the flow $(M_i, g_i(t))_{t \leq 0}$ is $\delta_1$-symmetric of type A.}

We take $\ep \leq \tau_4(\delta_1,n)$, where $\tau_4$ is from Lemma \ref{L601} and $\delta_1$ is the constant in Theorem \ref{Tneck}. It follows from Theorem \ref{T101} that any point in $\Sigma_i^1(\eta^{-1}r_i,\infty)$ or $\Sigma_i^2(\eta^{-1}r_i,\infty)$ is the center of an evolving normalized $\tau_4$-neck. Combining with property 2 above, we conclude from Lemma \ref{L601} that each point $(x,0) \in M_i \times \{0\}$ is $\delta_1$-symmetric. Moreover, it is easy to see $M_i$ is diffeomorphic to $S^{n-1} \times \R$. Therefore, $(M_i,g_i)$ is $\delta_1$-symmetric of type-A at time $0$, see Definition \ref{def:sym1}.

Since the associated Ricci flow $g_i(t)$ of $(M_i,g_i)$ is self-similar, we claim that each spacetime point $(\bar x,\bar t)\in M_i \times (-\infty,0]$ is also $\delta_1$-symmetric. Indeed, it follows from the definition of the diffeomorphisms $\psi_i^t$ \eqref{E211} that
\begin{align*}
\psi_i^{\theta(s)} \circ \psi_i^{\bar t}=\psi_i^{s},
\end{align*}
where $\theta(s)=\frac{s-\bar t}{1-\bar t}$. Therefore, for any $s \le \bar t$,
\begin{align*}
g_i(s)=(1-s)(\psi_i^s)^*g_i=(1-\bar t)(1-\theta(s))(\psi_i^{\bar t})^*(\psi_i^{\theta(s)})^*g_i=(1-\bar t)(\psi_i^{\bar t})^*g_i(\theta(s)).
\end{align*}
We set $\psi_i^{\bar t}(\bar x)=x$, then $(x,0)$ is $\delta_1$-symmetric. By comparing the similar parabolic neighborhoods in the Definition \ref{neck_symmetry} based at $(x,0)$ and $(\bar x, \bar t) $ respectively, it is easy to see $(\bar x,\bar t)$ is also $\delta_1$-symmetric, since the CMC foliation is unique, see \cite[Proposition D.1]{BK20}. In particular, if $\mathcal U = \{U^{(a)} : 1 \leq a \leq {n \choose 2}\}$ is the collection of vector fields in the Definition \ref{neck_symmetry} around $(x,0)$, then $ \bar{\mathcal U} = \{\bar U^{(a)}=(\psi_i^{\bar t})^{-1}_* U^{(a)}: 1 \leq a \leq {n \choose 2}\}$ is the collection around $(\bar x,\bar t)$. In summary, we know the flow $(M_i, g_i(t))_{t \leq 0}$ is $\delta_1$-symmetric of type-A . \\

\textit{Case B: If (B) happens, then the flow $(M_i, g_i(t))_{t \leq 0}$ is $\delta_2$-symmetric of type B, and Assumption \ref{assum1} and Assumption \ref{assum2} in Appendix~\ref{app:B} are satisfied.}

In this case, it follows from Theorem \ref{T101} that there exists $s_i \ge 2r_i$ and a point $q_i \in \Sigma(s_i)$ with $f_i(q_i)=\ep_1 R_i(q_i)$. By our choice of parameters, any point $x \in M_i \backslash D_i$ is the center of an evolving normalized $\ep$-neck, where $D_i=E^1_i \backslash \Sigma_i^1(r_i,s_i)$. Moreover, we fix a point $z_i \in D_i$ to be a maximum point of $f_i$.

\textbf{Claim 1}: On $P_{2,i}=\{(x,t) \mid F_i(x,t) \le 2r_i,\, 0\le t <1\}$, we have
\begin{align}\label{E603b}
R_i \ge \frac{1}{2}.
\end{align}

\emph{Proof of Claim 1}: For any $(x,t) \in M_i \times [0,1)$ with $F_i(x,t) \le 2r_i$, we have $(1-t)f_i(\psi^t_i(x)) \le 2r_i$ by the definition of $F_i$. If $f_i(\psi^t_i(x)) \le r_i$, then $\psi^t_i(x)$ is the center of an evolving normalized $\ep$-neck. In particular, $R_i(\psi^t_i(x)) \ge 1$ and hence
\begin{align*}
R_i(x,t)=\frac{1}{1-t}R_i(\psi^t_i(x)) \ge 1.
\end{align*}
If $f_i(\psi^t_i(x)) \ge r_i$, then it follows from Lemma \ref{L401} that
\begin{align*}
R(\psi^t_i(x)) \ge \frac{r_i}{f_i(\psi^t_i(x))}
\end{align*}
and hence
\begin{align*}
R_i(x,t)=\frac{1}{1-t}R_i(\psi^t_i(x)) \ge \frac{1}{1-t}\frac{r_i}{f_i(\psi^t_i(x))} \ge \frac{1}{2}.
\end{align*}

\textbf{Claim 2}: If we set $\tilde f_i=f_i-s_i$ and $\tilde g_i=s_ig_i$, then $(M_i,q_i,\tilde g_i,\tilde f_i)$ converges smoothly in the Cheeger-Gromov sense to the Bryant soliton.

\emph{Proof of Claim 2}: We consider a continuous family of cones $\hat C(\theta)$ from Example \ref{exam:cone}, for $\theta$ a small positive number.

For any small $\theta$, since the curvature tensor of $S^{n-1} \times \R$ is contained in $\hat C(\theta)$, then there exists a sequence $v_i \to 0^+$ such that
\begin{align*}
Rm_i+v_i I \in \hat C(\theta).
\end{align*}
on $\{(x,0) \in M_i \times \{0\} \mid F_i(x,0)\le 2r_i\}$. Therefore, it follows from \eqref{E603b} and Theorem \ref{thm:T601} that
\begin{align}\label{E603d}
Rm_i+(v_i+C_1r_i^{-2}) I \in \hat C(\theta)
\end{align}
on $P_{1,i}=\{(x,t) \mid F_i(x,t) \le r_i,\, 0\le t <1\}$, for some constant $C_1=C_1(n)>0$, if $i$ is sufficiently large.

Now we define the set $X_i \coloneqq E_i^1\backslash \Sigma_i^1(r_i,2r_i)$ and the following subsets
\begin{align*}
\begin{cases}
A_i & \coloneqq \{x\in X_i \mid \text{there exists a point } y \text{ with } f_i(y) \le r_i \text{ such that } \psi_i^t(y)=x \text{ for some } t \in (0,1) \}, \\
B_i & \coloneqq \{x\in X_i \mid |\na f_i|(x)=0 \},\\
C_i & \coloneqq X_i\backslash (A_i \cup B_i).
\end{cases}
\end{align*}
\noindent
It is clear that $A_i,B_i$ and $C_i$ are disjoint and $A_i$ is open. Moreover, from our construction $B_i \cup C_i \subset D_i$. For any $x \in A_i$ with $\psi_i^t(y)=x$ for some $f_i(y) \le r_i$, we compute 
\begin{align*}
F_i(y,t)=(1-t)f_i(y,t)=(1-t)f_i(\psi_i^t(y)) \le f_i(y) \le r_i,
\end{align*}
where we have used the ODE inequality
\begin{align*}
\frac{df_i(\psi_i^t(y))}{dt} = \frac{|\na f_i|^2(\psi_i^t(y))}{1-t} \le \frac{f_i(\psi_i^t(y))}{1-t}.
\end{align*}

Therefore, it follows from \eqref{E603d} that
\begin{align*}
Rm_i(x)+(1-t)(v_i+C_1r_i^{-2}) I \in \hat C(\theta).
\end{align*}
In particular, we have for any $x \in A_i$,
\begin{align} \label{E603e}
Rm_i(x)+(v_i+C_1r_i^{-2}) I \in \hat C(\theta).
\end{align}

Next we define 
\begin{align*}
&Y_i \coloneqq \textrm{the closure of the interior part of $B_i \cup C_i$.} \\
&Z_i \coloneqq (B_i \cup C_i)\backslash Y_i. 
\end{align*}
It is clear that $Z_i$ has no interior point and $Z_i \subset \partial A_i$. From the continuity, we conclude from \eqref{E603e} that on $Z_i$
\begin{align} \label{E603f}
Rm_i+(v_i+C_1r_i^{-2}) I \in \hat C(\theta).
\end{align}

Now we claim that there exists a constant $C_2=C_2(n)>0$ such that 
\begin{align*} 
R_i \ge (1-C_2 \theta) s_i
\end{align*}
on $\partial Y_i$ if $i$ is sufficiently large.

Indeed, if $x \in B_i$, then $R_i(x)=f_i(x)$ and hence by \eqref{E002} that
\begin{align*}
\lim_{i \to \infty} \frac{R_i(x)}{s_i}=\lim_{i \to \infty} \frac{f_i(x)}{s_i}=1
\end{align*}
uniformly since $x \in D_i$. If $x \in C_i$, we denote the gradient flow generated by $\frac{\na f_i}{|\na f_i|^2}$ by $\phi_i^t$ and there exist $y \in B_i$ and $t_i<0$ such that
\begin{align*}
\lim_{t \to t_i^+}\phi_i^{t}(x)=y.
\end{align*}
Moreover, it follows from \eqref{E002} again that $t_i>-L$ since both $x,y \in D_i$. From the choice of $x$, there exists a sequence $x_j \in A_i \to x$. We define $y_j=\phi_i^{t_i}(x_j)$. Then it is clear that $y_j \to y$. Now we have the ODE:
\begin{align}\label{E603g}
\frac{d R_i(\phi_i^t(x_j))}{dt}=\frac{\la \na R_i,\na f_i \ra}{|\na f_i|^2}(\phi_i^t(x_j))=\frac{2 Rc_i (\na f_i,\na f_i)}{|\na f_i|^2}(\phi_i^t(x_j)).
\end{align}

From the definition of $\hat C(\theta)$ and \eqref{ex:001}, it is clear that $Rc \ge -C\theta R$ for any $Rm \in \hat C(\theta)$. Combining this fact with \eqref{E603f}, we have on $A_i$ that
\begin{align} \label{E603h}
Rc_i \ge -C\theta R_i
\end{align}
for some constant $C=C(n)>0$, if $i$ is sufficiently large. Hence, it follows from \eqref{E603g} and \eqref{E603h} that
\begin{align*}
\frac{d R_i(\phi_i^t(x_j))}{dt} \ge -C\theta R_i(\phi_i^t(x_j))
\end{align*}
and hence
\begin{align*}
R_i(x_j) \ge e^{-C\theta |t_i|} R_i(\phi_i^{t_i}(x_j))=e^{-C\theta |t_i|} R_i(y_j) \ge e^{-C\theta} R_i(y_j).
\end{align*}

Notice that $R_i(y)=f_i(y)$ since $y$ is a critical point of $f_i$. By taking $j \to \infty$, we conclude 
\begin{align*}
R_i(x) \ge e^{-C\theta} R_i(y)= e^{-C\theta} f_i(y)
\end{align*}
and hence by \eqref{E002}
\begin{align} \label{E604a}
R_i(x) \ge (1-C_2\theta) s_i
\end{align}
if $i$ is sufficiently large. Now we have the elliptic equation 
\begin{align*}
\Delta_{f_i} R_i=R_i-2|Rc_i|^2 \le R_i-\frac{2}{n} R_i^2 \le 0
\end{align*}
on $Y_i$, where for the last inequality we have used $Y_i \subset D_i$ and \eqref{E002}. From \eqref{E604a} and the maximum principle, we have on $Y_i$,
\begin{align} \label{E604b}
R_i \ge (1-C_2\theta) s_i.
\end{align}

From Theorem \ref{T101} (c), we have
\begin{align*}
(M_i, q_i, \tilde g_i, \tilde f_i) \longright{pointed-\hat{C}^{\infty}-Cheeger-Gromov} \left(X_{\infty}, q_{\infty}, d_{\infty}, f_{\infty} \right),
\end{align*}
where $(X_{\infty}, q_{\infty}, d_{\infty}, f_{\infty}) \in \mathfrak S(n)$ is a nontrivial steady soliton conifold. From \eqref{E603e}, \eqref{E603f} and \eqref{E604b}, we have
\begin{align} \label{E604c}
Rm_{\tilde g_i}+s_i^{-1}(v_i+C_1r_i^{-2}) I \in \hat C(\theta)
\end{align}
on $A_i \cup Z_i$ and
\begin{align} \label{E604d}
R_{\tilde g_i} \ge 1-C_2\theta
\end{align}
on $Z_i$. Therefore, for any $x \in \mathcal R$, it follows from \eqref{E604c}, \eqref{E604d} and the smooth convergence that
\begin{align*}
Rm_{\infty}(x) \in \hat C(\theta) \quad \text{or} \quad R_{\infty}(x) \ge 1-C_2\theta.
\end{align*}
Now we let $\theta \to 0$ and conclude that
\begin{align*} 
Rm_{\infty}(x) \in \hat C(0) \quad \text{or} \quad R_{\infty}(x) = 1.
\end{align*}
Here, we have used the fact that $|\na f_{\infty}|^2+ R_{\infty}(x)=1$ on $\mathcal R$. Notice that the set $\{x \in \mathcal R \mid R_{\infty}(x) = 1\}$ has no interior point since otherwise from $\Delta_{f_{\infty}} R_{\infty}=-2|Rc_{\infty}|^2$ 
and the analyticity that $Rc_{\infty} \equiv 0$ on $\mathcal R$, which contraditcs the fact that $R_{\infty}(q_{\infty})=\ep_1$. By continuity, we have
\begin{align*} 
Rm_{\infty}(x) \in \hat C(0) 
\end{align*}
for any $x \in \mathcal R$. In addition, it follows from \cite[Proposition $2.3$]{Bre18} that $\hat C(0) \subset \text{PIC2}$. Therefore, it implies that on $\mathcal R$
\begin{align*}
Rm_{\infty} \in \text{PIC2}.
\end{align*}

Now we claim that $(\mathcal R,g)$ has strictly PIC2. Indeed, it follows from Proposition \ref{prop:464}, Proposition \ref{prop:500} and the smooth convergence that
\begin{align*}
Rm_{\infty}(x) >\frac{c(n)}{t} R_{\infty}(x)>0
\end{align*}
if $f_{\infty}(x)=-t<0$. Therefore, it follows from the geodesic convexity of $\mathcal R$ and the strong maximum principle (see \cite[Proposition $9$]{BS08}) that $(\mathcal R,g)$ has strictly PIC2. 

Moreover, one can show as \cite[Lemma $6.1$]{CN09} that $-f_{\infty}$ is comparable to the distance function $d_{\infty}(z_{\infty},\cdot)$ outside a compact set, where $z_{\infty}$ is the limit point of $z_i$. From Proposition \ref{prop:500} we conclude
\begin{align*}
\frac{c_1}{d_{\infty}(z_{\infty},x)} \le R_{\infty}(x) \le \frac{c_2}{d_{\infty}(z_{\infty},x)} 
\end{align*}
for $x$ outside a compact set, where $c_1,c_2$ are positive constants. 

Combined with Proposition \ref{prop:463}, it is clear that all assumptions of Theorem \ref{thm:Bryant} are satisfied. Therefore, we conclude that the limit steady soliton conifold $ (X_{\infty}, d_{\infty})$ must be isometric to the Bryant soliton. In particular, it implies that $(M_i,q_i,\tilde g_i,\tilde f_i)$ converges smoothly in the Cheeger-Gromov sense to the Bryant soliton whose maximum of the scalar curvature is $1$.

In addition, if we set $\tilde g_i(t)=s_i g_i(s_i^{-1}t)$, then the Ricci flow $(M_i,q_i,\tilde g_i(t))_{t \le 0}$ converges smoothly to the Ricci flow associated with the Bryant soliton.

\textbf{Claim 3}: There exists a constant $\theta_0=\theta_0(n)>0$ such that for sufficiently large $i$ if $(x,t) \in M_i \times (-\infty,0]$ satisfies
\begin{align*}
\lambda_i(x,t) \le \theta_0 R_i(x,t),
\end{align*}
where $\lambda_i$ is the minimal eigenvalue of $Rc_i$, then $x$ is $\delta_1$-symmetric in the sense of Definition \ref{neck_symmetry}.

\emph{Proof of Claim 3}: We may assume that the parameter $\ep <\tau_4(\delta_1,n)$ so that any point $x \in M_i \backslash D_i$ is $\delta_1$-symmetric. From Claim 2, it is clear that $(D_i,s_i g_i)$ converges smoothly to a cap in the Bryant soliton. Therefore, there exists a constant $\theta_0$ depending only on $n$ that
\begin{align*}
\lambda_i \ge \theta_0 R_i,
\end{align*}
on $D_i$, if $i$ is sufficiently large. In other words, any point $x \in M_i$ with $\lambda_i(x) \le \theta_0 R_i(x)$ is contained in $M_i \backslash D_i$ and hence $\delta_1$-symmetric.

Since the Ricci flow $(M_i,g_i(t))_{t \le 0}$ associated with $(M_i,g_i)$ is self-similar, it is clear that if a spacetime point $(x,t)$ satisfies
\begin{align*}
\lambda_i(x,t) \le \theta_0 R_i(x,t),
\end{align*}
then $(x,t)$ is $\delta_1$-symmetric.

\textbf{Claim 4}: There exist constants $\ep'=\ep'(n)>0$ and $\Lambda_0=\Lambda_0(n)>0$ such that for $\ep \le \ep'$ and sufficiently large $i$, if $(x,\bar t) \in M_i \times (-\infty,0]$ and $d_{g_i(\bar t)}(z_i,x) \ge \Lambda_0 R_i(z_i,\bar t)^{-\frac 1 2}$, then
\begin{align*}
\lambda_i(y,t) \le \frac{\theta_0}{2} R_i(y,t),
\end{align*}
for any $(y,t) \in B_{g_i(\bar t)}(x, \bar L R_i(x,\bar t)^{-\frac 1 2}) \times [-\bar L R_i( x,\bar t)^{-1},\bar t]$, where $\bar L$ is the constant from Theorem \ref{Tneck}.

\emph{Proof of Claim 4}: We choose our parameter $\ep $ small enough so that if a point $x \in M_i$ is the center of an evolving normalized $\ep$-neck, then $\lambda_i(x) \le \frac{\theta_0}{2} R_i(x)$ and $R_i(u) \ge 0.9 R_i(v)$ for any $u,v \in B_{g_i}(x,\bar L R_i(x)^{-\frac 1 2})$.

Assume $y\in B_{g_i}(x,\bar L R_i(x)^{-\frac 1 2}) \cap D_i$, then 
\begin{align*}
R_i(x) \ge \frac{9}{10} R_i(y) \ge \frac{9s}{10 L},
\end{align*}
where the last inequality follows from \eqref{E002}. Therefore,
\begin{align*}
d_{g_i}(z_i,x) \le d_{g_i}(z_i,y)+d_{g_i}(y,x) \le \lc L+\bar L \lc \frac{10L}{9} \rc^{\frac 1 2}\rc s^{-\frac 1 2}.
\end{align*}
On the other hand, since $z_i$ is a critical point of $f_i$, it follows from \eqref{E002} that 
\begin{align*}
\lim_{i \to \infty} \frac{R_i(z_i)}{s_i}=\lim_{i \to \infty} \frac{f_i(z_i)}{s_i}=1.
\end{align*}
From \eqref{E002} again, if we take $\Lambda_0=2L+2\bar L \lc \frac{10L}{9} \rc^{\frac 1 2}$, then for any point $x$ with $d_i(x,z_i) \ge \Lambda_0 R_i(z_i)^{-\frac 1 2}$, $B_{g_i}(x,\bar L R_i(x)^{-\frac 1 2})$ must be contained in $M_i \backslash D_i$ and hence 
\begin{align*}
\lambda_i(y) \le \frac{\theta_0}{2} R_i(y)
\end{align*}
for any $y \in B_{g_i}(x,\bar L R_i(x)^{-\frac 1 2})$.

In addition, for any $(y,t) \in B_{g_i}(x,\bar L R_i(x)^{-\frac 1 2}) \times [-\bar L R_i(x)^{-1},0]$, $\psi_i^t(y) \notin D_i$ since $f_i(\psi_i^t(y)) \le f_i(y)$. 
Therefore, for the associated Ricci flow $(M_i,g_i(t))_{t \le 0}$, if $(x,0) \in M_i \times (-\infty,0]$ and $d_{g_i(0)}(z_i,x) \ge \Lambda_0 R_i(z_i,0)^{-\frac 1 2}$, then
\begin{align*}
\lambda_i(y,t) \le \frac{\theta_0}{2} R_i(y,t),
\end{align*}
for any $(y,t) \in B_{g(0)}(x, \bar L R_i(x,0)^{-\frac 1 2}) \times [-\bar L R_i( x,0)^{-1},0]$. Therefore, the Claim follows from the self-similarity of $(M_i,g_i(t))_{t \le 0}$.

Now we fix $\ep=\min\{ \tau_4(\delta_1,n),\tau_4(\delta_2,n),\ep'(n)\}$, where $\delta_2$ is from Theorem \ref{Thm:improve2}. From Claim 2-Claim 4, it is clear that Assumption \ref{assum1} and Assumption \ref{assum2} in Appendix~\ref{app:B} are satisfied.
Moreover, if $i$ is sufficiently large, $(M_i,g_i(t))$ is $\delta_2$-symmetric of type-B at time $0$ in the sense of Definition \ref{def:sym2}. From the self-similarity of $(M_i,g_i(t))$ as in Case A, we immediately conclude that $(M_i,g_i(t))$ is $\delta_2$-symmetric of type-B at time $t$ for any $t \le 0$.\\

\textit{Case C: If (C) happens, then the flow $(M_i, g_i(t))_{t \leq 0}$ is $\delta_4$-symmetric of type C, and Assumption \ref{assum1} and Assumption \ref{assum3} in Appendix~\ref{app:B} are satisfied.}

In this case, it follows from Theorem \ref{T101} that there exist $s^1_i,s^2_i \ge 2r_i$ and points $q^1_i \in \Sigma_i^1(s_i^1)$ and $q_i^2 \in \Sigma_i^2(s_i^2)$ with $f_i(q^1_i)=\ep_1 R_i(q^1_i)$ and $f_i(q^2_i)=\ep_1 R_i(q^2_i)$ . By our choice of parameters, any point $x \in M_i \backslash D_i$ is the center of an evolving normalized $\ep$-neck, where $D_i=D_i^1 \cup D_i^2=(E^1_i \backslash \Sigma_i^1(r_i,s^1_i)) \cup (E^2_i \backslash \Sigma_i^2(r_i,s^2_i))$. Moreover, we fix points $z^1_i \in D^1_i$ and $z_i^2 \in D_i^2$ to be a maximum point of $f_i$ on $D^1_i$ and $D^2_i$, respectively.

If we set $\tilde g_i^1=s_i^1g_i$, $\tilde g_i^2=s_i^2 g_i$, $\tilde f_i^1=f_i-s_i^1$ and $\tilde f_i^2=f_i-s_i^2$.
Following the the discussion in case B, it is clear that both $(M_i,q^1_i,\tilde g^1_i,\tilde f^1_i)$ and $(M_i,q^2_i,\tilde g^2_i,\tilde f^2_i)$ converge smoothly in the Cheeger-Gromov sense to the Bryant soliton.
Furthermore, there exist constants $\ep''=\ep''(n)>0$ and $\Lambda_1=\Lambda_1(n)>0$ such that for $\ep \le \ep''$ and sufficiently large $i$, if $(x,\bar t) \in M_i \times (-\infty,0]$, $d_{g_i(\bar t)}(z^1_i,x) \ge \Lambda_1 R_i(z_i,\bar t)^{-\frac 1 2}$ and $d_{g_i(\bar t)}(z^2_i,x) \ge \Lambda_1 R_i(z_i,\bar t)^{-\frac 1 2}$, then
\begin{align*}
\lambda_i(y,t) \le \frac{\theta_0}{2} R_i(y,t),
\end{align*}
for any $(y,t) \in B_{g_i(\bar t)}(x, \bar L R_i(x,\bar t)^{-\frac 1 2}) \times [-\bar L R_i( x,\bar t)^{-1},\bar t]$.

Now we fix $\ep=\min\{ \tau_4(\delta_1,n),\tau_4(\delta_4,n),\ep''(n)\}$, where $\delta_4$ is from Theorem \ref{Thm:improve3}. From Claim 3, Claim 4 and the discussion above, it is clear that Assumption \ref{assum1} and Assumption \ref{assum3} in Appendix~\ref{app:B} are satisfied. 
Moreover, if $i$ is sufficiently large, $(M_i,g_i(t))$ is $\delta_4$-symmetric of type-C at time $0$ in the sense of Definition \ref{def:sym3}. Applying the self-similarity of $(M_i,g_i(t))$ as in Case A, we immediately conclude that $(M_i,g_i(t))$ is $\delta_4$-symmetric of type-C at time $t$ for any $t \le 0$.\\

In conclusion, for each large $i$, the associated ancient Ricci flow $(M_i, g_i(t))_{t \leq 0}$ must locate in one of the situations of (A), (B), or (C).
However, this contradicts our assumption at the beginning. This contradiction establishes the proof of this Proposition. 
\end{proof}

Now we are able to finish the proof of Theorem~\ref{T100}. 

\begin{proof}[Proof of Theorem~\ref{T100}:]
It follows from the combination of Theorem~\ref{Tneck}, Theorem~\ref{Thm:improve2} and Theorem~\ref{Thm:improve3} that each ancient Ricci flow solution $(M^n, g(t))_{t \leq 0}$ 
satisfying one of (A), (B) or (C) in Proposition~\ref{prn:OH06_3} is rotationally symmetric. 
Then we can apply Kotschwar's classification~\cite{Kot08} to obtain that $(M,g)$ must be isometric to $S^n$, $S^{n-1} \times \R$, or $\R^{n}$. 
Applying (\ref{E100}) again, we know that $(M,g)$ can only be $S^{n-1} \times \R$. 
\end{proof}

\section{Further discussion}

In this paper, we mainly focus on the model space $(S^{n-1} \times \R,g_c)$. One can also consider the model space $\lc (S^{n-1}/\Gamma) \times \R,g_c \rc$, where $\Gamma \le O(n)$ is a finite subgroup acting freely on $S^{n-1}$. 
Then we can also define a spacetime $(x,t)$ in a Ricci flow $(M,g(t))$ to be the center of an evolving (normalized) $\ep$-neck if the model space $S^{n-1} \times \R$ is replaced by $(S^{n-1}/\Gamma) \times \R$. 

The following theorem is another version of Theorem \ref{T101} in this case, whose proof is almost identical and hence omitted.

\begin{thm}\label{T701}
For any positive constants $n,N,A,B,\ep$ and $\delta_0\in (0,1)$, there exist positive constants $\sigma=\sigma(n,N,\delta_0,A,B)$, $L=L(n,N,\delta_0,A,B)$ and $\eta=\eta(n,N, \delta_0,A,B,\ep)$ satisfying the following property.

Let $(M^n,g,f)$ be a Ricci shrinker with
\begin{align*}
\begin{cases}
|\na^i Rm| \le BR^{\frac{i}{2}+1}, \quad \forall \, 0\le i \le 4 \quad &\text{on}\quad \Sigma(t_0,s_0), \\
R\le \ep_1 f \quad &\text{on} \quad \Sigma(t_0,s_0), \\
|\Rm| \le \ep_1' R \quad &\text{on}\quad \Sigma(t_0), \\
\left |R-\frac{n-1}{2} \right| \le \ep'_1 \quad &\text{on}\quad \Sigma(t_0), \\
(1-\delta_0)s_0\ge t_0 \ge n \ep_1^{-1}, \\
\boldsymbol{\mu}(g) \ge -A,\\
\Sigma(t_0) \text{ is diffeomorphic to } S^{n-1}/\Gamma \text{ with } |\Gamma |\le N.
\end{cases}
\end{align*}
If $\ep_1 \le \sigma$ and $\ep_1' \le \eta$, then one of the following statements holds.
\begin{itemize}
\item[(a).] There exists an end $E$ with $\partial E=\Sigma(t_0)$ such that any point in $\Sigma(\eta^{-1} t_0,\infty) \subset E$ is the center of an evolving normalized $\ep$-neck. Moreover, $E$ is asymptotic to $\lc (S^{n-1}/\Gamma) \times \R,g_c \rc$ with rate $O(r^{-\tau(n)})$. 
\item[(b).] There exists an end $E$ with $\partial E=\Sigma(t_0)$ such that any point in $\Sigma(\eta^{-1} t_0,\infty) \subset E$ is the center of an evolving $\ep$-neck. Moreover, $E$ is asymptotic to a regular cone with cross section diffeomorphic to $S^{n-1}/\Gamma$. 
\item[(c).] There exists a compact set $E$ with $\partial E=\Sigma(t_0)$, a number $s\ge s_0$ and a point $q \in \Sigma(s)$ such that $R(q)=\sigma s$. Moreover, any point in $\Sigma(\eta^{-1} t_0,s) \subset E$ is the center of an evolving $\ep$-neck and any point in the cap $D:=E \backslash \Sigma(t_0,s)$ is the center of an $\ep$-steady soliton conifold. Furthermore, 
\begin{align*}
\emph{diam}_g D \le \frac{L}{\sqrt{s}}, \quad \sup_D|f-s| \le L \quad \text{and} \quad L^{-1}s \le \inf_D R \le \sup_D R \le s+L.
\end{align*} 
\end{itemize}
\end{thm}

However, the following conjecture originally proposed in \cite{LWs1} is still open.

\begin{conj}\label{conj:1}
For any $n>0$ and $N>0$, there exists a small constant $\epsilon=\epsilon(n,N)>0$ with the following property.

Suppose $(M^n,p,g,f)$ is a Ricci shrinker such that
\begin{align*}
d_{PGH} \left\{ (M^n,p,g), \lc (S^{n-1}/\Gamma) \times \R, p_c, g_c \rc \right\}<\epsilon
\end{align*} 
with $|\Gamma|\le N$, then $(M,g)$ is isometric to $\lc (S^{n-1}/\Gamma) \times \R, p_c, g_c \rc$. 
\end{conj}

Notice that the same proof of Theorem \ref{T100} does not go through. The key point is another version of Theorem \ref{thm:Bryant} corresponding to $(S^{n-1}/\Gamma) \times \R$ does not hold. 
In fact, in this case one can show $(X,d)$ in Theorem \ref{thm:Bryant} is isometric to $(\R^n/\Gamma, g_{Br})$, where $g_{Br}$ is the quotient metric of the Bryant soliton. 
In other words, even though the cap region is modeled by $(\R^n/\Gamma, g_{Br})$, the underlying topology of the cap is unclear. A deeper bubble may exist and converge to the singularity of $(\R^n/\Gamma, g_{Br})$. 
For this reason, one cannot apply the theorems of symmetry as done in the proof of Theorem \ref{T100}.

However, if we assume $f$ has no critical point outside a compact set, we have the following weaker theorem.

\begin{thm} \label{T702}
There exists a small constant $\ep=\hat \ep(n,N)>0$ satisfying the following property.

Suppose $(M^n,p,g,f)$ is a Ricci shrinker such that $|\na f|>0$ on $\{x\in M \mid f(x)\ge n\}$ and 
\begin{align*}
d_{PGH} \left\{ (M^n,p,g), \lc (S^{n-1}/\Gamma) \times \R, p_c, g_c \rc \right\}<\hat \ep 
\end{align*} 
with $|\Gamma|\le N$, then $(M,g)$ is isometric to $\lc (S^{n-1}/\Gamma) \times \R, p_c, g_c \rc$. 
\end{thm}

\begin{proof}
Suppose the conclusion does not hold, there exists a sequence of Ricci shrinkers $(M^n_i,p_i,g_i,f_i)$ such that
\begin{align*}
(M_i, p_i, g_i, f_i) \longright{pointed-Gromov-Hausdorff} \lc (S^{n-1}/\Gamma_i) \times \R,p_c,g_c, f_c\rc
\end{align*}
for $|\Gamma_i| \le N$, but no $(M_i,g_i)$ is isometric to $\lc (S^{n-1}/\Gamma_i) \times \R,g_c \rc$. Since there are only finitely many conjugacy classes of subgroups of $O(n)$ with order no larger than $N$, we may assume $\Gamma_i=\Gamma$ for a fixed $\Gamma \le O(n)$ acting freely on $S^{n-1}$. 

If $i$ is large, the conclusion of Theorem \ref{T701} holds. Since $f_i$ is assumed to be regular on $\{x\in M_i \mid f_i(x) \ge n\}$, we conclude as in the proof of Theorem \ref{T100} that each $M_i$ is diffeomorphic to $(S^{n-1}/\Gamma) \times \R$ and each point is the center of an evolving normalized $\delta_1$-neck. By considering the universal cover $(\tilde M_i, \tilde g_i,\tilde f_i)$ of $(M_i,g_i,f_i)$. We can apply Theorem \ref{Tneck} by iteration and conclude that $(\tilde M_i, \tilde g_i,\tilde f_i)$ is rotationally symmetric. From \cite[Theorem $1$]{Kot08}, $(\tilde M_i, \tilde g_i)$ is isometric to $(S^{n-1} \times \R, g_c)$. This implies that $(M_i,g_i)$ is isometric to $\lc (S^{n-1}/\Gamma) \times \R,g_c \rc$, which is a contradiction.
\end{proof}

On the other hand, inspired by the main theorem of \cite{CIM15}, it is natural to guess that all generalized round cylinders are also rigid in Ricci shrinkers. 

\begin{conj}\label{conj:2}
For any $n$ and $2\le k \le n-2$, there exists a small constant $\ep=\ep(n)>0$ with the following property.

Suppose $(M^n,p,g,f)$ is a Ricci shrinker such that
\begin{align*}
d_{PGH} \left\{ (M^n,p,g), (S^{n-k} \times \R^k,p_c,g_c) \right\}<\epsilon,
\end{align*} 
then $(M,g)$ is isometric to $(S^{n-k} \times \R^k,g_c)$. 
\end{conj}

\newpage

\appendixpage
\addappheadtotoc
\appendix
\section{Steady soliton conifold} 
\label{app:A}

Following Chen-Wang~\cite{CW17}, we introduce the concept of the steady soliton conifold and prove many fundamental properties. 

\begin{defn}
Let $\mathfrak S(n)$ be the collection of geodesic spaces $(X^n,p , d,g, f)$ with the following properties.

\begin{enumerate}
\item $f$ is a locally Lipschitz function defined on $X$ and $f(p)=0$. 

\item $X$ has a disjoint regular-singular decomposition $X=\mathcal{R} \cup \mathcal{S}$, where $\mathcal{R}$ is the regular part, $\mathcal{S}$ is the singular part.
A point is called regular if it has a neighborhood that is isometric to a totally geodesic convex domain of some smooth Riemannian manifold. A point is called singular
if it is not regular.

\item $(\mathcal{R}, g)$ is an open (possibly incomplete) Riemannian manifold of dimension $n$ with $R \ge 0$ and satisfies the Ricci steady soliton equation
\begin{align}
\label{BE000}
Rc_f:=Rc+ \emph{Hess}\, f=0.
\end{align}

\item $\mathcal{R}$ is geodesically convex, i.e., for any pair of points $x,y \in \mathcal{R}$, there exists a shortest geodesic in $\mathcal R$ connecting $x$ and $y$.

\item $\dim_{\mathcal{M}} \mathcal{S} \le n-4$, where $\dim_{\mathcal{M}}$ means Minkowski dimension.

\item Every tangent space of $x \in \mathcal S$ is a metric cone of Hausdorff dimension $n$. Moreover, if $Y$ is a tangent cone of $x$, then the unit ball $B(\hat x,1)$ centered at vertex $\hat x$ must satisfy
\begin{align*}
|B(\hat x,1)| \le (1-\delta_0) \omega_n
\end{align*}
for some uniform positive number $\delta_0=\delta_0(n)>0$. Here the volume is the $n$-dimensional Hausdorff measure and $\omega_n$ is the volume of the unit ball in $\R^n$.
\end{enumerate}
\label{dfn:B001}
\end{defn}

Note that $(X, d)$ is called a (Riemannian) conifold if all the properties in the above definition except equation (\ref{BE000}) are satisfied (cf.~Definition 1.2 of~\cite{CW17}). 
The geodesic space $(X,d)$ is the completion of $(\mathcal R,g)$.   The equation (\ref{BE000}) justifies the ``steady soliton" in the name. 
By taking the divergence of \eqref{BE000}, we obtain $R+|\na f|^2=\lambda $ on $\mathcal R$ for a constant $\lambda \ge 0$. If $\lambda=0$, then $R \equiv 0$ and $f$ is a constant. From the equation $\Delta_f R=-2|Rc|^2$ on $\mathcal R$, we conclude that $Rc \equiv 0$ on $\mathcal R$. In this case, $(X,d)$ is a Ricci-flat Riemannian conifold and is called a \textbf{trivial} steady soliton conifold. If $\lambda>0$, we rescale the metric such that $\lambda =1$ and hence
\begin{align} \label{BE002}
R+|\na f|^2=1
\end{align}
on $\mathcal R$. In this case, $(X,d)$ is called a \textbf{nontrivial} Ricci steady soliton. In the following, we derive some geometric and analytic properties of $(X,d) \in \mathfrak S(n)$. Most results are generalizations of the corresponding results in \cite{HLW21} for smooth metric measure spaces and can be proved similarly as in \cite{CW17}. Moreover, we define the weighted measure $e^{-f} \,dV$ by $\mu$, where $dV$ is the $n$-dimensional Hausdorff measure and denote the volume with respect to $d\mu$ by $|\cdot|_{\mu}$.

\begin{prop}(Volume comparison) \label{prop:B001}
Let $(X,p,d,f) \in \mathfrak S(n)$. For any $0 <r_1<r_2$ and $x \in X$, we have
\begin{align} \label{BE003}
\frac{|B(x,r_2)|_{\mu }}{r_2^n} \le e^{r_2} \frac{|B(x,r_1)|_{\mu}}{r_1^n}.
\end{align}
\end{prop}

\begin{proof}
If $(X,d)$ is smooth, then \eqref{BE003} follows from \cite[Theorem $1.2$]{WW09} since $Rc_f\ge 0$ and $|\na f| \le 1$ by \eqref{BE002}. For the general case, we can derive the result as \cite[Proposition $2.3$]{CW17}.
\end{proof}

Next, we have the following segment inequality, which is a generalization of \cite[Theorem $2.6$]{HLW21} by following the argument of \cite[Proposition $2.6$]{CW17}

\begin{prop}(Segment inequality) \label{prop:B002}
Let $(X,p,d,f) \in \mathfrak S(n)$. For any $L>0$, there exists a constant $C_1=C_1(n,L)>0$ such that if $A_1,A_2$ are two subsets of $B(q,r) \subset \subset B(p,L)$, we have
\begin{align*}
\int_{A_1 \times A_2} \mathcal F_u(x,y)\,d\mu(x) d\mu(y) \le C_1r(|A_1|_{\mu}+|A_2|_{\mu}) \int_{B(q,3r)} u \,d\mu,
\end{align*}
where $u$ is a nonnegative continuous function on $B(q,3r)$ and \
\begin{align*}
\mathcal F_u(x,y):=\inf_{\gamma} \int_0^{d(x,y)} u(\gamma(t)) \,dt
\end{align*}
with infimum being taken over all geodesics connecting $q_1$ and $y$. 
\end{prop}

\begin{defn} \label{def:sobo}
The Sobolev space $N^{1,2}(X,\mu)$ is the subspace of $L^2(X,\mu)$ consisting of functions $u$ for which the norm
\begin{align*}
\| u \|^2_{N^{1,2}}=\|u\|^2_{L^2}+\inf_{u_i} \liminf_{i \to \infty} \| h_i \|^2_{L^2}<\infty,
\end{align*}
where the limit infimum is taken over all upper gradients $h_i$ of the functions $u_i$ satisfying $\|u_i-u\|_{L^2} \to 0$. Notice that for any domain $\Omega \subset X$, one can define $N^{1,2}(\Omega,\mu)$ similarly. Also, $N^{1,2}_c$, $N^{1,2}_0$ and $N^{1,2}_{\text{loc}}$ can be defined as the usual Sobolev space.
\end{defn}

The definition of the upper gradient can be found in Cheeger \cite[Definition $1.1$]{CH99}. Notice that the only difference between $N^{1,2}(X,\mu)$ and the one defined in \cite[Definition $2.10$]{CW17} is the former uses the measure $d\mu$. Similar to \cite[Corollary $2.12$]{CW17} and \cite[Proposition $2.12$]{CW17}, we have

\begin{prop}\label{prop:B003}
$C^{\infty}_c(\mathcal R) \cap N^{1,2}(X,\mu)$ is dense in $N^{1,2}(X,\mu)$. Moreover, for any open set $\Omega \subset X$, the restriction map $N^{1,2}(\Omega,\mu) \to W^{1,2}(\mathcal R \cap \Omega,\mu)$ is an isomorphic isometry.
\end{prop}

Combining Proposition \ref{prop:B001} and Proposition \ref{prop:B002}, we obtain the following local $L^2$-Poincar\'e inequality and local $L^2$-Sobolev inequalities as \cite[Proposition $2.7,2.9$]{HLW21}

\begin{prop}(Local $L^2$-Poincar\'e inequality) \label{prop:B004}
Let $(X,p,d,f) \in \mathfrak S(n)$. For any $L>0$, there exists a constant $C_2=C_2(n,L)>0$ such that for any $B(q,r) \subset B(p,L)$, we have
\begin{align*}
\aint_{B(q,r)} \left| u-\aint_{B(q,r)} u \,d\mu \right|^2 \,d\mu \le C_2r^2 \aint_{B(q,r)} |\na u|^2 \,d\mu,
\end{align*}
for any $u \in N^{1,2}_{\text{loc}}(X,\mu)$, where
\begin{align*}
\aint_{B(q,r)} u \,d\mu:=\frac{1}{|B(q,r)|_{\mu}} \int_{B(q,r)} u \,d\mu.
\end{align*}
\end{prop}

\begin{prop}(Local $L^2$-Sobolev inequality) \label{prop:B005}
Let $(X,p,d,f) \in \mathfrak S(n)$. For any $L>0$, there exists a constant $C_3=C_3(n,L)>0$ such that for any $B(q,r) \subset B(p,L)$, we have
\begin{align} \label{BE005}
\lc \int_{B(q,r)} u^{\frac{n}{n-2}} \,d\mu \rc^{\frac{n-2}{n}} \le \frac{C_3r^2}{|B(q,r)|_{\mu}^{\frac{2}{n}}} \int_{B(q,r)} |\na u|^2+r^{-2}u^2 \,d\mu,
\end{align}
for any $u \in N^{1,2}_{\text{loc}}(X,\mu)$.
\end{prop}

\begin{defn} \label{def:weak}
Suppose $u \in N^{1,2}_{\text{loc}}(\Omega,\mu)$ and $h \in L^{2}_{\text{loc}}(\Omega,\mu)$. Then $\Delta_f u \le h$ in the weak sense if for any nonnegative $v \in N^{1,2}_{c}(\Omega,\mu)$,
\begin{align*}
-\int_{\Omega \cap \mathcal R} \la \na u ,\na v \ra \,d\mu \le \int_{\Omega} v h \,d\mu.
\end{align*}
Moreover, we say $u$ is a harmonic function if $\Delta_f u=0$ in the weak sense.
\end{defn}

As the singular part $\mathcal S$ has high codimension, we can extend any bounded subharmonic function on the regular part globally. The following lemma can be proved similarly as \cite[Proposition $2.19$]{CW17}.

\begin{lem}\label{lem:extend}
Let $(X,p,d,f) \in \mathfrak S(n)$ and $\Omega$ a bounded open domain. Suppose $u$ is a bounded function on $\mathcal R \cap \Omega$ satisfying 
\begin{align*}
\Delta_f u = h \ge 0.
\end{align*}
Then $u \in N^{1,2}_{\text{loc}}(\Omega,\mu)$ and $\Delta_f u=h$ in the weak sense on $\Omega$.
\end{lem}

Next, we prove a quantitative version of the strong maximum principle.

\begin{prop}\label{prop:B006}
Let $(X,p,d,f) \in \mathfrak S(n)$. Given any $L$ and $q \in B(p,L)$, suppose $u$ is a bounded, nonnegative, continuous function on $B(q,3) \cap \mathcal R$ such that 
\begin{align} \label{BE006}
\Delta_f u \le 0
\end{align}
in the weak sense on $B(q,2) \cap \mathcal R$. Then there exists a constant $\delta_1=\delta_1(n,L)>0$ such that if $u(y)=\tau$ for some $y \in B(q,1) \cap \mathcal R$, then
\begin{align*}
|\{u \le 2\tau \} \cap B(q,2)|_{\mu} \ge \delta_1 |B(q,2)|_{\mu}.
\end{align*}
In particular, either $u \equiv 0$ on $B(q,2) \cap \mathcal R$ or 
\begin{align*}
\inf_{B(q,1) \cap \mathcal R} u >0.
\end{align*}
\end{prop}

\begin{proof}
It follows from Lemma \ref{lem:extend} that $u \in N^{1,2}(B(q,2),\mu)$ and \eqref{BE006} holds on $B(q,2)$, by the same argument of \cite[Proposition $2.19$]{CW17}.
If $u=0$ at some point in $B(q,2) \cap \mathcal R$, it follows from the classical strong maximum principle that $u \equiv 0$ on $B(q,2) \cap \mathcal R$. Therefore, we may assume $u >0$ on $B(q,2) \cap \mathcal R$. 

Since $(2\tau-u)^+$ is a bounded subharmonic function (with respect to $\Delta_f$) in $N^{1,2}(B(q,2),\mu)$. We can apply \eqref{prop:B005}, \eqref{BE005} and the standard Moser iteration to obtain
\begin{align} \label{BE007}
\|(2\tau-u)^+\|_{L^{\infty}(B(q,1))} \le C \lc \frac{1}{|B(q,1)|_{\mu}} \int_{B(q,2)} |(2\tau-u)^+|^2 \,d\mu \rc^{\frac 1 2} \le C \lc \aint_{B(q,2)} |(2\tau-u)^+|^2 \,d\mu \rc^{\frac 1 2}
\end{align}
for some constant $C=C(n,L)>0$, where we have used \eqref{BE003} for the last inequality. Since $u$ is continuous on $B(q,2) \cap \mathcal R$ and $u(y)=\tau$ for some $y \in B(q,1) \cap \mathcal R$, it follows from \eqref{BE007} that
\begin{align*} 
\tau \le C \lc \aint_{B(q,2)} |(2\tau-u)^+|^2 \,d\mu \rc^{\frac 1 2} \le 2C\tau \lc \frac{|\{u \le 2\tau \} \cap B(q,2)|_{\mu}}{|B(q,2)|_{\mu}} \rc^{\frac 1 2}.
\end{align*}
Therefore, the proof of \eqref{BE006} is complete. For the last conclusion, if $\inf_{B(q,1) \cap \mathcal R} u =0$, then there exists a sequence $y_i \in B(q,1) \cap \mathcal R$ with $u(y_i)=\tau_i \to 0$.
Applying \eqref{BE006} for $\tau=\tau_i$, we conclude that there exists a point $z \in \bar B(q,1) \cap \mathcal R$ such that $u(z)=0$, which is a contradiction. 
\end{proof}

As an application, we have

\begin{prop}\label{prop:B007}
Let $(X,p,d,f) \in \mathfrak S(n)$ be compact. Then $Rc \equiv 0$ on $\mathcal R$. 
\end{prop}

\begin{proof}
Since $X$ is compact and $f$ is locally Lipschitz, there exists a point $q \in X$ such that
\begin{align*}
f(q)=\inf_X f.
\end{align*}
From the equation $\Delta_f(f-f(q))=-(|\na f|^2+R) \le 0$, we conclude from Proposition \ref{prop:B006} that $f \equiv f(q)$. In this case, it follows from the steady soliton equation that $Rc \equiv 0$ on $\mathcal R$.
\end{proof}

Next, we have the following existence result and gradient estimates for harmonic functions. The proof follows verbatim from \cite[Proposition $2.26$]{CW17}.

\begin{prop}\label{prop:B008}
Let $\Omega \subset X$ be a bounded domain and $v$ a continuous function in $N^{1,2}(\Omega,\mu)$. Then there is a unique solution $u \in N^{1,2}(\Omega,\mu)$ solving the Dirichlet problem
\begin{align*}
\Delta_f u=0 \text{ in } \Omega \text{ and } u-v \in N_0^{1,2}(\Omega,\mu).
\end{align*}
Moreover, the weak maximum principle holds for $u$, that is,
\begin{align*}
\sup_{\Omega} u=\sup_{\partial \Omega} u,\quad \inf_{\Omega} u=\inf_{\partial \Omega} u.
\end{align*}
\end{prop}

Next, we have the following splitting result.

\begin{prop}\label{prop:B009}
Let $(X,p,d,f) \in \mathfrak S(n)$. Suppose there exists a smooth function $b$ on $\mathcal R$ satisfying
\begin{align*}
\Delta_f b=s \quad \text{and} \quad |\na b|=t 
\end{align*}
on $\mathcal R$, where $s$ and $t$ are constants with $t>0$. Then $(X,d)$ is isometric to $(Y \times \R,d' \times d_{g_E})$, where $(Y,d')$ is a Ricci steady soliton conifold.
\end{prop}

\begin{proof}
It follows from the Weitzenb\"{o}ck formula that on $\mathcal R$
\begin{align*}
0=\frac{1}{2}\Delta_f |\na b|^2=|\text{Hess}\,b|^2+Rc_f(\na b,\na b)+\la \na \Delta_f b, \na b \ra=|\text{Hess}\,b|^2.
\end{align*}
Therefore, $\text{Hess}\,b=0$ and one can prove as \cite[Lemma $2.31$]{CW17} that the gradient flow of $\na b$ preserves the regular point. Now the conclusion follows easily.
\end{proof}

For any $u,v \in N^{1,2}(X,\mu)$, we define a nonnegative symmetric bilinear from $\mathscr{E}$ by 
\begin{align*}
\mathscr{E}(u,v) := \int_{\mathcal R} \la \na u, \na v \ra \,d\mu.
\end{align*}

It can be proved (see \cite[Proposition $2.15$]{CW17}) that $\mathscr E$ is an irreducible, strongly local and regular Dirichlet form. The associated semigroup $(P_t)_{\ge 0}$ can be expressed as
\begin{align*} 
P_t(u)(y)=\int u(x) p(t,x,y) \,d\mu(x), \quad \forall u \in L^2(X,\mu),\,t>0.
\end{align*}
Here, $p(t,x,y)$ is called the heat kernel with respect to $\mathscr E$. Notice the $\Delta_f$ can be regarded as the unique generator concerning $\mathscr E$ and the domain of $\Delta_f$, denoted by $\Sigma(\Delta_f)$, is a subspace of $L^2(X,\mu)$.

We have the following estimates which follow from \cite{Stu95} and \cite{Stu96} as \cite[Proposition $2.20$]{CW17}, see also \cite[Theorem $3.1$]{HLW21}.

\begin{thm}(Heat kernel estimates)\label{thm:B009}
There exists a unique heat kernel $p(t, x, y)$ with respect to the Dirichlet form $\mathscr E$.
\begin{enumerate}
\item $p(t,x,y)$ is stochastically complete. That is,
\begin{align*}
\int p(t,x,y) \, d\mu(x)=1.
\end{align*}
\item For any $L>0$, there exists a constant $C_5=C_5(n,L)>0$ such that 
\begin{align*}
\frac{C_5^{-1}}{|B(x,\sqrt t)|_{\mu}} \exp \lc -\frac{d^2(x,y)}{C_5^{-1}t} \rc \le p(t,x,y) \le \frac{C_5}{|B(x,\sqrt t)|_{\mu}} \exp \lc -\frac{d^2(x,y)}{C_5 t} \rc
\end{align*}
for any $x,y \in B(p,L)$ and $0 <t<L^2$. 
\end{enumerate}
\end{thm}

Next, we have the following Bakry-\'Emery condition.

\begin{prop}(Bakry-\'Emery condition)\label{prop:BE}
Let $(X,p,d,f) \in \mathfrak S(n)$. Then for any $u \in \Sigma(\Delta_f)$ with $\Delta_f u \in N^{1,2}(X,\mu)$ and $v \in L^{\infty} \cap \Sigma(\Delta_f)$ with $v \ge 0$ and $\Delta_f v \in L^{\infty}$,
\begin{align*}
\frac{1}{2} \int |\na u|^2 \Delta_f v \,d\mu \ge \int v\la \na u,\na \Delta_f u \ra \,d\mu.
\end{align*}
\end{prop}

\begin{proof}
If $u \in C_c^{\infty}(\mathcal R)$, then it follows from the Weitzenb\"{o}ck formula that
\begin{align*}
\frac{1}{2} \Delta_f|\na u|^2=|\text{Hess}\,u|^2+Rc_f(\na u, \na u)+\la \na u,\na \Delta_f u \ra \ge \la \na u,\na \Delta_f u \ra.
\end{align*}
Then for any $v \in L^{\infty} \cap \Sigma(\Delta_f)$ with $v \ge 0$ and $\Delta_f v \in L^{\infty}$, it is clear from the integration by parts that
\begin{align*}
\frac{1}{2} \int |\na u|^2 \Delta_f v \,d\mu \ge \int v\la \na u,\na \Delta_f u \ra \,d\mu.
\end{align*}
Now the general case follows by approximating $u$ by functions in $C_c^{\infty}(\mathcal R)$ as Proposition \ref{prop:B003}.
\end{proof}

Now we obtain the following result. For the definition of RCD$(K,\infty)$ space, see \cite[Definition $2.1$]{Gi18}.

\begin{thm}\label{thm:RCD}
Let $(X,p,d,f) \in \mathfrak S(n)$. Then $(X,d,\mu)$ is a \emph{RCD}$(0,\infty)$ space.
\end{thm}

\begin{proof}
(i), (ii) and (iv) in \cite[Definition $2.1$]{Gi18} follow from the definition of $\mathscr E$, Proposition \ref{prop:B001} and Proposition \ref{prop:BE}, respectively. (iii) in \cite[Definition $2.1$]{Gi18} follows from Proposition \ref{prop:B003} and the corresponding result on $\mathcal R$.
\end{proof}

For applications, we have the following gradient estimate for the positive harmonic function, see also \cite[Theorem $1.1$]{MW11}.

\begin{prop}[Cheng-Yau estimate]\label{prop:cheng}
Let $(X,p,d,f) \in \mathfrak S(n)$. Suppose $\Omega=B(x,4r)$ and $u \in L^{\infty}(\Omega)\cap N^{1,2}(\Omega,\mu)$ satisfies
\begin{align*}
\Delta_f u=0.
\end{align*} 
Then there exists a constant $C_4=C_4(n)>0$ such that if $x \in \mathcal R$,
\begin{align*}
|\na u|(x) \le \frac{C_4}{r} \| u\|_{L^{\infty}(\Omega)}.
\end{align*} 
\end{prop}

\begin{proof}
The proof follows verbatim from \cite[Proposition $2.24$]{CW17} by using the $\delta$-function property of the heat kernel and the weighted Sobolev inequality. The latter can be derived from \cite[Proposition $2.1$]{Ba97}, since $(X,d,\mu)$ is a RCD$(0,\infty)$ space.
\end{proof}

From Proposition \ref{prop:cheng}, we immediately obtain

\begin{cor}[Liouville theorem]\label{cor:Liou}
Suppose $u$ is a harmonic function on $\mathcal R$ with sublinear growth, then $u \equiv C$ on $\mathcal R$.
\end{cor}

Now we recall

\begin{defn} \label{def:eign}
The first eigenvalue fo $\Delta_f$ is defined by
\begin{align*}
\lambda:=\inf_{u \in N^{1,2}(X,\mu) } \frac{\mathscr E(u,u)}{\int u^2\, d\mu}=\inf_{u \in C_c^{\infty}(\mathcal R)} \frac{ \int |\na u|^2\,d\mu}{\int u^2\, d\mu}.
\end{align*}
\end{defn}
Notice that the last equality holds from Proposition \ref{prop:B003}. For any non-compact Ricci steady soliton conifold, we have the following theorem, which is a generalization of \cite[Proposition $2.1$]{MW11}.

\begin{thm}\label{thm:B010}
Let $(X,p,d,f) \in \mathfrak S(n)$ be non-compact. Then $\lambda=\frac{1}{4}$.
\end{thm}

\begin{proof}
First we show that $\lambda \le 1/4$. This can be done by considering the test function $u(x):=e^{-\frac{1}{4}(1+\ep)d(p,x)} \phi(x)$, where $\phi$ is a cutoff function with $\phi=1$ on $B(p,L-1)$ and $\phi=0$ outside $B(p,L)$. 
Since the volume $|B(p,L)| \le L^n e^L|B(p,1)|_{\mu}$ by \eqref{BE003}, one can compute directly that $\lambda \le \frac{1}{4}(1+\ep)^2$ for any $\ep>0$.

Next we compute on $\mathcal R$,
\begin{align*}
\Delta_f e^{\frac{f}{2}}=\lc -\frac{1}{2}+\frac{|\na f|^2}{4} \rc e^{\frac{f}{2}} \le -\frac{1}{4}e^{\frac{f}{2}}.
\end{align*}
Since $f$ is locally bounded, the above equation also holds on $X$ in the weak sense. Therefore, we have $\lambda \ge 1/4$ by the same argument of \cite[Lemma $2.2$]{MW11}.
\end{proof}

Theorem \ref{thm:B010} implies, in particular, that any non-compact $(X,p,d,f) \in \mathfrak S(n)$ admits a positive Green's function (see, e.g. \cite[Lemma $5.2$]{GH14}):
\begin{align*}
G(x,y)= \int_0^{\infty}p(t,x,y) \,dt.
\end{align*}

Next, we recall the following definition.

\begin{defn} \label{def:end}
Let $(X,p,d,f) \in \mathfrak S(n)$ be non-compact and $E$ an end of $X$. $E$ is said to be parabolic if it does not admit a positive
harmonic function $h$ satisfying
\begin{align*}
h=1
\end{align*}
on $\partial E$ and 
\begin{align*}
\liminf_{x \to E(\infty)} h(x)<1,
\end{align*}
where $E(\infty)$ denotes the infinity of $E$. Otherwise, $E$ is said to be nonparabolic.
\end{defn}

It is clear from the existence of a positive Green's function that $X$ has at least one nonparabolic end. 

Now we can follow the same proof as in \cite[Theorem $4.1$]{MW11} to prove the following result.

\begin{thm}
\label{thm:oneend}
Let $(X,p,d,f) \in \mathfrak S(n)$ be non-compact and nontrivial. Then $X$ has only one end.
\end{thm}

\begin{proof}
We assume $X$ has at least two ends and derive a contradiction. We first show $X$ has at most one nonparabolic end. Indeed, if $X$ has two nonparabolic end, one can follow the arguments of \cite[Theorem $21.1,21.3$]{Peter12} to construct a positive, bounded, nonconstant, harmonic function $u$ satisfying
\begin{align*}
\mathscr E(u,u)=\int |\na u|^2 \,d\mu <\infty.
\end{align*}
However, this contradicts Corollary \ref{cor:Liou}. Hence, $X$ has exactly one nonparabolic end $E$. Suppose $X$ has more than one end, then we can choose $F$ as a parabolic end. 
For simplicity, we may assume that $X$ has exactly two ends $E$ and $F$. 

\textbf{Claim 1:} $F$ contains a geodesic ray $\gamma \subset F \cap \mathcal{R}$. 

Actually, we fix a point $q \in \mathcal{R} \cap F$
and a sequence of points $y_i \in \mathcal{R} \cap F$. Let $\alpha_i \subset \mathcal R$ be a shortest geodesic connecting $q$ and $y_i$, whose existence is guaranteed by item 4 of Definition A.1. Then
$\alpha_i$ converges to a geodesic $\gamma \subset F$, naturally parametrized by arc length, such that $\gamma(0)=q$. Since $q$ is regular point, there exists a small constant $\epsilon$ such that
$\gamma([0, \epsilon]) \subset \mathcal{R}$. In particular, at each point $y \in \gamma([0, \epsilon])$, each tangent space at $y$ is $\R^n$. 
By volume comparison and heat kernel estimate(cf. \cite[Theorem 4.5]{HLW21}), for each sequence $r_{\alpha} \to 0^{+}$, the unit balls in tangent spaces associated to this sequence and points $y=\gamma(t)$ varies continuously in Gromov-Hausdorff topology for $t \in [0.5 \epsilon, \infty)$. 
In particular, the volume of unit balls in tangent spaces is then a continuous function of $t$. Thus, by the gap property(Item 6 of Definition A.1), it is a constant independent of $t$. This means that for each $t \in [0.5 \epsilon, \infty)$, a tangent space of $\gamma(t)$ is $\R^n$. Therefore, it follows from definition that $\gamma(t)$ is a regular point, i.e., $\gamma(t) \in \mathcal{R}$ for any $t \ge 0$. 
We define the Busemann function 
\begin{align*}
\beta(x) =\lim_{t \to \infty} (t-d(x, \gamma(t))). 
\end{align*}
It follows from the Laplacian comparison that
\begin{align*}
\Delta_{f} \beta \geq -1, \quad |\nabla \beta|=1, 
\end{align*}
which in turn implies that
\begin{align}
\Delta_{f} e^{\beta} \geq 0. \label{eqn:TG01_4}
\end{align}

\textbf{Claim 2:} Fix $p \in X$, for all $L>0$ large enough, we have
\begin{align}
|B(p,L) \cap E|_{\mu} \leq C e^{L}. \label{eqn:TG01_2}
\end{align}

On the end $E$, the function $d(p, \cdot)+\beta$ is uniformly bounded. Therefore, it suffices to show that for some large $s>0$, we have
\begin{align}
|\{s<-\beta<t\} \cap E|_{\mu} \leq C e^{t}, \quad \forall \; t>s. \label{eqn:TG01_1}
\end{align}
Actually, by perturbing distance function and the high co-dimension assumption of $\mathcal{S}$(cf. \cite[Corollary B.3]{CW17}), for each small $\delta>0$, we can find an almost ``tubular neighborhood" $T_{\delta}$ of $\mathcal{S}$.
The distance from any point in $\partial T_{\delta}$ to $\mathcal{S}$ is comparable to $\delta$,
and the $(n-1)$-dimensional Hausdorff measure of $\partial T_{\delta} \cap K$ is bounded by $C(K) \delta^2$ for dimensional reason, where $K$ is any compact set. Applying integration by parts away from $T_{\delta}$, we have 
\begin{align*}
0 \leq \int_{ \{ s<-\beta< t\} \backslash T_{\delta}} (\Delta_f e^{\beta}) e^{-f} =\int_{\{-\beta=s\} \cap (E \backslash T_{\delta})} e^{\beta} e^{-f} - \int_{\{-\beta=t\} \cap (E \backslash T_{\delta})} e^{\beta} e^{-f} 
+\int_{\partial T_{\delta} \cap E} \langle \nabla \beta, \vec{n} \rangle e^{\beta} e^{-f}, 
\end{align*}
where $\vec{n}$ is the outward unit normal vector of $\partial T_{\delta}$. 
It follows that 
\begin{align*}
|\{\beta=-t\} \cap (E \backslash T_{\delta})|_{\mu} \leq e^{t-s} |\{\beta=-s\} \cap E|_{\mu} + C(E,s,t) \delta^2 e^{t},
\end{align*}
where $\mu$ is the induced measure on the hypersurfaces. Integrating the above inequality, we obtain 
\begin{align*}
|\{s<-\beta<t\} \cap (E \backslash T_{\delta}) |_{\mu} \leq e^{t-s} |\{\beta=-s\} \cap E|_{\mu} + C(E,s,t) \delta^2 e^{t}. 
\end{align*}
Letting $\delta \to 0$, we arrive at (\ref{eqn:TG01_1}). 
By choosing proper cutoff functions, integration by parts(cf. \cite[Proposition 2.17]{CW17}, \cite[Theorem 1.4]{PLWang01}) then implies that
\begin{align}
|F \backslash B(p,L)|_{\mu} \leq C e^{-L}. \label{eqn:TG01_3}
\end{align}
Then we verbatim follow the proof in~\cite{MW11}. Define a cut-off function $\phi$ with support in $B(p, 2L)$ such that $\phi=1$ on $B(p,L)$ and $|\nabla \phi| \leq CL^{-1}$. Using the fact $\lambda_1(\Delta_{f})=\frac{1}{4}$, and the volume estimate (\ref{eqn:TG01_1}) and (\ref{eqn:TG01_3}), integration by parts implies that
\begin{align*}
\frac12 \int_{X} \left\{ \Delta_{f} e^{\beta} \right\} \phi^2 e^{-f}= \int_{X} e^{\frac12 \beta} \left( \Delta_{f} e^{\frac12 \beta} + \frac14 e^{\frac12 \beta} \right) \phi^2 e^{-f} \leq \int_{X} |\nabla \phi|^2 e^{\beta} e^{-f} \leq \frac{C}{L} \to 0. 
\end{align*}
It then follows from the combination of (\ref{eqn:TG01_4}) and the above inequality that on $\mathcal R$,
\begin{align*}
\Delta_{f} e^{\beta}=0, \quad \Delta_f \beta =-1 \quad \text{and} \quad |\na \beta|=1.
\end{align*}
It follows from Proposition \ref{prop:B009} that $(X,d)$ is isometric to $(Y \times \R,d' \times d_{g_E})$, where $(Y,d')$ is a Ricci steady soliton conifold. Since we assume $X$ has at least two ends, $Y$ must be compact. Therefore, it follows from Proposition \ref{prop:B007} that $Rc_Y$ and hence $Rc_X$ are identically $0$ on $\mathcal R$.

In sum, we obtain a contradiction, and the original conclusion holds.

\end{proof}

\section{Improvement theorems of symmetry} 
\label{app:B}

We recall the following definition, see \cite[Definition $4.2$]{BK20}.

\begin{defn}[Neck-symmetry]\label{neck_symmetry}
Let $(M^n, g(t))$ be an $n$-dimensional Ricci flow solution and let $(\bar x, \bar t)$ be a spacetime point with $R(\bar x, \bar t) = \frac{n-1}{2}r^{-2}$. Assume that $(\bar x, \bar t)$ is the center of an evolving $\bar \ep$-neck for some small positive number $\bar \ep$. We say $(\bar x, \bar t)$ is $\ep$-symmetric if there exists a smooth, time-independent family of vector fields $\mathcal U = \{U^{(a)} : 1 \leq a \leq {n \choose 2}\}$ defined on the closed ball $\bar B_{g(\bar t)}(\bar x, 100r)$ with the following properties: 
\begin{enumerate}
\item[$\bullet$] In $B_{g(\bar t)}(\bar x, 100 r) \times [\bar t - 100 r^2, \bar t]$, we have the estimate
\[
\sum_{l = 0}^2 \sum_{a = 1}^{{n \choose 2}} r^{2l} \big|D^l (\mathcal L_{U^{(a)}}(g(t)))|^2 \leq \ep^2.
\]
\item[$\bullet$] If $t \in [\bar t - 100 r^2, \bar t]$ and $\nu$ denotes the unit normal vector to Hamilton's CMC foliation of the $\bar \ep$-neck at time $t$, then in $B_{g(\bar t)}(\bar x, 100 r)$, we have the estimate
\[
\sum_{a = 1}^{{n \choose 2}} r^{-2} |\langle U^{(a)}, \nu \rangle|^2 \leq \ep^2. 
\]
\item[$\bullet$] If $t \in [\bar t -100 r^2, \bar t]$ and $\Sigma \subset B_{g(\bar t)}(\bar x, 100 r)$ is a leaf of Hamilton's CMC foliation of the $\bar \ep$-neck at time $t$, then 
\[
\sum_{a, b = 1}^{{n \choose 2}} \bigg| \delta_{ab} - \mathrm{area}_{g(t)}(\Sigma)^{-\frac{n+1}{n-1}} \int_\Sigma \langle U^{(a)}, U^{(b)} \rangle_{g(t)} \, d\mu_{g(t)} \bigg|^2 \leq \ep^2. 
\]
\end{enumerate}
\end{defn}

Now we state the Neck Improvement Theorem proved by Brendle \cite[Theorem $8.6$]{Bre20} in dimension $3$ and Brendle-Naff \cite[Theorem $4.8$]{BK20} in general dimension. 

\begin{thm}[Improvement Theorem A]
\label{Tneck}
There exists a large constant $\bar L$ (depending only upon $n$) and a small constant $\delta_1$ (depending only upon $\bar L$ and $n$) with the following property. Let $(M^n, g(t))$ be a Ricci flow solution, and let $(x_0, t_0)$ be a spacetime point that is the center of an evolving $\delta_1$-neck and satisfies $R(x_0, t_0) = \frac{n-1}{2}r^{-2}$. Moreover, suppose that every point in the parabolic neighborhood $B_{g(t_0)}(x_0, \bar L r) \times [t_0 - \bar L r^2, t_0)$ is $\ep$-symmetric, where $\ep \leq \delta_1$. Then $(x_0, t_0)$ is $\frac{\ep}{2}$-symmetric. 
\end{thm}

Next, we define the global symmetry of the Ricci flow. We have the following three types.

\begin{defn}[$\ep$-symmetry of type-A]
\label{def:sym1}
Let $(M^n, g(t))_{t\le \bar t}$ be an $n$-dimensional Ricci flow solution. $(M^n, g(t))$ is said to be $\ep$-symmetric of type-A at time $\bar t$ if
\begin{enumerate}
\item[$\bullet$] $M$ is diffeomorphic to $S^{n-1} \times \R$.

\item[$\bullet$] $(x,\bar t)$ is $\ep$-symmetric in the sense of Definition \ref{neck_symmetry} for any $x \in M$.
\end{enumerate}
\end{defn}

Now we make the following assumptions for ancient solutions to the Ricci flow.

\begin{assum}\label{assum1}
There exists a constant $\theta_0>0$ such that if spacetime point $(x,t)$ of the Ricci flow $(M^n, g(t))_{t \le 0}$ satisfies
\begin{align*}
\lambda_1(x,t) \le \theta_0 R(x,t),
\end{align*}
where $\lambda_1$ to be the minimal eigenvalue of $Rc$, then $(x,t)$ is $\delta_1$-symmetric in the sense of Definition \ref{neck_symmetry}, where $\delta_1$ is the constant in Theorem \ref{Tneck}.
\end{assum}

\begin{assum}\label{assum2}
There exists a constant $\Lambda_0>0$ such that for the Ricci flow solution $(M^n,g(t))_{t \le 0}$ with a marked point $z$, if $(\bar x,\bar t) \in M \times (-\infty,0]$ and $d_{g(\bar t)}(z,\bar x) \ge \Lambda_0 R(z,\bar t)^{-\frac 1 2}$, then
\begin{align*}
\lambda_1(x,t) \le \frac{\theta_0}{2} R(x,t),
\end{align*}
for any $(x,t) \in B_{g(\bar t)}(\bar x, \bar L R(\bar x,\bar t)^{-\frac 1 2}) \times [\bar t-\bar L R(\bar x,\bar t)^{-1},\bar t]$, where $\bar L$ is the constant in Theorem \ref{Tneck}.
\end{assum}

The definition of the symmetry of type-B follows from \cite[Definition $5.2$]{BK20}. In the following, we set $r_{\max}(t)=(\max_M R(t))^{-\frac 1 2}$.

\begin{defn}[$\ep$-symmetry of type-B]
\label{def:sym2}
The Ricci flow $(M^n, g(t))_{t \le \bar t}$ with Assumption \ref{assum1} and Assumption \ref{assum2} is called $\ep$-symmetric of type-B at time $\bar t$ if $M$ is diffeomorphic to $\R^n$ and there exists a compact domain $D \subset M$ and a family of time-independent vector fields $\mathcal U = \{U^{(a)} : 1 \leq a \leq {n \choose 2}\}$ which are defined on an open subset containing $D$ such that the following statements hold: 
\begin{enumerate}
\item[$\bullet$] There exists a point $x \in \partial D$ such that $\lambda_1(x, \bar t) < \theta_0 R( x, \bar t)$. 
\item[$\bullet$] For each $x \in D$, we have $\lambda_1(x, \bar t) > \frac{1}{2} \theta_0 R(x, \bar t)$. 
\item[$\bullet$] The boundary $\partial D$ is a leaf of Hamilton's CMC foliation at time $\bar t$. 
\item[$\bullet$] For each $x \in M \setminus D$, the point $(x, \bar t)$ is $\ep$-symmetric in the sense of Definition \ref{neck_symmetry}.
\item[$\bullet$] In $D \times [\bar t - r_{\max}(\bar t)^{2}, \bar t]$, we have the estimate
\[
\sum_{l = 0}^2 \sum_{a = 1}^{{n \choose 2}} r_{\max}(\bar t)^{2l} \big|D^l (\mathcal L_{U^{(a)}}(g(t)))|^2 \leq \ep^2.
\]
\item[$\bullet$] If $\Sigma \subset D$ is a leaf of Hamilton's CMC foliation of $(M, g(\bar t))$ that has distance at most $50\,r_{\mathrm{neck}}(\partial D)$ from $\partial D$, then 
\[
\sup_{\Sigma} \sum_{a = 1}^{{n \choose 2}} r_{\max}(\bar t)^{-2}|\langle U^{(a)}, \nu \rangle|^2 \leq \ep^2, 
\]
where $\nu$ is the unit normal vector to $\Sigma$ in $(M, g(\bar t))$ and $r_{\mathrm{neck}}(\partial D)$ is defined by the identity $\mathrm{area}_{g(\bar t)}(\partial D) = \mathrm{area}_{g_{S^{n-1}}}(S^{n-1}) r_{\mathrm{neck}}(\partial D)^{n-1}$. 
\item[$\bullet$] If $\Sigma \subset D$ is a leaf of Hamilton's CMC foliation of $(M, g(\bar t))$ that has distance at most $50 \,r_{\mathrm{neck}}(\partial D)$ from $\partial D$, then 
\[
\sum_{a, b = 1}^{{n \choose 2}} \bigg| \delta_{ab} - \mathrm{area}_{g(\bar t)}(\Sigma)^{-\frac{n+1}{n-1}} \int_\Sigma \langle U^{(a)}, U^{(b)} \rangle_{g(\bar t)} \, d\mu_{g(\bar t)} \bigg|^2 \leq \ep^2, 
\]
where $r_{\mathrm{neck}}(\partial D)$ is defined by the identity $\mathrm{area}_{g(\bar t)}(\partial D) = \mathrm{area}_{g_{S^{n-1}}}(S^{n-1}) r_{\mathrm{neck}}(\partial D)^{n-1}$. 
\end{enumerate}
\end{defn}

Now we can state the improvement theorem for the second type. The proof follows verbatim from \cite[Section $9$]{Bre20} and \cite[Section $5$]{BK20}.

\begin{thm}[Improvement Theorem B]
\label{Thm:improve2}
There exist positive constants $\delta_2$ and $\delta_3$ depending only $\theta_0$, $\Lambda_0$ and $n$ with the following property. Let $(M^n, g(t))_{t \le \bar t}$ be a Ricci flow solution satisfying Assumption \ref{assum1} and Assumption \ref{assum2}. Suppose $(M^n, g(t))$ is $\ep$-symmetric of type-B at time $t$ for any $t \le \bar t$ with $\ep \le \delta_2$ and the marked point $z$ in the Assumption \ref{assum2} belongs to $D$. Moreover, after rescaling the metric by $r^{-2}$, the parabolic neighborhood $B_{g(\bar t)}(z,\delta_3^{-1}r) \times [\bar t-\delta_3^{-1}r,\bar t]$ is $\delta_3$-close in the $C^{[\delta_3^{-1}]}$-topology to the Bryant soliton based at the tip whose scalar curvature is $1$, where $R(z,\bar t)=r^{-2}$. Then $(M^n, g(t))$ is $\frac{\ep}{2}$-symmetric of type-B at $\bar t$. 
\end{thm}

\begin{assum}\label{assum3}
There exists a constant $\Lambda_1>0$ such that for the Ricci flow solution $(M^n,g(t))_{t \le 0}$ with two marked points $z_1,z_2$, if $(\bar x,\bar t) \in M \times (-\infty,0]$, $d_{g(\bar t)}(z_1,\bar x) \ge \Lambda_1 R(z_1,\bar t)^{-\frac 1 2}$ and $d_{g(\bar t)}(z_2,\bar x) \ge \Lambda_1 R(z_2,\bar t)^{-\frac 1 2}$, then
\begin{align*}
\lambda_1(x,t) \le \frac{\theta_0}{2} R(x,t),
\end{align*}
for any $(x,t) \in B_{g(\bar t)}(\bar x, \bar L R(\bar x,\bar t)^{-\frac 1 2}) \times [\bar t-\bar L R(\bar x,\bar t)^{-1},\bar t]$, where $\bar L$ is the constant in Theorem \ref{Tneck}.
\end{assum}

Next, we have the following definition, see \cite[Definition $3.1$]{BDNS21}.

\begin{defn}[$\ep$-symmetry of type-C]
\label{def:sym3}
The Ricci flow $(M^n, g(t))_{t \le \bar t}$ with Assumption \ref{assum1} and Assumption \ref{assum3} is called $\ep$-symmetric of type-C at time $\bar t$ if $M$ is diffeomorphic to $S^n$ and there exists a compact domain $D \subset M$ and a family of time-independent vector fields $\mathcal U = \{U^{(a)} : 1 \leq a \leq {n \choose 2}\}$ which are defined on an open subset containing $D$ such that the following statements hold: 
\begin{itemize}
\item The domain $D$ is a disjoint union of two domains $D_1$ and $D_2$, each of which is diffeomorphic to $B^n$.
\item $\lambda_1(x, \bar t) < \theta_0 R(\bar x, \bar t)$ for all points $x \in M \setminus D$.
\item $\lambda_1(x, \bar t) > \frac{1}{2} \theta_0 R(x, \bar t)$ for all points $x \in D$.
\item $\partial D_1$ and $\partial D_2$ are leaves of Hamilton's CMC foliation of $(M, g(\bar t))$. 
\item For each $x \in M \setminus D$, the point $(x, \bar t)$ is $\ep$-symmetric in the sense of Definition \ref{neck_symmetry}.
\item The Lie derivative $\mathcal L_{U^{(a)}}(g(t))$ satisfies for each $k \in \{1,2\}$ the estimate 
\[
\sup_{D_k \times [\bar t - \rho_k^2, \bar t]} \sum_{l = 0}^2 \sum_{a = 1}^{{n \choose 2}} \rho_k^{2l} \big|D^l (\mathcal L_{U^{(a)}}(g(t)))|^2 \leq \ep^2,
\]
where $\rho_k^{-2} := \sup_{x \in D_k} R(x, \bar t)$. 

\item For each $k \in \{1,2\}$, the following property holds. 
If $\Sigma \subset D_k$ is a leaf of Hamilton's CMC foliation of $(M, g(\bar t))$ that has distance at most $50\,r_{\mathrm{neck}}(\partial D_k)$ from $\partial D_k$, then 
\[
\sup_{\Sigma} \sum_{a = 1}^{{n \choose 2}} \rho_k^{-2}|\langle U^{(a)}, \nu \rangle|^2 \leq \ep^2, 
\]
where $\nu$ is the unit normal vector to $\Sigma$ in $(M, g(\bar t))$ and $r_{\mathrm{neck}}(\partial D_k)$ is defined by the identity $\mathrm{area}_{g(\bar t)}(\partial D_k) = \mathrm{area}_{g_{S^{n-1}}}(S^{n-1}) r_{\mathrm{neck}}(\partial D_k)^{n-1}$. 
\item[$\bullet$] For each $k \in \{1,2\}$, the following property holds. 
If $\Sigma \subset D_k$ is a leaf of Hamilton's CMC foliation of $(M, g(\bar t))$ that has distance at most $50 \,r_{\mathrm{neck}}(\partial D_k)$ from $\partial D_k$, then 
\[
\sum_{a, b = 1}^{{n \choose 2}} \bigg| \delta_{ab} - \mathrm{area}_{g(\bar t)}(\Sigma)^{-\frac{n+1}{n-1}} \int_\Sigma \langle U^{(a)}, U^{(b)} \rangle_{g(\bar t)} \, d\mu_{g(\bar t)} \bigg|^2 \leq \ep^2. 
\]
\end{itemize}
\end{defn}

Following the same arguments of \cite[Section $3$]{BDS20} and \cite[Section $3$]{BDNS21}, we have the following theorem.

\begin{thm}[Improvement Theorem C]\label{Thm:improve3}
There exist positive constants $\delta_4$ and $\delta_5$ depending only on $\theta_0$, $\Lambda_1$ and $n$ with the following property. Let $(M^n, g(t))_{t \le \bar t}$ be a Ricci flow solution satisfying Assumption \ref{assum1} and Assumption \ref{assum3}. Suppose $(M^n, g(t))$ is $\ep$-symmetric of type-C at time $t$ for any $t \le \bar t$ with $\ep \le \delta_4$ and the marked points $z_1$ and $z_2$ in the Assumption \ref{assum3} belong to $D_1$ and $D_2$ respectively. Moreover, after rescaling the metric by $r_i^{-2}$, the parabolic neighborhood $B_{g(\bar t)}(z_i,\delta_5^{-1}r_i) \times [\bar t-\delta_5^{-1}r_i,\bar t]$ is $\delta_5$-close in the $C^{[\delta_5^{-1}]}$-topology to the Bryant soliton based at the tip whose scalar curvature is $1$ for $i=1,2$, where $R(z_i,\bar t)=r_i^{-2}$. Then $(M^n, g(t))$ is $\frac{\ep}{2}$-symmetric of type-C at $\bar t$. 
\end{thm}

\vskip10pt

Yu Li, Institute of Geometry and Physics, University of Science and Technology of China, No. 96 Jinzhai Road, Hefei, Anhui Province, 230026, China; yuli21@ustc.edu.cn.\\

Bing Wang, Institute of Geometry and Physics, School of Mathematical Sciences, University of Science and Technology of China, No. 96 Jinzhai Road, Hefei, Anhui Province, 230026, China; topspin@ustc.edu.cn.\\

\end{document}